\documentclass[12pt,leqno]{amsart}
\usepackage{amsmath,amsthm,amssymb,amscd,verbatim}
\newtheoremstyle{fancy}{}{}{\itshape}{}{\textsc\bgroup}{.\egroup}{ }{}
\newtheoremstyle{fancy2}{}{}{\rm}{}{\textsc\bgroup}{.\egroup}{ }{}
\newcounter{counteroman}

\theoremstyle{fancy}
\newcounter{intro}

\numberwithin{equation}{section}
\newtheorem{cor}[equation]{Corollary}
\newtheorem{lem}[equation]{Lemma}
\newtheorem{prop}[equation]{Proposition}
\newtheorem{thm}[equation]{Theorem}
\newtheorem{named}[equation]{\name}
\newcommand{\name}{Proof of}

\theoremstyle{fancy2}
\newtheorem{dfn}[equation]{Definition}

\newtheorem{exa}[equation]{Example}
\newtheorem{exas}[equation]{Examples}
\newtheorem{rem}[equation]{Remark}
\newtheorem{rems}[equation]{Remarks}
\newcommand{\cref}[1]{Corollary~\ref{#1}}
\newcommand{\lref}[1]{Lemma~\ref{#1}}
\newcommand{\pref}[1]{Proposition~\ref{#1}}
\newcommand{\tref}[1]{Theorem~\ref{#1}}

\def\C{\mathbb C}       
\def\R{\mathbb R}       \def\Q{\mathbb Q}
\def\Z{\mathbb{Z}}  
\def\la{\langle} \def\ra{\rangle}
\newcommand{\ad}{\mathrm{ad}}
\newcommand{\Ad}{\mathrm{Ad}}

\newcommand{\cl}{\operatorname{\C{\rm l}}}

\newcommand{\coker}{\operatorname{coker}}

\newcommand{\diam}{\operatorname{diam}}

\renewcommand{\div}{\operatorname{div}}
\newcommand{\dom}{\operatorname{dom}}

\newcommand{\ext}{\operatorname{ext}}
\newcommand{\Gl}{\operatorname{Gl}}

\newcommand{\id}{\operatorname{id}}
\newcommand{\ind}{\operatorname{ind}}

\newcommand{\grad}{\operatorname{grad}}

\newcommand{\im}{\operatorname{im}}
\renewcommand{\Im}{\operatorname{Im}}
\newcommand{\Lip}{\operatorname{Lip}}
\newcommand{\loc}{\operatorname{loc}}

\renewcommand{\O}{\operatorname{O}}

\renewcommand{\Re}{\operatorname{Re}}
\newcommand{\Ric}{\operatorname{Ric}}
\newcommand{\Se}{\mathrm{S}}
\newcommand{\sign}{\operatorname{sign}}
\newcommand{\SO}{\operatorname{SO}}
\newcommand{\spec}{\operatorname{spec}}

\newcommand{\Spin}{\operatorname{Spin}}
\newcommand{\SU}{\operatorname{SU}}

\newcommand{\tr}{\operatorname{tr}}
\newcommand{\U}{\operatorname{U}}

\newcommand{\vol}{\operatorname{vol}}

\begin{document}


\title[Index Theorems on Manifolds with Straight Ends.]
{Index Theorems on Manifolds with Straight Ends.}
\author{Werner Ballmann, Jochen Br\"uning, \and Gilles Carron}
\address{
Mathematisches Institut\\
Universi\-t\"at Bonn\\
Endenicher Allee 60, 53115 Bonn
and Max-Planck-Institut f\"ur Mathematik,
Vivatsgasse 7, 53111 Bonn, Deutschland\\}
\email{hwbllmnn\@@math.uni-bonn.de}
\address{
Institut f\"ur Mathematik\\
Humboldt--Universit\"at\\
Rudower Chaussee 5, 12489 Berlin, Germany\\}
\email{bruening\@@mathematik.hu-berlin.de}
\address{
D\'epartement de Math\'ematiques\\
Universit\'e de Nantes\\
2 rue de la Houssini\'ere, BP 92208\\
44322 Nantes Cedex 03, France\\}
\email{Gilles.Carron\@@math.univ-nantes.fr}
\date{\today}


\begin{abstract}
We study Fredholm properties and index formulas for Dirac operators
over complete Riemannian manifolds with straight ends.
An important class of examples of such manifolds are complete
Riemannian manifolds with pinched negative sectional curvature and finite volume.
\end{abstract}

\maketitle

\newpage
\tableofcontents

\newpage
\section{Introduction}
\label{secint}

The celebrated Atiyah-Singer index theorem establishes a connection
between analysis, geometry, and topology of closed manifolds.
It contains the Chern-Gauss-Bonnet formula, Hirzebruch's signature theorem,
and the Hirze\-bruch-Riemann-Roch formula as special cases.
Later, Atiyah, Patodi, and Singer found a generalization of the index theorem
for certain first order differential operators
on compact manifolds with boundary \cite{APS}.
In this article,
they also discuss index theory for their class of operators
on non-compact manifolds with cylindrical ends,
and our work builds on that part of their work.

It is obvious that the structure of the underlying manifold
and of the differential operator close to infinity
plays an important role in this theory.
Without restrictions on these data, not much can be expected.

Motivated by previous work of Barbasch-Moscovici \cite{BM},
Lott \cite{Lo1,Lo2}, and the first two authors \cite{BB1,BB2},
our main objective are Dirac operators on complete manifolds
with pinched negative sectional curvature and finite volume.
The structure of the ends of such manifolds has been determined
by Patrick Eberlein
and is related to the existence of so-called strictly invariant horospheres,
see \cite{Eb}.

To set the stage, let $M$ be a complete and connected
Riemannian manifold of dimension $m$
with Levi-Civita connection $\nabla$ and curvature tensor $R$.
Let $E\to M$ be a complex Dirac bundle\footnote{in the sense
of Gromov and Lawson, compare Section \ref{susdb}}
with Hermitian connection $\nabla^E$, curvature tensor $R^E$,
and Dirac operator $D$.
For convenience,
we assume throughout that $R$ and $R^E$ are uniformly bounded,
\begin{equation}
  | R(X,Y)Z | \le C_R | X | | Y | | Z | , \quad
  | R^E(X,Y)\sigma | \le C_R^E | X | | Y | | \sigma | ,
  \label{boundrre}
\end{equation}
for all vector fields $X$, $Y$, $Z$ on $M$ and sections $\sigma$ of $E$.
The bound on $R$ is equivalent to assuming a uniform bound
on the modulus of the sectional curvature $K_M$ of $M$.

Recall that $D$ is an elliptic differential operator of first order.
Consider $D$ as an unbounded operator on $L^2(M,E)$
with domain $C^\infty_c(M,E)$,
where $L^2(M,E)$ denotes the space of square-integrable sections of $E$
and $C^\infty_c(M,E)$ the space  of smooth sections of $E$
with compact support,
and note that $D$ is symmetric on the latter.
The closure of $D$ has domain $H^1(M,E)$,
by \eqref{boundrre} and the general Bochner identity,
see \eqref{lic1} and \eqref{lic2}.
Furthermore, $D:H^1(M,E)\to L^2(M,E)$ is self-adjoint,
see \cite{Wo} or Theorem II.5.7 in \cite{LM}.

We may ask under which conditions $D:H^1(M,E)\to L^2(M,E)$
is a Fredholm operator.
By self-adjointness, this is the case if and only if $0$
is not in the essential spectrum of $D$;
according to a result of Nicolae Anghel,
this holds if and only if there is a compact subset $L$
in $M$ and a constant $C=C(L)$ such that
\begin{equation}
  \| \sigma \|_{L^2(M,E)} \le C \| D\sigma \|_{L^2(M,E)} ,
  \label{infred}
\end{equation}
for all smooth sections $\sigma$ of $E$
with compact support in $M\setminus L$, see \cite{An}.
If such an estimate holds, we say that $D$ is of {\em Fredholm type}.

Better adapted to our investigations and more flexible
is a somewhat weaker notion,
introduced by the third named author in \cite{Ca1}:

\begin{dfn}\label{nonpai0}
We say that $D$ is {\em non-parabolic at infinity}
if there is a compact subset $L$ in $M$ such that,
for any relatively compact open subset $K$ of $M$,
there is a constant $C=C(K,L)$ such that
\begin{equation}
  \| \sigma \|_{L^2(K,E)} \le C \| D\sigma \|_{L^2(M,E)} ,
  \label{nonpai}
\end{equation}
for all smooth sections $\sigma$ of $E$ with compact support in $M\setminus L$.
\end{dfn}

It follows from \cite[Th\'eor\`eme 1.2]{Ca1} that $D$ is non-parabolic at infinity
if and only if there is a Hilbert space $W$ of sections of $E$
which are locally $H^1$, such that $H^1(M,E)$ is a dense subspace of $W$,
such that the inclusions
\begin{equation}
  H^1(M,E) \subseteq W \subseteq H^1_{\loc}(M,E)
  \label{introw}
\end{equation}
are continuous,
and such that $D$ extends to a Fredholm operator
\begin{equation}
  D_{\ext}: W \to L^2(M,E)\,.
  \label{indext}
\end{equation}
It then follows that $\im D_{\ext}$ is equal
to the closure of $\im D$ in $L^2(M,E)$
and that $H^1(M,E) = W$ if and only if $D$
is of Fredholm type\footnote{Compare Section \ref{subdec}.}.

If $D$ is non-parabolic at infinity, with associated Hilbert space $W$,
then elements of $\ker D_{\ext}$ will be called {\em extended solutions} of $D$.
In the case of cylindrical ends,
they correpond exactly to the extended solutions in \cite{APS}.
By the density of $C^\infty_c(M,E)$ in $W$,
the orthogonal complement of the image of $D_{\ext}$ in $L^2(M,E)$
is equal to the space of $L^2$-solutions of $D$.
Since $D_{\ext}$ is a Fredholm operator,
the spaces of extended solutions and $L^2$-solutions of $D$ are of finite dimension,
and their difference, $\ind D_{\ext}$, is called the {\em extended index} of $D$.
As a consequence of one of our main results concerning non-parabolicity,
\tref{mainth2} below,  we obtain the following assertion:

\begin{thm}\label{mainth}
If the sectional curvature of $M$ is negatively pinched
and the volume of $M$ is finite, then $D$ is non-parabolic at infinity.
In particular, the space of $L^2$-solutions of $D$ is finite-dimensional.
\end{thm}

Under a more general assumption on the geometry of the ends of $M$,
similar to Condition (1) in \tref{mainth1} below,
John Lott showed that the space of square-integrable harmonic
differential forms is finite dimensional, see Theorem 1 in \cite{Lo1}.

For manifolds with ends as in the case of finite volume manifolds
of pinched negative sectional curvature,
Lott also discusses the essential spectrum of $(d+d^*)^2$
on the space of differential forms, see Theorem 2 in \cite{Lo1}.
Under the same assumption on the geometry of the ends
and for Dirac bundles as in Condition (2) of \tref{mainth1} below,
he investigates the essential spectrum of the associated Dirac operator,
see Theorem 5 in \cite{Lo2}.
Similar results have been obtained in \cite{BB2}.
In this article, we do not deal with the essential spectrum,
but would like to mention that our investigations lead to extensions
of these results.

It is clear from the definition of non-parabolicity
that it only depends on the structure of $D$ at infinity\footnote{
The same applies to the essential spectrum of $D$.}.
To state our results in that context,
we need to introduce a further notion.

\begin{dfn}\label{straight}
We say that the ends of $M$ are {\em straight} if $M$ can be decomposed
into a compact part $M_0$ and an unbounded part $U_0$
with common boundary $N$
such that there is an open set $U\supseteq U_0$
and a $C^2$ distance function\footnote{Compare Section \ref{secdifu}.}
$f:U\to\R$ whose gradient flow establishes a $C^1$ diffeomorphism
\begin{equation}
  F: (-r,\infty)\times N \to U ,
\end{equation}
where $r>0$, $U_0=f^{-1}([0,\infty))$, $N = f^{-1}(0)$, and $f(F(t,x))=t$.
If $f$ is smooth, then we say that the ends of $M$ are {\em smooth}.
\end{dfn}

If the ends of $M$ are straight,
then $M$ is diffeomorphic to the interior of the compact manifold $M_0$,
and the connected components of $N$ correspond to the different ends of $M$.
Furthermore, the induced Riemannian metric on $\R_+\times N$ is of the form
\begin{equation}
  dt^2 + g_t ,
  \label{straight2}
\end{equation}
where $(g_t)_{t\ge0}$ is a family of Riemannian metrics on $N$.
The regularity of this family is a technical problem
which we address in Section \ref{secdifu}
and which motivated our previous work \cite{disy}
on Dirac systems with Lipschitz coefficients.
{\em Cylindrical ends} as mentioned above
correspond to the case of Riemannian products,
that is, $f$ is smooth and $g_t=g_0$, for all $t\in(-r,\infty)$.

If the ends of $M$ are straight,
we fix the setup as in Definition \ref{straight},
identify $(-r,\infty)\times N\simeq U$ via $F$,
and call the hypersurfaces $N_t=f^{-1}(t)$,
endowed with the Riemannian metric $g_t$, the {\em cross sections} of $U$.
For convenience, we will always assume in this situation
that the Weingarten operators $W=W_t$ of the cross sections are uniformly bounded,
\begin{equation}
  | WX | \le C_W | X | ,
  \label{boundw}
\end{equation}
for all vector fields $X$ on $U$.

\begin{dfn}\label{defthin}
Let $\varepsilon>0$.
We say that the ends of $M$ are $\varepsilon$-{\em thin}
if they are straight and the connected components of the cross sections $N_t$
have diameter at most $\varepsilon$, for all sufficiently large $t$.
We say that the ends of $M$ are {\em cuspidal}
if they are straight and there are positive constants $c$ and $C$
such that the metrics $g_t$ as in \eqref{straight2}
satisfy $g_t\le Ce^{c(s-t)}g_s$, for all sufficiently large $s<t$.
\end{dfn}

For example,
if $M$ has finite volume and pinched negative sectional curvature,
say $-b^2\le K_M\le-a^2<0$,
then the ends of $M$ are cuspidal with $c=2a$ and $C=1$.
We note that, in this example, the distance function arises from
Busemann functions on the universal covering space of $M$
and that such Busemann functions are $C^2$,
see \cite{HI} or Proposition IV.3.2 in \cite{Ba}.
Better regularity is, in general, not expected
and, at least for non-positively curved manifolds,
better regularity does not hold, see \cite{BBB}.

\begin{thm}\label{mainth1}
There is a positive constant $\varepsilon=\varepsilon(m,C_R,C_W)$
such that $D$ is non-parabolic at infinity
if the following two conditions hold:
\begin{enumerate}
\item
All ends of $M$ are $\varepsilon$-thin,
for all sufficiently large $t$.
\item
$E$ is a Hermitian vector bundle associated to $M$
via a unitary representation of $\O(m)$, $\SO(m)$ (if $M$ is oriented),
or $\Spin(m)$ (with respect to a spin structure of $M$), respectively.
\end{enumerate}
\end{thm}

Extending \tref{mainth} above, we also have:

\begin{thm}\label{mainth2}
If the ends of $M$ are cuspidal,
then $D$ is non-parabolic at infinity.
\end{thm}

Suppose now that the dimension $m$ of $M$ is even
and that $E=E^+\oplus E^-$ is a super-symmetry\footnote{
See Section \ref{susdb}.}.
Then $D$ maps sections of $E^+$ to sections of $E^-$ and conversely,
and thus we obtain operators
\begin{equation}
  D^\pm: H^1(M,E^\pm) \to L^2(M,E^\mp) .
\end{equation}
In the case of closed manifolds,
these are the operators the Atiyah-Singer index theorem is concerned with.
The local index theorem associates an index form $\omega_{D^+}$
to $D^{+}$, determined by local data of $D^+$,
whose integral is equal to the index of $D^+$.
If $D$ is non-parabolic at infinity,
we obtain corresponding Fredholm operators
\begin{equation}
  D_{\ext}^\pm : W^\pm \to L^2(M,E^\mp) ,
\end{equation}
where $W^\pm$ consists of those sections in $W$ which take values in $E^\pm$.
One of our general results on the extended index of $D^+$ is

\begin{thm}\label{pindfor}
If $M$ has at most finitely many ends,
$D$ is non-parabolic at infinity, and $\omega_{D^+}$ is integrable, then
\begin{equation*}
  \ind D_{\ext}^+ = \int_M \omega_{D^+}
  + \sum_{\mathcal C} {\rm Corr} (\mathcal C) ,
\end{equation*}
where $\omega_{D^+}$ is the index form associated to $D^+$,
$\mathcal C$ runs over the ends of $M$,
and $ {\rm Corr} (\mathcal C)$ is a correction term determined
by the end $\mathcal C$.
\end{thm}

\tref{pindfor} is a kind of relative index theorem and,
assuming the non-parabolicity of $D$,
can also be proved along the lines of relative index formulas
as in Theorem 4.18 in \cite{GL} (see also Proposition 4.33),
Theorem 6.2 in \cite{Do}, or Theorem 0.5 in \cite{Ca1}.

Clearly, the assumptions of \tref{pindfor} are satisfied
if the ends of $M$ are cuspidal.
We assume the latter in the following discussion.

In dimension $m=2$,
the correction terms are known explicitly in terms of the type of $E$
along the ends, see \cite{BB1}.
In higher dimensions and under strong pinching assumptions
on the sectional curvature of $M$,
they are known explicitly for the Gauss-Bonnet operator, see \cite{BB2}.

In the case where the ends of $M$ are homogeneous cusps and,
over them, $E$ is a twisted associated bundle as in Chapter \ref{sechc},
we obtain explicit formulas for the extended index of $D^+$.
Over each cusp $\mathcal C$,
$D^+$ can then be decomposed into two orthogonal parts,
a low energy part $D^{\rm{le},+}_{\mathcal C}$
and a high energy\footnote{
Our usage of the notion {\em high energy} follows
the  terminology introduced in \cite{Lo1}.}
part $D^{\rm{he},+}_{\mathcal C}$.
Moreover, these are naturally equivalent to operators of the form
\begin{equation}
  \frac{d}{dt} + A^{\rm{le},+}_{\mathcal C}
  \quad\text{and}\quad
  \frac{d}{dt} + A^{\rm{he},+}_{\mathcal C,t} ,
  \label{inttrafo}
\end{equation}
respectively,
where $A^{\rm{le},+}_{\mathcal C}$ is a self-adjoint operator
on some finite dimensional Hilbert space 
and where $A^{\rm{he},+}_{\mathcal C,t}$, $t\ge0$,
is a family of invertible self-adjoint operators on some (other) Hilbert space,
with eigenvalues of finite multiplicity
and a spectral gap that tends to infinity as $t\to\infty$.

\begin{thm}\label{intexinfin}
If the ends $\mathcal C$ of $M$ are homogeneous cusps and,
over them, $E$ is a twisted associated bundle as in Chapter \ref{sechc}, then
\begin{equation*}
  \ind D_{\ext}^+ = \int_M \omega_{D^+}
  + \frac12\sum_{\mathcal C} \left(
  \lim_{t\to\infty} \eta(A^{\rm{he},+}_{\mathcal C,t})
  + \eta(A^{\rm{le},+}_{\mathcal C})
  + \dim\ker  A^{\rm le,+}_{\mathcal C} \right) .
\end{equation*}
Moreover, $D^+$ is of Fredholm type if and only if
$\dim\ker  A^{\rm le,+}_{\mathcal C}=0$, for all ends $\mathcal C$ of $M$.
\end{thm}

Denote by $h^\pm_\infty$ the difference in dimension of the spaces of extended
and $L^2$-solutions of $D^\pm$.
These quantities  determine the difference
between the extended and $L^2$-indices of $D^\pm$,
\begin{equation}
  \ind D^\pm_{\ext} = \ind_{L^2} D^\pm + h^\pm_{\infty} ,
  \label{hinfty}
\end{equation}
where $\ind_{L^2} D^\pm := \dim \ker D^\pm - \dim \ker D^\mp$.

\begin{thm}\label{intmaxinf}
If the ends $\mathcal C$ of $M$ are homogeneous cusps and,
over them, $E$ is a twisted associated bundle as in Chapter \ref{sechc}, then
\begin{equation*}
  \ind_{L^2} D^+ = \int_M \omega_{D^+}
  + \frac12\sum_{\mathcal C} \left(
  \lim_{t\to\infty} \eta(A^{\rm{he},+}_{\mathcal C,t})
  + \eta(A^{\rm{le},+}_{\mathcal C}) \right)
  - \frac12 \big( h^+_{\infty} - h^-_{\infty} \big) .
\end{equation*}
\end{thm}

The term $h^+_{\infty} - h^-_{\infty}$ is non-local,
and therefore the $L^2$-index does not admit an analog of \tref{pindfor},
where the correction terms are determined by the different ends separately.

The most important class of examples to which Theorems \ref{intexinfin}
and \ref{intmaxinf}
apply are finite volume quotients of symmetric spaces
of negative sectional curvature,
that is, of real, complex, or quaternionic hyperbolic spaces
or of the Cayley hyperbolic plane.
The work of Barbasch-Moscovici \cite{BM} is a milestone
in the index theory of Dirac operators of homogeneous Dirac bundles
over such spaces.
Their arguments rely on harmonic analysis on symmetric spaces,
notably the Selberg trace formula.

Werner M\"uller has obtained an index formula for manifolds with cusps of $\Q$-rank one.
His approach follows the one of Atiyah-Patodi-Singer:
using a good approximation of the heat operator $e^{-tD^2}$ and harmonic analysis,
he computed the correction terms  (see theorems 10.32  and 11.77 in \cite{Mu2}).
In particular,
he determined the non-local term $h^+_\infty - h^-_\infty$
using the scattering matrix at zero energy,
see also \cite{Mu1}.

M\"uller's approach has been extended by Mark Stern \cite{St1}.
The derivations of index formulas by Boris Vaillant in \cite{Va}
and Leichtnam, Mazzeo, and Piazza in \cite{LMP}
are based on the pseudodifferential calculus of Mazzeo-Melrose \cite{MM}.
However, the fibered boundary geometry discussed in \cite{Va} and \cite{LMP}
does not cover all finite volume locally symmetric space with $\Q$-rank one.
In their recent paper \cite{GH},
Daniel Grieser and Eugenie Hunsicker settle the main properties
of the pseudodifferential calculus for multi-fibred boundary structures.

Our proof is different in nature and is based on our previous work:
applying our results from \cite{disy},
we are able to discuss the contribution of each end individually.
This leads to a more general setting and more transparent index formulas.
Note, in particular, that our results also apply in the case
where $D$ is not of Fredholm type.

In this article, we concentrate on complex hyperbolic cusps,
more precisely,
cusps as they arise for quotients of complex hyperbolic space $\C H^n$
of dimension $m=2n$.
Here we discuss only two specific results for such quotients
and refer the reader to our r\'esum\'e in Section \ref{exaeta}
for our general index formulas pertaining to complex hyperbolic cusps.
The first result concerns the $L^2$-arithmetic genus and follows
from the Hirzebruch proportionality principle
and Theorems \ref{intexinfin} and \ref{corrcom};
see also Example 1 in Section \ref{exaeta}.

\begin{thm}\label{thmdolb}
If $X$ is a quotient of complex hyperbolic space $\C H^{n}$
by a neat lattice\footnote{
Here neat means that the group generated by the eigenvalues of any non-identity
element of the given lattice contains no roots of unity. Neat lattices are torsion-free.},
then the Dolbeault operator on X is of Fredholm type
and its index $\chi_{L^2}(X, \mathcal O)$ is given by
\begin{equation*}
  \chi_{L^2}(X, \mathcal O)
  = (-1)^n \frac{\vol X}{\vol\C P^{n}}
  + \begin{cases}
  0 &\text{if $n$ is odd} , \\
  \zeta(1-n)\sum\nolimits_{\mathcal C} |\Gamma_{\mathcal C}|
  &\text{if $n$ is even} .
  \end{cases}
\end{equation*}
\end{thm}

Here $|\Gamma_{\mathcal C}|$ denotes a basic invariant of the fundamental
group $\Gamma_{\mathcal C}$ of the cusp $\mathcal C$,
for each cusp $\mathcal C$ of $X$; compare \eqref{type}.

The second result concerns the signature operator on $X$
when $n$ is even, that is, when $m$ is a multiple of $4$.
It is also a consequence of the Hirzebruch proportionality principle
and Theorems \ref{intexinfin} and \ref{corrcom};
see also Example 2 in Section \ref{exaeta}.

\begin{thm}\label{thmsign}
If $X$ is a quotient of complex hyperbolic space $\C H^{n}$
by a neat lattice, where $n$ is even,
then the signature operator on $X$ is of Fredholm type
and its index $\sigma_{L^2}(X)$ is given by
\begin{multline*}
  \sigma_{L^2}(X) = \frac{\vol X}{\vol\C P^{n}}
   + 2^{n} \zeta(1-n) \sum\nolimits_{\mathcal C} |\Gamma_{\mathcal C}| \\
  + \nu (-1)^{n/2} \big({\textstyle\binom{n-2}{n/2}-\binom{n-2}{n/2-1}}\big) ,
\end{multline*}
where $\nu$ is equal to the number of ends of $X$.
\end{thm}

For complex hyperbolic manifolds with finite volume,
the space of $L^2$-harmonic forms has a topological interpretation.
In fact, in this case, the  $L^2$-signature coincides with the topological signature \cite{Zu}.
Formulas for $\sigma_{L^2}(X)$ are also stated in Theorem 7.6 of \cite{BM}
and Stern's article \cite{St} (compare Formula 6.4 there).
Our correction terms consist of two terms:
What we call the {\em high energy $\eta$-invariant}
can be identified with a zeta contribution in \cite{St} and with the unipotent
contribution in the Arthur-Selberg trace formula in \cite{BM}.
Our {\em low energy $\eta$-invariant} corresponds to the eta term in \cite{St}
and the weighted unipotent contribution in \cite{BM}.
Since our correction terms are obtained by different methods,
we obtain, in particular, different interpretations of the corresponding terms
in \cite{BM} and \cite{St}.

The formulas in Theorems \ref{thmdolb} and \ref{thmsign}
show that the volume of the quotient $X$ of $\C H^n$ in question
is a rational multiple of the volume of $\C P^n$.
This was already known by Harder's Gauss-Bonnet theorem
which says that $(n+1)\vol(X)/\vol\C P^{n}=(-1)^n\chi(X)$,
where $\chi(X)\in\Z$ denotes the Euler characteristic of $X$, see \cite{Ha}.
\tref{thmdolb} implies that $\vol(X)/\vol\C P^{n}$ is integral for odd $n$.
The question of the integrality of $\vol(X)/\vol\C P^{n}$ has been brought
to our attention by Martin Olbrich:
The half-integrality of $\vol(X)/\vol\C P^{n}$ implies
that certain Selberg type zeta functions are meromorphic.

As another example where our methods apply,
we mention the Dirac operator $D$ on the spinor bundle,
supposing that $M$ admits a spin structure.
The case $n=1$, that is, of surfaces of finite area with cusps of constant
negative curvature, has been dealt with in \cite{Bae}, see also \cite{BB1}.
In particular, $D$ is of Fredholm type if and only if the spin structure
is not periodic along (the cross sections of) any of the cusps,
see Theorem 2 in \cite{Bae} or Theorem 0.1 in \cite{BB1}.
We will discuss the correction terms for spinor bundles along complex hyperbolic cusps
in more detail in Examples \ref{exaspin} and Example \ref{exaeta2}.3.

Our formulas for complex hyperbolic cusps apply to more examples,
but we refer the reader to \tref{corrcom} for the full scope of our results.

In Chapter \ref{secpre} we discuss some notions and results
which are basic for our later investigations.
Chapter 3 is devoted to distance functions and their relation
to Dirac systems.
In particular, Section \ref{secdifu} contains a detailed study
of $C^2$ distance functions as we need it in our application
to Busemann functions.
In this section, we clarify and correct some of the statements
from \cite{BB2}.
Some essential parts of our later analysis
depend on our previous results in \cite{disy}\footnote{
In some cases, the work of  Marius Mitrea could also be used:
In Section 5 of \cite{Mi},
Mitrea investigates the regularity of
the Calder\'on projector for Dirac operators on Lipschitz domains
 with $C^{1,1}$ symbol  and  metric tensor,
using paradifferential calculus.}.
That the applications of these results are justified is the topic
of Section \ref{subdisy}.
In Chapter 4, we discuss boundary value problems and Fredholm
properties of Dirac systems which are underlying Dirac operators
over straight ends.
\pref{nonpar2} is one of the corner stones of our later discussion.
Chapter \ref{subdec} contains the first applications
to index formulas and a proof of \tref{pindfor}.
Chapter \ref{almflat} and Chapter \ref{secthi}
contain the proofs of Theorems \ref{mainth1} and \ref{mainth2},
respectively.
In Chapter \ref{susifo}, we prove Theorems \ref{intexinfin} and \ref{intmaxinf}.
The last three chapters are devoted to a discussion
of the index contributions of homogeneous cusps.
Ideas from the work of Deninger-Singhof \cite{DS} are basic
in our computation of high energy $\eta$-invariants
of Dirac operators on compact quotients of Heisenberg groups.
Following the discussion of Gordon-Wilson in \cite{GW},
we compute in Section \ref{replap} the spectrum
of twisted Laplacians on compact quotients of Heisenberg groups.
This is needed in our computation of high energy $\eta$-invariants
in Section \ref{heieta}.
In Chapter \ref{secle},
we discuss the low energy $\eta$-invariants of Dirac bundles over
complex hyp�erbolic cusps.
One of the main ingredients in this latter discussion is a theorem
of Kostant concerning Lie algebra cohomology
(Theorem 4.139 in \cite{KV}).

We would like to thank Patrick Ghanaat, Jean Louis Milhorat,
Henri Moscovici, and Martin Olbrich for helpful discussions.
We are also very grateful to the referee whose many detailed comments
and suggestions helped to improve the exposition of the article considerably.
We enjoyed the hospitality of the ESI in Vienna (W.B.),
the University of Nantes (W.B. and J.B.),
the MPIM in Bonn (J.B. and G.C.),
the University of Kyoto (J.B.),
and the MSRI in Berkeley (G.C.).
W.B. would like to thank the MPIM and HCM in Bonn for their continuous support.
J.B.  appreciates the support by SFB 647.
G.C. acknowledges the support by the grant GeomEinstein 06-BLAN-0154
gratefully.

\newpage
\section{Preliminaries}
\label{secpre}

Let $M$ be a Riemannian manifold of dimension $m$
with Levi-Civita connection $\nabla$ and curvature tensor $R$.
Let $E\to M$ be a Hermitian vector bundle over $M$,
endowed with a Hermitian connection $\nabla^E$
and associated curvature $R^E$.
Recall that we assume that the norms of $R$ and $R^E$ are uniformly bounded,
compare \eqref{boundrre}.

We denote by $C^\infty(M,E)$ and $L^2(M,E)$ the spaces of smooth
and square-integrable sections of $E$, respectively.
We let $H^1(M,E)$ be the closure of $C^\infty(M,E)$ with respect
to the $H^1$-norm, that is, the norm associated to the inner product
\begin{equation}
  (\sigma,\tau)_{H^1(M,E)} := (\sigma,\tau)_{L^2(M,E)}
  + (\nabla^E\sigma,\nabla^E\tau)_{L^2(M,E\otimes T^*M)} .
  \label{preh1}
\end{equation}
We denote by $C^\infty_c(M,E)$, $L^2_c(M,E)$, and $H^1_c(M,E)$
the subspaces of corresponding sections with compact support
and by $L^2_{\loc}(M,E)$ and $H^1_{\loc}(M,E)$
the spaces of measurable sections $\sigma$ of $E$
such that $\varphi\sigma$ belongs to $L^2(M,E)$ and $H^1(M,E)$,
respectively, for any smooth function $\varphi$ on $M$ with compact support.
In the case where the boundary of $M$ is non-empty,
we use a double index $cc$ to indicate compact support in the interior of $M$
and an index $0$ to indicate vanishing along the boundary.

For better readability,
we have arranged the rest of the preliminaries into sections.
In Section \ref{susdb} we introduce Dirac bundles and operators,
in Section \ref{decsig}
we collect some generalities about spinors,
and in Section \ref{susech}
we introduce complex hyperbolic spaces.

\subsection{Dirac Bundles}
\label{susdb}
We say that $E$ is a {\em Dirac bundle} over $M$
if $E$ is endowed with a compatible {\em Clifford multiplication}, that is, a field
\begin{equation}
  TM\times E\to E , \quad (x,v) \mapsto x\cdot v ,
\end{equation}
of bilinear maps such that
\begin{align}
  XX\sigma
  &= - |X|^2 \sigma ,
  \label{cl0} \\
  |X\sigma|
  &= |X| |\sigma| ,
  \label{cl1} \\
  \nabla^E_X(Y\sigma)
  &= (\nabla_XY)\sigma + Y(\nabla^E_X\sigma) ,
  \label{cl2}
\end{align}
for all vector fields $X,Y$ on $M$ and sections $\sigma$ of $E$,
where we use $X\sigma$ as a shorthand for $X\cdot\sigma$.

Suppose now that $E$ is a Dirac bundle over $M$.
Then the {\em Dirac operator} $D$ associated to $E$ is given by
\begin{equation}
  D\sigma = \sum_{1\le i\le m} X_i\nabla^E_{X_i}\sigma ,
  \label{diop}
\end{equation}
where $(X_1,\dots,X_m)$ is a local orthonormal frame of $M$
and $\sigma$ is a section of $E$.
For any function $\varphi$ on $M$ and section $\sigma$ of $E$,
\begin{equation}
  D(\varphi\sigma) = \grad\varphi\cdot\sigma + \varphi D\sigma .
  \label{diopps}
\end{equation}
In particular, the principal symbol of $D$ at $\xi\in T^*M$
is given by Clifford multiplication with the dual vector $\xi^\sharp\in TM$,
and hence $D$ is elliptic.
Note also that $D$ is formally self-adjoint, that is,
$D$ is symmetric on $C^\infty_{cc}(M,E)$.

Suppose now that $M$ has boundary, $N:=\partial M$,
let $T$ be the inward normal field along $N$, and set $W:=\nabla T$,
the {\em Weingarten operator} of $N$ with respect to $T$.
We assume that the operator norm of $W$ is uniformly bounded
by a constant $C_W$.
Change Clifford multiplication and connection of $E$ along $N$ by
\begin{align}
  X\ast\sigma
  &:= TX\sigma , \label{tast} \\
  \nabla^ T_X\sigma
  &:= \nabla^E_X\sigma - \frac12(WX)\ast\sigma
  = \nabla^E_X\sigma - \frac12(T\nabla_XT)\sigma .
  \label{nast}
\end{align}
It is well known that, with these new data,
the restriction of $E$ to $N$ is again a Dirac bundle
such that Clifford multiplication by $T$ is $\nabla^T$-parallel,
see for example Section 3.10.1 in \cite{Gi}.
The associated Dirac operator is given by
\begin{equation}
  D^T\sigma = \sum_{2\le i\le m} X_i\ast\nabla^ T_{X_i}\sigma
  = \sum_{2\le i\le m} TX_i\nabla^E_{X_i}\sigma + \frac{\kappa}2\sigma,
  \label{dast}
\end{equation}
where $(X_1,X_2,\dots,X_m)$ is a local orthonormal frame of $M$ along $N$
with $X_1=T$, and
\begin{equation}
  \kappa = \tr W
  \label{meanc}
\end{equation}
is the {\em mean curvature} of $N$ with respect to $T$.
The curvature of $\nabla^T$ is
\begin{equation}
  R^T(X,Y)\sigma
  = R^E(X,Y)\sigma - \frac12(R(X,Y)T)\ast\sigma
  - \frac14[WX,WY]\sigma .
  \label{cast}
\end{equation}
The general Bochner identity \cite[Theorem II.8.2]{LM} implies that
\begin{multline}
  ( \nabla^E\sigma_1,\nabla^E\sigma_2 )_{L^2(M,E\otimes T^*M)}
  + ( K^E\sigma_1,\sigma_2 )_{L^2(M,E)} \\
  = ( D\sigma_1,D\sigma_2 )_{L^2(M,E)}
  + ( D^T\sigma_1 - \frac{\kappa}2\sigma_1,\sigma_2 )_{L^2(N,E)} ,
  \label{lic1}
\end{multline}
for all $\sigma_1,\sigma_2\in C^\infty_{c}(M,E)$,
where $K^E$ is a curvature term,
\begin{equation}
  K^E\sigma = \sum_{1\le i<j\le m}X_iX_j R^E(X_i,X_j)\sigma .
  \label{lic2}
\end{equation}
We see that the operator norm of $K^E$ is bounded by $m(m-1)C_R^E/2$
and conclude that on sections with compact support in the interior of M,
the graph norm of $D$ is equivalent to the $H^1$-norm .

Since $N$ has no boundary, \eqref{lic1} applied to $N$ turns into
\begin{multline}
  ( \nabla^T\sigma_1,\nabla^T\sigma_2 )_{L^2(N,E\otimes T^*N)}
  + ( K^T\sigma_1,\sigma_2 )_{L^2(N,E)} \\
  = ( D^T\sigma_1, D^T\sigma_2 )_{L^2(N,E)} ,
  \label{lic3}
\end{multline}
where $K^T$ denotes the curvature term built from $R^T$
as $K^E$ is built from $R^E$ in \eqref{lic2}.
We see that the operator norm of $K^T$ is bounded here by
\begin{equation}
  \frac{(m-1)(m-2)}2 \big( C_R^E + \frac12 C_R + \frac14 C_W^2 \big) ,
  \label{lic4}
\end{equation}
where $C_W$ is a uniform bound for the operator norm of $W$
(compare \eqref{boundw}),
and conclude now that, along $N$, the graph norm of $D^T$
is equivalent to the $H^1$-norm.

Let $E$ be a Dirac bundle over $M$.
A {\em super-symmetry} of $E$ is an orthogonal decomposition $E=E^+\oplus E^-$,
where $E^\pm$ are parallel Hermitian subbundles of $E$
such that $XE^+\subseteq E^-$ and $XE^-\subseteq E^+$,
for all vector fields $X$ of $M$.
In particular, $E^+$ and $E^-$ are of the same dimension.
If $E=E^+\oplus E^-$ is a super-symmetry, then the Dirac operator $D$ of $E$
maps sections of $E^+$ into sections of $E^-$ and conversely
and therefore can be written as
\begin{equation}
  D = \begin{pmatrix} 0 & D^- \\ D^+ & 0 \end{pmatrix}
  \label{disusy}
\end{equation}
with respect to the super-symmetry.
We can also think of a super-symmetry as a parallel field
of unitary involutions of $E$ which anti-commute with Clifford multiplication,
where $E^\pm$ is the subbundle of eigenspaces of the involutions
for the eigenvalue $\pm1$, respectively.

If $M$ is oriented and $m=\dim M$ is even,
then the {\em complex volume form} of $M$ is defined to be
\begin{equation}
  \omega_{\C} := i^{m/2}X_1\cdots X_m\in\cl(M) ,
  \label{covofo}
\end{equation}
where $(X_1,\ldots,X_m)$ is an oriented  local orthonormal frame of $M$.
For any Dirac bundle $E$ over $M$,
multiplication by $\omega_{\C}$ is a parallel field of unitary involutions
of $E$ which anti-commutes with Clifford multiplication with vector fields,
and hence it defines a  super-symmetry $E=E^+\oplus E^-$.

Suppose now $M$ is complete and that the boundary of $M$ is empty,
and consider $D$ as an unbounded operator in $L^2(M,E)$
with domain $C^\infty_c(M,E)$.
Since $D$ is symmetric on $C^\infty_c(M,E)$, it is closable in $L^2(M,E)$.
Since the graph norm of $D$ is equivalent to the $H^1$-norm,
$H^1(M,E)$ is the domain of the closure of $D$.
By \cite{Wo} or Theorem II.5.7 in \cite{LM},
$D$ on $H^1(M,E)$ is self-adjoint in $L^2(M,E)$.

\subsection{Decomposition of Spinors}
\label{decsig}
Let $m$ be even, $m=2n$,
and consider the complex Clifford algebra $\cl(2n)=\cl(\R^{2n})$,
where we denote the complex structure on $\cl(2n)$ by $\sqrt{-1}$.
Fix an orthonormal basis $(e_1, \dots , e_{2n})$ of $\R^{2n}$
and set\footnote{Note that the sign convention
is opposite to the one in \cite{LM}, page 43.}
\begin{align}
  \omega_j := \sqrt{-1}e_{2j-1}e_{2j} \in \cl(2n) , \quad 1\le j \le n .
  \label{involu}
\end{align}
Then
\begin{equation}\label{omecom}
  \omega_j^2 = 1
  \quad\text{and}\quad
  \omega_j\omega_k = \omega_k\omega_j ,
\end{equation}
for all $1\le j,k\le n$,
and the complex volume form is given by
\begin{equation}
  \omega_{\C} = \omega_1 \cdots \omega_n ,
   \label{covofo2}
\end{equation}
compare \eqref{covofo}.
Let $\Sigma=\Sigma_{2n}$ be the spinor representation.
Then Clifford multiplication by the $\omega_j$
defines unitary involutions of $\Sigma$.
By \eqref{omecom},
there is an orthogonal decomposition of $\Sigma$
into simultaneous eigenspaces $\Sigma_\varepsilon$,
where $\varepsilon$ runs over all $n$-tuples in $\{1,-1\}^n$
and where $\omega_j$ acts by multiplication with $\varepsilon_j$
on $\Sigma_\varepsilon$, $1\le j\le n$.
Because Clifford multiplication with $e_{2j-1}$ or $e_{2j}$
anti-commutes with $\omega_j$ and commutes with $\omega_k$,
for $1\le k\ne j\le n$, we have
\begin{equation}
  e_{2j-1}\Sigma_\varepsilon
  = e_{2j}\Sigma_\varepsilon
  = \Sigma_\delta ,
  \label{sigeps}
\end{equation}
where $\delta_k=\varepsilon_k$ for all $1\le k\ne j\le n$ and $\delta_j=-\varepsilon_j$.
In particular, all the subspaces $\Sigma_\varepsilon$ have the same dimension,
which is, for that reason, equal to $\dim\Sigma/2^n=1$.
Clifford multiplication by the complex volume form
acts by $\varepsilon_1\cdots\varepsilon_n$ on $\Sigma_\varepsilon$,
by \eqref{covofo2},
and hence the summands of the usual super-symmetry
\begin{equation}
  \Sigma = \Sigma^+ \oplus \Sigma^-
  \label{sisusy}
\end{equation}
are given by
\begin{equation}
  \Sigma^+ = \oplus_{\varepsilon_1\cdots\varepsilon_n=1} \Sigma_{\varepsilon}
  \quad\text{and}\quad
  \Sigma^- = \oplus_{\varepsilon_1\cdots\varepsilon_n=-1} \Sigma_{\varepsilon} .
  \label{sisusy2}
\end{equation}

\subsection{Complex Hyperbolic Spaces}
\label{susech}
The discussion in this section will be used in Section \ref{suschc}.
We represent complex hyperbolic space $\C H^n$
by the symmetric pair $(\SU(1,n), \Se(\U(1)\times\U(n)))$
and endow the Lie algebra $\mathfrak{su}(1,n)$ of $\SU(1,n)$
with the non-degenerate symmetric bilinear from
\begin{equation}
  (X,Y) := \frac12\Re\tr XY ,
  \label{chpr}
\end{equation}
a multiple of the Killing form of  $\mathfrak{su}(1,n)$.
We identify
\begin{equation}
  \Se(\U(1)\times\U(n))
  = \left\{ \begin{pmatrix}
  \det A^{-1} & 0 \\ 0 & A \end{pmatrix} \biggm| A\in\U(n) \right\}
  \cong \U(n)
  \label{chK}
\end{equation}
and, correspondingly,
\begin{equation}
  \mathfrak{s}(\mathfrak{u}(1)\oplus\mathfrak{u}(n))
  = \left\{  \begin{pmatrix}
  - \tr A & 0 \\ 0 & A \end{pmatrix} \biggm| A \in \mathfrak u(n) \right\}
  \cong \mathfrak u(n) .
  \label{chk}
\end{equation}
The orthogonal complement $\mathfrak p$
of $\mathfrak u(n)$ in $\mathfrak{su}(1,n)$ is
\begin{equation}
  \mathfrak p = \left\{
  \begin{pmatrix} 0 & x^* \\ x & 0 \end{pmatrix} \biggm| x \in \C^n \right\}
  \cong \C^n ,
  \label{chp}
\end{equation}
where we note that the latter isomorphism corresponds to the standard
complex structure and Riemannian metric of $\C H^n$.
With respect to the identifications \eqref{chK} -- \eqref{chp},
we get
\begin{equation}
  [A,B] = AB - BA , \;
  [A,x] = Ax + x \cdot \tr A , \;
  [x,y] = xy^* - yx^*
  \label{chlb}
\end{equation}
for the different Lie brackets and
\begin{equation}
  \alpha(A) x := \Ad_A x = A x \det A
  \label{chir}
\end{equation}
for the adjoint representation $\alpha:\U(n)\to\SO(\mathfrak p)\cong\SO(2n)$.
We note that $\alpha$ is an $n+1$ to $1$ immersion.
If $n$ is odd,
then $\alpha$ lifts to $\hat\alpha:\U(n)\to\Spin(\mathfrak p)$.
If $n$ is even, then $\alpha$ does not lift.

We note that the coefficients of the matrix $xy^* - yx^*\in\mathfrak u(n)$
in \eqref{chlb} are $x_j\bar y_k-y_j\bar x_k$.
In particular, for the standard unit vectors $e_j$ and $e_k$ in $\C^n$
and complex numbers $x,y$, we have
\begin{equation}\label{sypach4}
  [xe_j,ye_k] = x\bar yE_{jk} - y\bar xE_{kj} \in\mathfrak u(n) ,
\end{equation}
where $E_{jk}$ denotes the matrix with entries $\delta_{jk}$.

Let $T=e_1\in\C^n\cong\mathfrak p$ and set $\mathfrak a:=\R T$.
The orthogonal complement of $\mathfrak a$ in $\C^n$ consists of all
$x\in\C^n$ with $x_1\in\Im\C$, that is, $x_1$ is purely imaginary.
Let $\mathfrak z:=\R Z$ with
\begin{equation}
  Z := ie_1 - iE_{11} \in \mathfrak p \oplus \mathfrak u(n) .
  \label{chz}
\end{equation}
We have $[\mathfrak z,\mathfrak z]=0$ and
\begin{equation}
  [T,Z] = 2Z .
  \label{chtz}
\end{equation}
Let $\mathfrak x$ be the space of all
\begin{equation}
  X_x:= x + \bar x_2E_{12} - x_2E_{21}
  + \ldots + \bar x_nE_{1n} - x_nE_{n1}
  \in \mathfrak p \oplus \mathfrak u(n) ,
  \label{chx}
\end{equation}
where $x\in\C^{n-1}=\{x\in\C^n \mid x_1=0\}$.
Then $[\mathfrak z,\mathfrak x]=0$ and
\begin{align}
  [T,X_x] &= X_x , \label{chtx} \\
  [X_x,X_y] &= 2\Im(\bar xy) Z . \label{chxx}
\end{align}
Set $\mathfrak n:=\mathfrak z\oplus\mathfrak x$
and $\mathfrak s:=\mathfrak a\oplus\mathfrak n$.
By the above, $\mathfrak n$ is a two-step nilpotent subalgebra
of $\mathfrak{su}(1,n)$ and $\mathfrak s$ is a
solvable extension of $\mathfrak n$.
The subgroups $A$, $N$, and $S$ of $\SU(1,n)$
corresponding to $\mathfrak a$, $\mathfrak n$, and $\mathfrak s$
satisfy $S=AN$ and $\SU(1,n)=\U(n)AN$
(Iwasawa decomposition of $\SU(1,n)$).

Let $p\in\C H^n$ be the point fixed by $\U(n)$.
Then the orbit map
\begin{equation}
  \Phi: S \to \C H^n , \quad \Phi(s) = sp ,
  \label{chom}
\end{equation}
is a diffeomorphism,
that is, $S$ acts simply transitively on $\C H^n$.
Endow $S$ with the left-invariant Riemannian metric
such that the differential $d\Phi:T_eS\to T_p\C H^n$ is isometric.
Since $S$ acts isometrically on $\C H^n$,
we then get that $\Phi$ is an $S$-equivariant isometry.
That is, we can think of $\C H^n$ as $S$,
endowed with the chosen left-invariant metric.
With respect to this metric,
we get that $\mathfrak a$, $\mathfrak z$, and $\mathfrak x$
are perpendicular and that
\begin{equation}
  |T| = 1 , \quad  |Z| = 1 , \quad
  \langle X_{x},X_{y}\rangle = \Re\bar xy .
  \label{chsm}
\end{equation}
Define
\begin{equation}
  JX_x = J_ZX_{x} := X_{ix} .
  \label{chj}
\end{equation}
Then $J$ is skew-symmetric with $J^2=-1$.
Moreover, by \eqref{chxx} and \eqref{chsm},
\begin{equation}
  \langle [X_{x},X_{y}] , Z\rangle
  = 2 \langle JX_{x},X_{y}\rangle .
  \label{chht}
\end{equation}
As a preferred basis of $\mathfrak s$,
we choose  the $2n$-tuple of vectors $X_1:=T$, $Y_1=Z$,
\begin{equation}
  X_j := e_j + E_{1j} - E_{j1}
  \quad\text{and}\quad
  Y_j := JX_j = ie_j - iE_{1j} - iE_{j1} ,
  \label{heibas}
\end{equation}
where $2\le j\le n$.
By \eqref{chxx} and \eqref{chj},
\begin{equation}
  [X_j,Y_k] = 2 \delta_{jk} Z .
  \label{heibas2}
\end{equation}
In conclusion,
$N$ is isomorphic to the standard Heisenberg group of dimension $2n-1$.
Since left-invariant vector fields on $S$ are $T$-parallel,
the Weingarten operator $W=\nabla T$ of $N$ in $S$ is given by
\begin{equation}
  WZ = [Z,T] = - 2Z \quad\text{and}\quad  WX_x = [X_x,T] = - X_x ,
  \label{heiwzx}
\end{equation}
by \eqref{chtz} and \eqref{chtx}.

\section{Dirac Systems and Distance Functions}
\label{dsdf}

\subsection{Dirac Systems}
\label{preds}
The setup and the results from \cite{disy}
are fundamental for the discussion of this section.
Let $I\subseteq\R$ be an interval
and $H$ be a separable complex Hilbert space.
Fix an origin $t_0\in I$.

For each $t\in I$, let $(.,.)_t$ be a scalar product on $H$
which is compatible with the Hilbert space structure of $H$
and such that $(.,.)_{t_0}$ coincides with the given scalar product of $H$.
Let $\|.\|_t$ be the norm associated to $(.,.)_t$.
Let $H_t$ be $H$, but equipped with $(.,.)_t$,
and denote by $\mathcal H$ the family of Hilbert spaces $H_t$, $t\in I$.
Assume that, for all $a<b$ in $I$,
there is a constant $C=C(a,b)$ such that
\begin{equation}
   |(\sigma_1,\sigma_2)_s - (\sigma_1,\sigma_2)_t|
   \le C \|\sigma_1\|_{s}\,\|\sigma_2\|_{s} |s-t| ,
   \label{dsax1}
\end{equation}
for all $s,t\in[a,b]$ and $\sigma_1,\sigma_2\in H$.
In other words,
if $G_t\in\mathcal L(H)$ denotes the positive definite and symmetric operator
on $H=H_{t_0}$ with
\begin{equation}
  (G_t\sigma_1,\sigma_2)_{t_0} = (\sigma_1,\sigma_2)_t ,
  \label{dsgt}
\end{equation}
for all $\sigma_1,\sigma_2\in H$, then the map
\begin{equation*}
  G: I\to\mathcal L(H) , \quad G(t):=G_t ,
   \label{dsax1a}
\end{equation*}
is in $\Lip_{\loc}(I,\mathcal L(H))$.
In particular, $G$ is weakly differentiable almost everywhere in $I$
with weak derivative $G'$ in $L_{\loc}^\infty(I,\mathcal L(H))$.
Moreover, $G'_t$ is symmetric on $H_{t_0}$ (for almost all $t\in I$)
and we have
\begin{equation}
  \Gamma := \frac12G^{-1}G' \in L_{\loc}^\infty(I,\mathcal L(H)) .
   \label{dsGamm}
\end{equation}
We set
\begin{equation}
  \partial := \big(\frac{d}{dt} + \Gamma\big):
  \Lip_{\loc}(I,H) \to L_{\loc}^\infty(I,H) .
  \label{dspart}
\end{equation}
By the definition of $\partial$,
the function $(\sigma_1,\sigma_2)=(\sigma_1(t),\sigma_2(t))_t$ satisfies
\begin{equation}
  (\sigma_1,\sigma_2)'
  = (\partial\sigma_1,\sigma_2) + (\sigma_1,\partial\sigma_2) ,
  \label{dspain}
\end{equation}
for all $\sigma_1,\sigma_2\in\Lip_{\loc}(I,H)$,
where the prime indicates differentiation with respect to $t$.

As a second data, let $\mathcal A$ be a  family of operators $A_t$, $t\in I$, on $H$
with common dense domain $H_A$
such that  $A_t$ is self-adjoint in $H_t$
and such that the inclusion $H_A\hookrightarrow H$ is compact
with respect to the graph norms of the $A_t$.
Assume that, for all $a<b$ in $I$,
there is a constant $C=C(a,b)$ such that
\begin{equation}
   | (A_s\sigma_1,\sigma_2)_{s} - (A_t\sigma_1,\sigma_2)_{t}|
   \le C (\|\sigma_1\|_s + \|A_s\sigma_1\|_{s}) \|\sigma_2\|_{s} |s-t| ,
   \label{dsax2}
\end{equation}
for all $s,t\in[a,b]$ and $\sigma_1,\sigma_2\in H_A$.

As a final data, let
\begin{equation}
  T \in \Lip_{\loc}(I,\mathcal L(H)) \cap L_{\loc}^\infty(I,\mathcal L(H_A)) ,
  \label{dsax3}
\end{equation}
and suppose that
\begin{alignat}{2}
  T_t^* = T_t^{-1} &= - T_t \quad &
  &\text{on $H_t$, $\forall t\in I$,} \label{dsax3a}  \\
  A_tT_t &= - T_tA_t \quad&
  &\text{on $H_A$, $\forall t\in I$,} \label{dsax3b}  \\
  \partial T &= T\partial & &\text{on $\Lip_{\loc}(I,H)$.}
  \label{dsax3c}
\end{alignat}
Following \cite{disy},
a {\em Dirac system} over $I$ consists of data
$\mathcal H$, $\mathcal A$, and $T$ as above.

Let $\mathcal D:=(\mathcal H,\mathcal A,T)$ be a Dirac system over $I$.
Set
\begin{equation}
   \mathcal L_{\loc}(\mathcal D) := \Lip_{\loc}(I,H) \cap L_{\loc}^\infty(I,H_A) ,
   \label{dslloc}
\end{equation}
and denote by $\mathcal L_{c}(\mathcal D)$ and $\mathcal L_{cc}(\mathcal D)$
the subspaces of $\mathcal L_{\loc}(\mathcal D)$ of maps with compact support
in $I$ and the interior of $I$, respectively.
On $\mathcal L_{c}(\mathcal D)$, we define the inner product
\begin{equation}
  (\sigma_1,\sigma_2) := \int_I (\sigma_1,\sigma_2)
  = \int_I (\sigma_1(t),\sigma_2(t))_t dt ,
  \label{l2d}
\end{equation}
and let $L^2(\mathcal D)$
be the corresponding Hilbert space of square-integrable maps,
also denoted by $L^2(\mathcal H)$.

The {\em Dirac operator} of $\mathcal D$ is the operator
\begin{equation}
  D:= T(\partial + A): \mathcal L_{\loc}(\mathcal D) \to L_{\loc}^\infty(I,H) .
  \label{dsdo}
\end{equation}
By \eqref{dspain} and \eqref{dsax3a}--\eqref{dsax3c},
\begin{equation}
  \int_{[a,b]} (D\sigma_1,\sigma_2)
  = \int_{[a,b]} (\sigma_1,D\sigma_2) - (\sigma_1,T\sigma_2)\big|_a^b ,
  \label{dspaind}
\end{equation}
for all $\sigma,\tau\in\mathcal L_{\loc}(\mathcal D)$
and $a<b$ in $I$.

A {\em super-symmetry} for a Dirac system $\mathcal D$ as above
is a decomposition $H=H^+\oplus H^-$ such that,
with $H_A^{\pm}:=H_A\cap H^\pm$,
\begin{align}
  &\text{$H^+ \perp H^-$ in $H_t$}
  \quad\text{and}\quad
  \text{$T_tH^\pm=H^\mp$} , \label{susy1} \\
  &\text{$H_A = H_A^+ \oplus H_A^-$}
  \quad\text{and}\quad
  \text{$A_t H_A^{\pm} \subseteq H^{\pm}$} . \label{susy2}
\end{align}
We write $H_t^\pm$ for $H^\pm$ endowed with the inner product $(.,.)_t$.
By \eqref{susy1}, $H_t=H_t^+\oplus H_t^-$ as a Hilbert space.
By \eqref{susy1},
\begin{equation}
  A_t = \begin{pmatrix} A_t^+ & 0 \\ 0 & A_t^- \end{pmatrix} ,
  \label{susya}
\end{equation}
where $A_t^\pm$ is a self-adjoint operator in $H_t^\pm$
with domain $H_A^\pm$ and where
\begin{equation}
  A_t^- = T_tA_t^+T_t = T_t(-A_t^+)T_t^{-1} ,
  \label{susyac}
\end{equation}
by \eqref{dsax3b}.
We can decompose
\begin{equation}
  L^2(\mathcal D)
  = L^{2+}(\mathcal D) \oplus L^{2-}(\mathcal D)
  =  L^{2}(\mathcal H^+) \oplus L^{2}(\mathcal H^-) ,
  \label{susysdc}
\end{equation}
where $L^{2}(\mathcal H^\pm)$ consists of the subspace of sections
in $L^{2}(\mathcal H)$ with image in $H^+$.
Similar notation will be employed for other spaces.

By \eqref{susy2} and the definition of $\partial$, see \eqref{dspart},
\begin{equation}
  \partial = \begin{pmatrix} \partial^+ & 0 \\ 0 & \partial^- \end{pmatrix} ,
  \quad\text{where}\quad
  \partial^- = T\partial^+T^{-1} ,
  \label{susypc}
\end{equation}
by \eqref{dsax3c}.
Hence, by \eqref{susy1},
\begin{equation}
  D = \begin{pmatrix} 0 & D^- \\ D^+ & 0 \end{pmatrix} .
  \label{susydd}
\end{equation}
Clearly, $D^-$ is the formal adjoint of $D^+$.

\subsection{Distance Functions}
\label{secdifu}
Let $U$ be an open subset in a Riemannian manifold $M$.
We say that a function $f:U\to\R$ is a {\em distance function}
if $f$ is $C^1$ and $T:=\grad f$ is a unit vector field.
There is a synthetic characterization of distance functions,
compare \cite[pp 24--25]{BGS} or also Proposition IV.3.1 in \cite{Ba}.
If $f$ is a distance function, then the solution curves
of the vector field $T$ are unit speed geodesics,
called {\em $T$-geodesics}.

Busemann functions are $C^2$ distance functions,
see Proposition 3.1 in \cite{HI} or Proposition IV.3.2 in \cite{Ba}.
We assume from now on that $f:U\to\R$ is a $C^2$ distance function.
Then $T:=\grad f$ is a $C^1$ unit vector field
and the {\em cross sections} $N_t=f^{-1}(t)$ are $C^2$ hypersurfaces.
For simplicity, we assume throughout that the cross sections $N_t$ are compact
and that the flow of $T$ induces a $C^1$ diffeomorphism
\begin{equation}
  F: I\times N \to U,
  \label{uprodu}
\end{equation}
where $I$ is some interval and $N=N_{t_0}$ for some $t_0\in I$.
In what follows, we often identify $U$ with $I\times N$
by identifying $(t,x)\in I\times N$ with $F(t,x)\in U$.

Let $c = c(s)$ be a $C^1$ curve in $U$ and $T(s) := T(c(s))$ be $T$ along $c$,
a $C^1$ curve of unit vectors.
Then the variation field $J=J(t):=(\partial_s\gamma)(0,t)$
of the geodesic variation $\gamma_s=\gamma(s,t):=\exp(t T(s))$
satisfies $J(0)=\dot c(0)$.
A vector field which arises in this way will be called a {\em $T$-Jacobi field}.

\begin{lem}\label{jacob}
A $T$-Jacobi field $J$ satisfies the Jacobi equation
\begin{equation*}
   J'' + R(J,T)T = 0 .
\end{equation*}
Moreover, $J$ and $J'$ depend continuously on $J(0)$.
\end{lem}

\begin{proof}
Recall that the Riemannian manifold $M$ is smooth,
hence its geodesic flow $\Phi=\Phi(t,v)$, $t\in\R$ and $v\in TM$, is also smooth.
Since $\gamma'(s,t)=\Phi(t,T(s))$, we get that $\gamma$ and $\gamma'$ are $C^1$.
Therefore
\begin{equation}
  J=\partial_s\gamma
  \quad\text{and}\quad
 J' = \nabla_t \partial_s \gamma = \nabla_s \partial_t \gamma = \nabla_sT
 \label{jacob2}
\end{equation}
exist and are continuous.
Moreover,
\begin{equation}
  (\partial_s\gamma')(s,t) = \Phi_{t*}(\partial_sT(s))  = (J(s,t),J'(s,t))
  \label{dgeoflo}
\end{equation}
with respect to the standard decomposition of $TTM$
in horizontal and vertical component,
see for example Proposition IV.1.13 in \cite{Ba}.
Hence $J$ and $J'$ depend continuously on $\dot c(0)$
and $J$ satisfies the asserted Jacobi equation.
\end{proof}

\begin{rem}\label{corrbb}
With respect to the $(t,x)$-coordinates,
the Riemannian metric on $U$ is of the form $g = dt^2 + g_t$,
where $g_t$, $t\in I$, is a family of Riemannian metrics on $N$.
In \cite{BB2}, pages 596 and 609, it is stated erroneously
that $g_t$ and $\partial_tg_t$ are $C^1$ on $U$.
This is wrong in general, since it would imply that $T$ is $C^2$.
Clearly, since $T$ is $C^1$, $g_t(x)$ is $C^1$ in $(t,x)$.
\lref{jacob} implies that $g_t(x)$ is two times continuously differentiable in $t$.
This is sufficient for the discussion in \cite{BB2} and the arguments below.
\end{rem}

For $t\in I$, we let $S=S_t$ and $W=W_t$  be, respectively,
the {\em second fundamental form} and the {\em Weingarten operator}
of the $C^2$ submanifold $N_t$
with respect to the normal vector field $T$,
\begin{equation}\label{sff}
    WX = \nabla_XT ,\quad
    S(X,Y) = \la \nabla_XY, T \ra = - \la WX, Y \ra ,
\end{equation}
where $X$ and $Y$ are $C^1$ vector fields tangent to $N_t$.
Since $T$ is $C^1$, $S$ and $W$ are continuous tensor fields over $U$.
By \eqref{jacob2}, Jacobi fields $J$ as in \lref{jacob} satisfy $J'=WJ$.

Let $E\to M$ be a smooth vector bundle with smooth connection $\nabla^E$.

\begin{lem}\label{c1reg}
Let $X$ be a vector field and $\sigma$ be a section of $E$ over $U$,
respectively.
Assume that the restrictions of $X$ and $\sigma$ to $N$ are $C^1$
and that $X$ and $\sigma$ are parallel in the $T$-direction.
Then $X$ and $\sigma$ are $C^1$.
Moreover, $\nabla^E_T\nabla^E_X\sigma$ exists, is continuous, and satisfies
\begin{equation*}
  \nabla^E_T\nabla^E_X\sigma + \nabla^E_{WX}\sigma + R^E(X,T)\sigma = 0 .
\end{equation*}
\end{lem}

\begin{proof}
Let $\Psi:\R\times(TM\oplus E)\to E$ be the smooth map
which associates to $t\in\R$ and $(v,e)\in TM\oplus E$
(where $v\in TM$ and $e\in E$ have the same foot point)
the parallel translate $\sigma(t)$ of $e$
along the geodesic $\gamma$ with $\gamma'(0)=v$.
Then, with $\sigma$ as in the assertion,
we have $\sigma(F(t,x))=\Psi(t-t_0,T(x),\sigma(x))$,
where we recall that $N=N_{t_0}$.
Hence $X$ and $\sigma$ are $C^1$,
where $X$ corresponds to the special case $E=TM$.

Since $\nabla^E_T\sigma=0$ and $T$ is $C^1$,
the $T$ derivatives of the coefficients of $\sigma$
with respect to a smooth local frame of $E$ are $C^1$.
Hence $\nabla^E_T\nabla^E_X\sigma$ exists, is continuous,
and is given by
\begin{align*}
  \nabla^E_T\nabla^E_X\sigma
  &= \nabla^E_X\nabla^E_T\sigma
  - \nabla^E_{\nabla_XT}\sigma + R^E(T,X)\sigma \\
  &= - \nabla^E_{WX}\sigma - R^E(X,T)\sigma .
  \qedhere
\end{align*}
\end{proof}

Among others, the case $E = TM$ is interesting.
In this case, vector fields over $N$ which are tangent to $N$ can,
in general, only be chosen to be $C^1$.

\begin{cor}\label{ricca}
The tensor field $W$ has a continuous derivative $W'$ in the $T$-direction
and satisfies the Riccati equation
\begin{equation*}
  W' + W^2 + R(.,T)T = 0 .
\end{equation*}
\end{cor}

\begin{proof}
Choose $\sigma=T$ in \lref{c1reg} and recall that $W=\nabla T$.
\end{proof}

The eigenvalues $\kappa_2,\ldots,\kappa_m$ of $W_t$
are the {\em principal curvatures} of the cross section $N_t$.
We let
\begin{equation}
    \kappa := \kappa_2 + \cdots + \kappa_m = \tr W = \div T .
    \label{kappa}
\end{equation}
The maps
\begin{equation}\label{shift}
    F_{t}: N = N_{t_0} \to N_t , \quad F_t(x) := F(t,x)
\end{equation}
are diffeomorphisms whose Jacobian determinants will be denoted by $j=j(t,x)$.
Since $\kappa=\div T$, the latter satisfy the differential equation
\begin{equation}
    j' = \kappa j .
    \label{odejac}
\end{equation}
By \cref{ricca}, we also have
\begin{equation}
  \kappa' = - \| W \|^2 - \Ric(T,T) ,
  \label{ricca2}
\end{equation}
where $\|W\|=(\tr W^2)^{1/2}$ is the Euclidean norm of $W$.

Let $C_R$, $C_R^E$, and $C_W$ be uniform upper bounds
for the operator norms of the curvature $R$ of $M$,
the curvature $R^E$ of $E$, and $W$, respectively.
Then $\kappa$, the $t$-derivative $\kappa'$ of $\kappa$, and $\|W\|$
are uniformly bounded,
and as respective uniform upper bounds
$C_\kappa$, $C_\kappa'$, and $C_w$ we may take
\begin{equation}
  C_\kappa = mC_W , \quad C_\kappa' = m(C_W^2 + C_R) ,
   \quad C_w = \sqrt m C_W ,
  \label{bounds1}
\end{equation}
where we use \eqref{ricca2} for the second assertion.
By \eqref{odejac}, we have
\begin{equation}
  e^{-C(t-s)} j(s,x) \le j(t,x) \le e^{C(t-s)} j(s,x) ,
  \label{estjac}
\end{equation}
or all $s<t$ in $I$ and $x\in N$, where $C=C_\kappa$.

\subsection{From Distance Functions to Dirac Systems}
\label{subdisy}
Let $E\to M$ be a smooth Dirac bundle.
Denote the Hermitian product on $E$ by  $\la .\,,.\ra$.
Our aim is to identify these data over $U$ with a Dirac system over $I$
as in Section \ref{preds}.

For any $t\in I$ and any given Riemannian or Hermitian vector bundle over $U$
with any given metric connection,
we let $P_t^{/\!\!/}$ be parallel translation along the $T$-geodesics from $N$ to $N_t$.
For a section $\sigma$ of the vector bundle over $N$,
we define a section $P^{/\!\!/}\sigma$ over $U$ by
\begin{equation}
  (P^{/\!\!/}\sigma)(t,x) := P_t^{/\!\!/}(\sigma(x)) , \quad x \in N .
  \label{sigmat}
\end{equation}
Thus $P^{/\!\!/}\sigma$ is the extension of $\sigma$ to $U$
which is parallel along the $T$-geodesics,
and this point of view is convenient in arguments and formulations below.
Furthermore, time dependent sections over $N$
correspond to the space of all sections over $U$,
\begin{equation}
  (P^{/\!\!/}\sigma)(t,x) := P_t^{/\!\!/}(\sigma(t,x)) , \quad t\in I, x\in N .
  \label{sigmatt}
\end{equation}
We also let $P_t^{/\!\!/}\sigma:=P^{/\!\!/}\sigma|_{N_t}$.

Now let $H:=L^2(N,E)$,
the Hilbert space of square integrable sections of $E$ over $N=N_{t_0}$.
For $\sigma,\tau\in H$,
the $L^2$ product of the sections $P_t^{/\!\!/}\sigma,P_t^{/\!\!/}\tau$
with respect to $N_t$ is given by
\begin{equation}
  ( \sigma,\tau )_t
  := \int_{N} \la \sigma(x),\tau(x) \ra j(t,x) dx ,
  \label{l2proj}
\end{equation}
where $dx$ denotes the volume element of $N$.
Hence, for each $t\in I$,
the correspondence $\sigma\leftrightarrow P_t^{/\!\!/}\sigma$
identifies the Hilbert space $L^2(N_t,E)$ topologically with $H$.
The following estimate settles the requirement on the family $\mathcal H$
formulated in \eqref{dsax1}.

\begin{lem}\label{ax1}
For all $s<t$ in $I$ and $\sigma_1,\sigma_2\in H$,
\begin{equation*}
   |(\sigma_1,\sigma_2)_{t} - (\sigma_1,\sigma_2)_{s}|
   \le (e^{C(t-s)} - 1) \|\sigma_1\|_{s}\,\|\sigma_2\|_{s} ,
\end{equation*}
where $C=C_\kappa$.
\end{lem}

\begin{proof}
By \eqref{l2proj} and \eqref{estjac},
\begin{align*}
  |(\sigma_1,\sigma_2)_{t} - (\sigma_1,\sigma_2)_{s}|
  &\le \int_N | \la  \sigma_1(x),\sigma_2(x)\ra (j(t,x) - j(s,x) | dx \\
  &\le \int_N |\sigma_1(x)| |\sigma_2(x)| (e^{C(t-s)} - 1) j(s,x) dx \\
  &\le (e^{C(t-s)} - 1) \|\sigma_1\|_s \|\sigma_2\|_s .
  \qedhere
\end{align*}
\end{proof}

\begin{lem}\label{ax2a}
For all $s<t$ in $I$ and $C^1$ sections $\sigma$ of $E$ over $U$
which are parallel in the $T$-direction,
\begin{equation*}
  e^{C_0(s-t)} ( \| \sigma \|_{s}^2 + \| \nabla^E\sigma \|_{s}^2 )
  \le \| \sigma \|_t^2 + \| \nabla\sigma \|_t^2
  \le e^{C_0(t-s)} ( \| \sigma \|_{s}^2 + \| \nabla^E\sigma \|_{s}^2 ) ,
\end{equation*}
where $C_0=C_\kappa + mC_R^E + 2 C_W$.
\end{lem}

\begin{proof}
Using $\la \sigma,\sigma \ra'=0$, we obtain
\begin{equation*}
  (\| \sigma \|_t^2 + \| \nabla^E\sigma \|_t^2 )'
  = \int_N \big(\la \nabla^E\sigma,\nabla^E\sigma \ra'
    + (\la \sigma,\sigma \ra + \la \nabla^E\sigma,\nabla^E\sigma \ra) \kappa \big) j .
\end{equation*}
By \lref{c1reg},
\begin{equation*}
  \la \nabla^E\sigma,\nabla^E\sigma \ra'
  = 2\sum_{2\le i\le m} \big( \la R^E(T,X_i)\sigma, \nabla^E_{X_i}\sigma \ra
   - \la \nabla^E_{WX_i}\sigma, \nabla^E_{X_i}\sigma \ra \big) ,
\end{equation*}
where $(T,X_1,\ldots,X_n)$ is a local orthonormal frame of $M$.
Hence
\begin{align*}
   |(\| \sigma \|_t^2 + \| \nabla^E\sigma \|_t^2 )'|
  &\le mC_R^E \| \sigma \|_t^2 + C_R^E \| \nabla^E\sigma \|_t^2 \\
  &\qquad + 2C_W \| \nabla^E\sigma \|_t^2
  + C_\kappa (\| \sigma \|_t^2 + \| \nabla^E\sigma \|_t^2) \\
  &\le (C_\kappa + mC_R^E + 2C_W) ( \| \sigma \|_t^2 + \| \nabla^E\sigma \|_t^2 ) .
  \qedhere
\end{align*}
\end{proof}

Along the cross sections $N_t$,
we change Clifford multiplication and connection of $E$
according to \eqref{tast} and \eqref{nast}.
Denote by $\nabla^t$ the new connection
and by $D_t$ the associated Dirac operator as in \eqref{dast}.
We note that Clifford multiplication with $T$ is $\nabla^t$-parallel.
For convenience, we will not keep the $\ast$ notation,
but will write $TX\sigma$ instead of $X\ast\sigma$.
With this in mind,
the Dirac operators $D$ and $D_t$ are related by
\begin{equation}
  D = T(\nabla^E_T + \sum TX_i\nabla^E_{X_i})
  = T\big(\big(\nabla^E_T + \frac{\kappa}2\big) - D_t\big) ,
  \label{sepvar}
\end{equation}
where $(T,X_2,\ldots,X_m)$ is a local orthonormal frame of $M$.

\begin{lem}\label{ddt}
For any $C^1$ section $\sigma$ of $E$ over $U$
which is parallel in the $T$-direction,
\begin{equation*}
  \| \nabla^E\sigma|_{N_t} - \nabla^t\sigma \|^2
  = \frac14 \|W\|^2 |\sigma|^2
  \quad\text{and}\quad
  | TD\sigma - D_t\sigma |^2 = \frac14\kappa^2 |\sigma|^2 .
\end{equation*}
\end{lem}

\begin{proof}
The second assertion is immediate from \eqref{sepvar}.
As for the first, let $(T,X_2,\ldots,X_m)$ be an orthonormal frame of $M$.
Then
\begin{align*}
  4 \| \nabla^E\sigma|_{N_t} - \nabla^t\sigma \|^2
  &= 4 \sum \la \nabla^E_{X_i}\sigma - \nabla^t_{X_i}\sigma,
  \nabla^E_{X_i}\sigma - \nabla^t_{X_i}\sigma \ra \\
  &= \sum \la TWX_i\sigma,TWX_i\sigma \ra \\
  &= \sum |WX_i |^2 |\sigma|^2 = \|W\|^2 |\sigma|^2 .
  \qedhere
\end{align*}
\end{proof}

Since the cross sections $N_t$ are $C^2$ submanifolds of $U$,
the restrictions of $E$ to them are $C^2$ bundles.
However, because of the term involving $W=\nabla T$,
the connection $\nabla^t$ is, in the generality we strive for, only continuous.
If $\nabla^t$ were a $C^1$ connection,
we would get \eqref{cast} for its curvature, now denoted $R^t$.
The right hand side of \eqref{cast} makes sense in the case
where $W$ is only continuous,
so that we may consider it as defining $R^t$.
Approximating $N_t$ by smooth submanifolds
and $C^1$ sections by smooth sections, \eqref{lic3} implies that
\begin{equation}
  (\nabla^t\sigma_1,\nabla^t\sigma_2)_{t} + (K^t\sigma_1,\sigma_2)_{t}
  = (D_t\sigma_1,D_t\sigma_2)_{t}
  \label{licht}
\end{equation}
for all $C^1$ sections $\sigma$ and $\tau$ of the restriction of $E$ to $N_t$,
where the curvature term in the Lichnerowicz formula as in \eqref{lic3}
is now denoted by $K^t$.
We recall from \eqref{lic4} that $K^t$ is uniformly bounded.

We extend our correspondence $\sigma\leftrightarrow P^{/\!\!/}\sigma$
as in \eqref{sigmat} and \eqref{sigmatt}:
Since $T$ is parallel in the $T$-direction,
Clifford multiplication by $T$ along $N$ satisfies
\begin{equation}
  TP^{/\!\!/}\sigma = P^{/\!\!/}T\sigma
  \quad\text{and}\quad
  \nabla_TP^{/\!\!/}\sigma = P^{/\!\!/}\sigma' ,
  \label{corrtg}
\end{equation}
for any time dependent section $\sigma$ of $E$ over $N$.
Finally, we define $A_t$ to be the differential operator
on sections of $E$ over $N$ which corresponds to the operator $-D_t$,
\begin{equation}
  P_t^{/\!\!/}(A_t\sigma) = - D_tP_t^{/\!\!/}\sigma .
  \label{defat}
\end{equation}
In this notation, $D$ corresponds to the operator
\begin{equation}
  T(\partial + A) , \quad\text{where $\partial:=\frac{d}{dt} + \frac{\kappa}{2}$} ,
  \label{corrd}
\end{equation}
where $\kappa/2$ takes the role of $\Gamma$ in \eqref{dsGamm} and \eqref{dspart}.
Thus we have associated the Dirac system
\begin{equation}
  \mathcal D := (\mathcal H,\mathcal A,T)
  \label{corrd2}
\end{equation}
to the distance function $f$ on and the Dirac bundle $E$ over $U$,
where we recall from \eqref{odejac} that $\kappa$
(which occurs in the definition of $\partial$) is defined by these data.
We will now proceed with discussing the requirement for Dirac systems
as formulated in Section \ref{preds}.
We already observed that \lref{ax1} settles \eqref{dsax1}.
Furthermore, Clifford multiplication by $T$ satisfies the requirements
\eqref{dsax3}--\eqref{dsax3c},
by \eqref{corrtg} and since Clifford multiplication by $T$ is $\nabla^t$-parallel.

It follows from \eqref{licht} that,
on sections of the restriction of $E$ to $N_t$,
the graph norm of $D_t$ is equivalent to the $H^1$ norm.
In particular,
$D_t$ is self-adjoint with domain $H^1(N_t,E)$ in $L^2(N_t,E)$.
Moreover,
since the inclusion $H^1(N_t,E)\hookrightarrow L^2(N_t,E)$ is compact,
the spectrum of $D_t$ consists of eigenvalues with finite multiplicities.
We also observe that, for any section $\sigma$ of $E$ over $N$,
$P\sigma|_{N_t}\in H^1(N_t,E)$
if and only if $\sigma\in H^1(N,E)$, by \lref{ax2a}.
Thus the operators $A_t$ are all self-adjoint with the same domain,
$H_A:=H^1(N,E)$, in $H=L^2(N,E)$,
and the embedding $H_A\to H$ is compact
with respect to the graph norm of any of the operators $A_t$.
This settles the first part of the requirements for the $A_t$
in Section \ref{preds}.

\begin{lem}\label{aprim}
For any $C^1$ section $\sigma$ of $E$ over $U$
which is parallel in the $T$-direction,
\begin{equation*}
  D_t' \sigma
  =  \sum_{2\le i\le m}
  TX_i \{ R^E(T,X_i)\sigma - \nabla^E_{WX_i}\sigma \} + \frac{\kappa'}2\sigma ,
\end{equation*}
where $(T,X_2,\ldots,X_m)$ is a local orthonormal frame of $M$.
\end{lem}

\begin{proof}
By \lref{c1reg},
\begin{align*}
  D_t'\sigma
  &= \sum_{2\le i\le m}
  TX_i\nabla^E_T\nabla^E_{X_i}\sigma + \frac{\kappa'}2\sigma \\
  &= \sum_{2\le i\le m}
  TX_i \{ R^E(T,X_i)\sigma - \nabla^E_{WX_i}\sigma \} + \frac{\kappa'}2\sigma .
  \qedhere
\end{align*}
\end{proof}

\begin{cor}\label{aprim2}
For any $C^1$ section $\sigma$ of $E$ over $U$,
which is parallel in the $T$-direction,
\begin{equation*}
  \| D_t'\sigma \|_t
  \le C_1\| \sigma \|_t + C_w\| \nabla^E\sigma \|_t
  \le C_2\| \sigma \|_t + C_w\| D_t\sigma \|_t ,
\end{equation*}
where $C_1=mC_R^E+C_\kappa'$ and
$C_2=mC_R^E + C_\kappa' + C_w^2 + C_wC_K^{1/2}$.
\end{cor}

\begin{proof}
By Lemmas \ref{aprim} and \ref{ddt},
we have, at any point $p$ of $N_t$,
\begin{align*}
  |D_t'\sigma|
  &\le (mC_R^E+\frac12C_\kappa')|\sigma|
  + \sum |\kappa_i| |\nabla^E_{X_i}\sigma| \\
  &\le (mC_R^E+\frac12C_\kappa')|\sigma|
  + \|W\| \|\nabla^E\sigma\| \\
  &\le (mC_R^E+\frac12C_\kappa')|\sigma| + C_w \|\nabla^E\sigma\| \\
  &\le (mC_R^E+\frac12C_\kappa' + \frac12C_w^2)|\sigma|
  + C_w \|\nabla^t\sigma\| ,
\end{align*}
where $(T,X_2,\ldots,X_m)$ is an orthonormal frame at $p$ such that the $X_i$
are eigenvectors of $W$ with corresponding eigenvalues $\kappa_i$.
By \eqref{licht},
\begin{equation*}
  \|\nabla^t\sigma\|_t^2 \le \|D_t\sigma\|_t^2 + C_K \| \sigma \|_t^2 .
  \qedhere
\end{equation*}
\end{proof}

\begin{lem}\label{ax2}
For all $s<t$ in $I$ and $C^1$ sections $\sigma_1,\sigma_2\in H$ of $E$,
\begin{equation*}
   |(A_t\sigma_1,\sigma_2)_{t} - (A_s\sigma_1,\sigma_2)_{s}|
   \le C (e^{C_0(t-s)/2} - 1) (\|\sigma_1\|_s + \|A_s\sigma_1\|_{s}) \|\sigma_2\|_{s} ,
\end{equation*}
where $C=C(C_R,C_R^E,C_W,m)$.
\end{lem}

\begin{proof}
Extend $\sigma_1$ and $\sigma_2$
by parallel translation along the $T$-geodesics.
Then $D_t$ corresponds to $-A_t$, and we get
\begin{align}
   | (D_t\sigma_1,\sigma_2 )_{t} - ( D_s\sigma_1,\sigma_2)_{s}|
   &\le \big| \int_s^t\int_N \big( \la D_r\sigma_1,\sigma_2\ra j \big)' \big|
   \notag \\
   &\le \int_s^t\int_N \big| \la D_r'\sigma_1,\sigma_2\ra + \la D_r\sigma_1,\sigma_2\ra \kappa \big| j
   \notag \\
   &\le \int_s^t\int_N \big( \|D_r'\sigma_1\| + C_\kappa \|D_r\sigma_1\| \big) \|\sigma_2\| j .
   \label{estint}
\end{align}
By \cref{aprim2} and \lref{ax2a},
the first term on the right hand side of \eqref{estint} can be estimated by
\begin{align*}
   \int_s^t\int_N \|D_r'\sigma_1\| &\|\sigma_2\| j
   \le 2(C_1+C_w) \int_s^t
   ( \|\sigma_1\|_r^2 + \|\nabla^E\sigma_1\|_r^2 )^{1/2} \|\sigma_2\|_r \\
   &\le  2(C_1+C_w) \int_s^t  e^{C_0(r-s)/2}
   ( \|\sigma_1\|_s^2 + \|\nabla^E\sigma_1\|_s^2 )^{1/2} \|\sigma_2\|_s \\
   &= 4\frac{C_1+C_w}{C_0} (e^{C_0(t-s)/2} - 1)
   ( \|\sigma_1\|_s^2 + \|\nabla^E\sigma_1\|_s^2 )^{1/2} \|\sigma_2\|_s .
\end{align*}
Concerning the second term on the right hand side of \eqref{estint},
namely the integral of $\|D_r\sigma_1\| \|\sigma_2\| j$,
we note that $\|D_r\sigma\|\le\sqrt{m-1}\|\nabla^r\sigma\|$.
Hence we can estimate this term in a similar way, using \lref{ddt}.
We arrive at an estimate
\begin{multline*}
   | (D_t\sigma_1,\sigma_2 )_{t} - ( D_s\sigma_1,\sigma_2)_{s}|
   \\ \le C'  (e^{C_0(t-s)/2} - 1)
   ( \|\sigma_1\|_s^2 + \|\nabla^E\sigma_1\|_s^2 )^{1/2} \|\sigma_2\|_s ,
\end{multline*}
where $C'=C'(C_R,C_R^E,C_W,m)$.
Finally, the Bochner formula \eqref{licht} and the ensuing lines show that
\begin{align*}
   \|\sigma_1\|_s^2 + \|\nabla^E\sigma_1\|_s^2
   &\le C(C_R^E,m) ( \|\sigma_1\|_s + \|D_s\sigma_1\|_{s} ) \\
   &= C(C_R^E,m) ( \|\sigma_1\|_s + \|A_s\sigma_1\|_{s} ) .
   \qedhere
\end{align*}
\end{proof}

Since $C^1$ sections are dense in $H^1(N,E)$,
\lref{ax2} confirms the remaining requirements for the operators $A_t$
in Section \ref{preds}.
Thus the system $\mathcal D=(\mathcal H,\mathcal A,T)$ over $I$
from \eqref{corrd2} is a Dirac system in the sense of Section \ref{preds}
and, therefore, in the sense of Section 2.1 in \cite{disy}.

\section{Boundary Values and Fredholm Properties}
\label{secbv}

Let $\mathcal D=(\mathcal H,\mathcal A,T)$ be a Dirac system over
\begin{equation}
  I = \R_+ := [0,\infty) .
  \label{rplus}
\end{equation}
with origin $t_0=0$,
where we note that an analogous discussion holds true for other intervals
with non-empty boundary.
By \eqref{dspain}, the restriction $D_{0,c}$ of the Dirac operator $D$ to
\begin{equation}
  \mathcal L_{0,c}(\mathcal D) :=
  \{ \sigma \in \mathcal L_{c}(\mathcal D) : \sigma(0) = 0 \}
  \label{bv0c}
\end{equation}
is symmetric.
The adjoint operator of $D_{0,c}$ with respect to
$L^2(\mathcal D)\supseteq\mathcal L_{0,c}(\mathcal D)$
is called the {\em maximal extension} of $D$ on $\mathcal L_{c}(\mathcal D)$.
We denote it by $D_{\max}$ and let $\dom D_{\max}$ be the domain of $D_{\max}$,
endowed with the graph norm of $D_{\max}$.
The adjoint operator $D_{\min}$ of $D_{\max}$ is equal to the closure of $D$
on $\mathcal L_{c}(\mathcal D)$.
It is called the {\em minimal extension} of $D$,
and its domain is denoted by $\dom D_{\min}$.
We also let $H^1(\mathcal D)$ be the completion of $\mathcal L_{c}(\mathcal D)$
with respect to the norm
\begin{equation}
  \|\sigma\|_{H^1(\mathcal D)}^2
  := \|\sigma\|_{L^2(\mathcal D)}^2 + \|\partial\sigma\|_{L^2(\mathcal D)}^2
  + \|A\sigma\|_{L^2(\mathcal D)}^2 .
  \label{bvh1}
\end{equation}
Obviously,
\begin{equation}
  \mathcal L_{c}(\mathcal D)
  \subseteq H^1(\mathcal D)
  \subseteq \dom D_{\max}
  \subseteq L^2(\mathcal D) .
  \label{bvincl}
\end{equation}
To formulate the main results on $\dom D_{\max}$ from \cite{disy},
we need to discuss boundary values of sections at $t=0$.
As for proofs of the corresponding assertions,
we refer to the discussion in Chapters 1 and 2 of \cite{disy}
and, in particular, to Proposition 2.30 loc.cit.

\subsection{Boundary Values}
\label{susebv}
Recall the convention $H=H_0$.
Recall also that $A_0$ is self-adjoint in $H$ with domain $H_A$.
It will be convenient, in this section, to denote elements of $H$
by letters $x,y$ and to call them vectors.
Fix an orthonormal basis $(x_i)$ of $H$ which consists
of eigenvectors of $A_0$, $A_0x_i=\lambda_ix_i$.

For $s\ge0$,
let $H^s=H^s(A_0)\subseteq H=H_0$ be the domain of $|A_0|^s$.
Then $H^0=H$, $H^1=H_A$,
and $H^\infty=H^\infty(A_0):=\cap_{s\ge0}H^s$ is a dense subspace of $H$.
For $s\in\R$, define an inner product $\la.,.\ra_s$ on $H^\infty$,
\begin{equation}
  \la x,y \ra_s
  := ( (I+A_0^2)^{s/2}x , (I+A_0^2)^{s/2}y ) .
  \label{bvips}
\end{equation}
For $s\ge0$, the norm $\|.\|_s$ associated to $\la.,.\ra_s$
is equivalent to the graph norm of $|A_0|^s$,
and $H^s$ is equivalent to the completion of $H^\infty$ with respect to $\|.\|_s$.
For $s<0$,
define $H^s=H^s(A_0)$ to be the completion of $H^\infty$
with respect to $\|.\|_s$
and set $H^{-\infty}=H^{-\infty}(A_0):=\cup_{s\in\R}H^s$.
In terms of the above basis $(x_i)$ of eigenvectors,
$H^s$ consists of all linear combinations $x=\sum\xi_ix_i$ with
\begin{equation}
  \sum (1+\lambda_i^2)^{s}|\xi_i|^2 < \infty .
  \label{bvhs}
\end{equation}
The pairing
\begin{equation}
  B_s: H^s \times H^{-s} \to \C , \quad
  B_s(x,y) := ( (I+A_0^2)^{s/2}x , (I+A_0^2)^{-s/2}y ) ,
  \label{bvdp}
\end{equation}
is perfect, that is, identifies $H^{-s}$ with the dual space of $H^s$.

For a subset $J\subset\R$, let $Q_J=Q_J(A_0)$
be the corresponding spectral projection of $A_0$ in the spaces $H^s$.
The image of $H^s$ under $Q_J$ is
\begin{equation}
  H^s_J =H^s_J(A_0)
  := \{ x=\sum \xi_ix_i \in H^s : \text{$\xi_i=0$ if $\lambda_i\notin J$} \} .
  \label{bvsd}
\end{equation}
For $x\in H^s$, we also let $x_J:=Q_Jx$ .
For any bounded subset $J$ of $\R$, we have $H^s_J\subseteq H^\infty$.
Since $T=T_0$ anti-commutes with $A_0$,
\begin{equation}
  TQ_J = Q_{-J} T
  \quad\text{and}\quad
  T H^s_J = H^s_{-J}
  \label{bvgam}
\end{equation}
As shorthand, we use, for $a\in\R$,
\begin{alignat}{2}
  Q_{>a} &:= Q_{(a,\infty)} , &\quad Q_{\ge a} &:= Q_{[a,\infty)} , \\
  Q_{<a} &:= Q_{(-\infty,a)} , &\quad Q_{\le a} &:= Q_{(-\infty,a]} ,
\end{alignat}
and similarly for the spaces $H^s_J=Q_J(H^s)$.
We also need to introduce the hybrid Sobolev space
\begin{equation}
   \check H = \check H(A_0)
   := H^{1/2}_{\le0} \oplus H^{-1/2}_{>0} .
   \label{checkh}
\end{equation}
Since $H_J\subseteq H^\infty$, for any bounded $J\subseteq\R$,
\begin{equation}
  \check H = H^{1/2}_{\le a} \oplus H^{-1/2}_{>a}
  = H^{1/2}_{<a} \oplus H^{-1/2}_{\ge a} ,
  \label{checkha}
\end{equation}
for any $a\in\R$.
By \eqref{bvdp} and \eqref{bvgam},
\begin{equation}
  \omega(x,y) := B_{1/2}(x_{\le-a},Ty_{\ge a})
  + B_{-1/2}(x_{>-a},Ty_{<a})
\end{equation}
is well defined for $x,y\in\check H$
and independent of the choice of $a$.
We note that $\omega$ is continuous, non-degenerate,
and skew-Hermitian on $\check H$.

\begin{prop}\label{dmax}
The maximal domain $\dom D_{\max}$ satisfies:
\begin{enumerate}
\item
$\mathcal L_c(\mathcal D)$ is dense in $\dom D_{\max}$.
\item
Evaluation at $t=0$ on $\mathcal L_c(\mathcal D)$
induces a continuous surjection
\begin{equation*}
  \mathcal R_{\max}: \dom D_{\max} \to \check H , \quad
  \mathcal R_{\max}(\sigma) =: \sigma(0) .
\end{equation*}
\item
$\sigma\in\dom D_{\max}$ is in $H^1_{\loc}(\mathcal D)$
iff $\sigma(0)\in H^{1/2}$.
\item
$\sigma\in\dom D_{\max}$ is in $\dom D_{\min}$ iff $\sigma(0)=0$.
\item
For all $\sigma_1,\sigma_2\in\dom D_{\max}$
\begin{equation*}
  (D_{\max}\sigma_1,\sigma_2)_{L^2(\mathcal D)}
  - (\sigma_1,D_{\max}\sigma_2)_{L^2(\mathcal D)}
  = \omega(\sigma_1(0),\sigma_2(0)) .
\end{equation*}
\end{enumerate}
\end{prop}

Closed extensions of $D$ between $D_{\min}$ and $D_{\max}$
correspond precisely to closed linear subspaces $B$ of $\check H$,
called {\em boundary conditions}.
For any such boundary condition $B$,
the domain of the corresponding extension $D_{B,\max}$ is given by
\begin{equation}
  \dom D_{B,\max} = \{ \sigma\in\dom D_{\max} : \sigma(0) \in B \} .
  \label{dmax2}
\end{equation}
We are also interested in the operator $D_B$ with domain
\begin{equation}
  \dom D_B = \dom D_{B,\max} \cap H^1_{\loc}(\mathcal D) .
  \label{dh1}
\end{equation}
A boundary condition $B\subseteq\check H$ is called {\em regular}
if $D_B=D_{B,\max}$.
By \pref{dmax}, $\sigma\in\dom D_{\max}$ is in $\dom D_B$
if and only if $\sigma(0)$ belongs to $B\cap H^{1/2}$.
In particular, $B$ is a regular boundary condition
if $B$ is a closed subspace of $\check H$
that is contained in $H^{1/2}\subseteq\check H$.

Let $B\subseteq\check H$ be a boundary condition.
Since $\omega$ is non-degenerate,
the adjoint operator of $D_{B,\max}$ is given by $D_{B^{a},\max}$,
where
\begin{equation}
  B^a = \{ x\in\check H : \text{$\omega(x,y) = 0$ for all $y\in B$} \} ,
  \label{abc}
\end{equation}
by \pref{dmax}.
We say that a boundary condition $B$ is {\em elliptic}
if $B$ and $B^a$ are regular.
Typical examples of elliptic boundary conditions are
the Atiyah-Patodi-Singer boundary condition $B_{\rm APS}=H^{1/2}_{<0}$
and the more general $B=H^{1/2}_{<a}$ and $B=H^{1/2}_{\le a}$.
The adjoint boundary conditions for the latter are given by
$B=H^{1/2}_{\le-a}$ and $B=H^{1/2}_{<-a}$, respectively.
The maximal operators corresponding to the boundary conditions
$B=H^{1/2}_{<a}$ and $B=H^{1/2}_{\le a}$
will be denoted by $D_{<a,\max}$ and $D_{\le a,\max}$, respectively,
and similarly in other cases.
By ellipticity, we actually have $D_{<a,\max}=D_{<a}$
and $D_{\le a,\max}=D_{\le a}$.

As for boundary conditions in the super-symmetric case,
\begin{equation}
  H = H^+ \oplus H^- ,
  \label{susy6}
\end{equation}
we may choose orthonormal bases $x_i^\pm$ of $H^\pm$
consisting of eigenvectors of $A_0^\pm$.
By \eqref{susyac}, we may actually choose $x_i^-=T_0x_i^+T_0^{-1}$.
We get
\begin{equation}
  H^s = H^{s+}\oplus H^{s-}
  \quad\text{and}\quad
  \check H = \check H^+ \oplus \check H^- ,
  \label{susy7}
\end{equation}
where
\begin{equation}
  H^{s+} = H^s(A_0^+) , \quad H^{s-}=H^s(A_0^-) , \quad
  \check H^+ = \check H(A_0^+) ,
  \label{susy8}
\end{equation}
and
\begin{equation}
\begin{split}
  \check H^- &= \check H(A_0^-) = T_0 \check H(-A_0^+) T_0^{-1} \\
  &\simeq \hat H^+(A_0^+)
  = H^{-1/2}_{\le0} \oplus H^{1/2}_{>0} .
  \label{susy9}
\end{split}
\end{equation}
Furthermore, $\check H^+$ and $\check H^-$ are Lagrangian
with respect to the non-degenerate skew-Hermitian form $\omega$.

In the super-symmetric case,
we consider super-symmetric boundary conditions $B\subseteq\check H$,
that is,
\begin{equation}
  B = B^+\oplus B^- ,
  \label{susybc}
\end{equation}
where $B^\pm=B\cap\check H^\pm$.
Then the adjoint boundary condition is super-symmetric as well.
Moreover, a super-symmetric boundary condition $B$ is regular or elliptic
if and only if $B^+$ and $B^-$ are regular or elliptic in $\check H^+$
and $\check H^-$, respectively.
For example,
$B=H^{1/2}_{<a}$ and $B=H^{1/2}_{\le a}$ are elliptic super-symmetric
boundary conditions.
The maximal operators  corresponding to these will be denoted
by $D^\pm_{<a,\max}$ and $D^\pm_{\le a,\max}$, respectively,
and similarly in other cases.

\subsection{More Function Spaces}
\label{subfun}
For convenience, we assume from now on that $\mathcal D$
is the Dirac system associated to a Dirac bundle $E$
over a straight end $U$ of $M$ with distance function $f$
and $C^1$ diffeomorphism $F:(-r,\infty)\times N\to U$
as in Definition \ref{straight}.
We also recall the notation $U_0=f^{-1}([0,\infty))$.

Let $H^1(U_0,E)$ be the space of sections $\sigma$ in $L^2(U_0,E)$
with square integrable weak derivative,
$\nabla^E\sigma\in L^2(U_0,T^*M\otimes E)$; that is, we have
\begin{equation}
  (\nabla^E\sigma,\tau)_{L^2(U_0,E)}
  = (\sigma,(\nabla^E)^*\tau)_{L^2(U_0,E)} ,
\end{equation}
for all $\tau\in C^\infty_{cc}(U_0,T^*M\otimes E)$,
where $(\nabla^E)^*$ is the formal adjoint of the operator $\nabla^E$.
Recall that $H^1(U_0,E)$ is a Hilbert space with respect to the norm
defined by the inner product
\begin{equation}
  ( \sigma,\tau )_{H^1(U_0,E)}
  = ( \sigma,\tau )_{L^2(U_0,E)}
  + ( \nabla^E\sigma,\nabla^E\tau )_{L^2(U_0,T^*M\otimes E)} .
  \label{h1no}
\end{equation}
There is the corresponding space $H^1(U,E)$,
and $C^\infty_c(U,E)$ is dense in $H^1(U,E)$.
Any section in $H^1(U_0,E)$ is the restriction of some section
from $H^1(U,E)$;
see Theorem 11.12 in \cite{Ag} or Theorem 7.25 in \cite{GT},
noting that the problem is local
and that $H^1_{\loc}$ is invariant under $C^1$ diffeomorphisms.
It follows that the space $C^\infty_c(U_0,E)$ of restrictions of
sections in $C^\infty_{c}(U,E)$ to $U_0$ is dense in $H^1(U_0,E)$.
The trace map
\begin{equation}
  \mathcal R: H^1(U_0,E) \to H^{1/2}(N,E)
  \label{trace}
\end{equation}
is a well defined bounded operator;
see Theorem 3.10 in \cite{Ag} or Proposition 4.4.5 in \cite{Ta},
noting again that the problem is local
and that $H^1_{\loc}$ is invariant under $C^1$ diffeomorphisms.
The closure of $C^1_{cc}(U_0,E)$ in $H^1(U_0,E)$,
and therefore also of $C^\infty_{cc}(U_0,E)$ in $H^1(U_0,E)$, is
\begin{equation}
  H^1_0(U_0,E) := \{ \sigma \in H^1(U_0,E) : \mathcal R\sigma = 0 \} .
  \label{h10}
\end{equation}
As for partial integration,
\begin{equation}
  (\nabla^E\sigma,\tau)_{L^2(U_0,T^*M\otimes E)}
  = (\sigma,(\nabla^E)^*\tau)_{L^2(U_0,E)} - (\sigma,\tau(T))_{L^2(N,E)} ,
  \label{parina}
\end{equation}
for all $\sigma\in H^1(U_0,E)$ and $\tau\in H^1(U_0,T^*M\otimes E)$.
It follows that
\begin{equation}
  (D\sigma,\tau)_{L^2(U_0,E)}
  = (\sigma,D\tau)_{L^2(U_0,E)} + (\sigma,T\tau)_{L^2(N,E)} ,
  \label{parind}
\end{equation}
for all $\sigma,\tau\in H^1(U_0,E)$.
In particular,
any $\sigma\in H^1(U_0,E)$ belongs to the domain $\dom D_{\max}$
of the adjoint operator $D_{\max}$ of $D$,
the latter considered as an unbounded operator on $L^2(U_0,E)$
with domain $C^\infty_{cc}(U_0,E)$ or, equivalently, $H^1_0(U_0,E)$.

We switch now to the associated Dirac system $\mathcal D$
over $\R_+=[0,\infty)$.
With respect to the natural identifications,
\begin{equation}
  C^\infty_c(U_0,E) \subseteq \mathcal L_c(\mathcal D)
  \subseteq H^1(\mathcal D) = H^1(U_0,E) ,
  \label{morefs}
\end{equation}
where we use \eqref{licht} and \eqref{lic4} for the latter equality.
The convenience we had in mind further up refers to the density
of $C^\infty_c(U_0,E)$ in $H^1(\mathcal D)$.
Another convenience:
We often write $\|.\|_{I}$ for the $L^2$-norm of maps
defined on an interval $I$ (if meaningful).

The Jacobian determinants $j$ are not $C^1$ in general,
hence the commutator $[A,\partial]$ may not be well defined.
In \pref{proest} below and, in particular, in its proof,
we circumvent this problem by using
the commutator $A'$ of $d/dt$ and $A$ instead.

\begin{prop}\label{proest}
For all $w\in\R$ and $\sigma\in H^1(\mathcal D)$,
\begin{multline*}
  \| D\sigma - wT\sigma \|_{\R_+}^2
  =  \|\partial\sigma\|_{\R_+}^2 + \|(A-w)\sigma\|^2_{\R_+} \\
  - \Re (A'\sigma,\sigma)_{\R_+} - ( \sigma(0),(A_0-w)\sigma(0))_0 .
\end{multline*}
\end{prop}

\begin{rem}\label{rempm12}
As for the meaning of the last term on the right, we note that
the trace $\sigma(0)$ of $\sigma$ is in $H^{1/2}(A_0)$,
hence $A_0$ applied to it is in $H^{-1/2}(A_0)$,
and hence $(\sigma(0),(A_{0}-w)\sigma(0))_0$ is well defined.
\end{rem}

\begin{proof}[Proof of \pref{proest}]
Replacing $A$ by $A_w:=A-w$ and $D$ by $D_w=T(\partial+A_w)$
reduces the assertion to the case $w=0$.
Furthermore, by the density of $C^\infty_c(U_0,E)$ in $H^1(U_0,E)$,
we may assume that $\sigma$ is smooth with compact support.
We have
\begin{equation*}
   \|D\sigma\|_{\R_+}^2
   = \|\partial\sigma\|^2_{\R_+} + \|A\sigma\|^2_{\R_+}
    + 2\Re (\partial\sigma, A\sigma)_{\R_+}
\end{equation*}
By \eqref{dspain} and \eqref{l2proj},
\begin{align*}
  (\partial\sigma,A\sigma)_{\R_+}
  &= \int_N\int_{\R_+} \big(\langle\sigma,A\sigma\rangle j \big)' dtdx
    - (\sigma,\partial A\sigma)_{\R_+}    \\
  &= - ( \sigma(0), A_0\sigma(0))_{0}
  - (\sigma,\partial A\sigma)_{\R_+}.
\end{align*}
Since $(A\sigma)'=A'\sigma+A\sigma'$, we conclude that
\begin{align*}
  (\sigma,\partial A\sigma)_{\R_+}
  &= (\sigma,A'\sigma)_{\R_+} + (\sigma,A\sigma')_{\R_+}
  + (\frac{\kappa}2\sigma,A\sigma)_{\R_+} \\
  &= (\sigma,A'\sigma)_{\R_+} + (A\sigma,\partial\sigma)_{\R_+}
  + i\Im(\kappa\sigma,A\sigma)_{\R_+} .
  \qedhere
\end{align*}
\end{proof}

\begin{rem}\label{c2smooth}
It is possible to approximate a distance function $f$ as in Definition \ref{straight}
by smooth functions $f_\varepsilon$ such that
\begin{equation*}
  | f - f_\varepsilon | + | \nabla f - \nabla f_\varepsilon |
  + | \nabla^2f - \nabla^2f_\varepsilon | < \varepsilon
\end{equation*}
uniformly on $U_0$.
However, the separation of variables formula,
as we need it for our integration by parts formula in \pref{proest},
would be more involved than the one derived there
and would contain additional terms which would be difficult to control.
\end{rem}

\subsection{Fredholm Properties of $\mathcal D$}
\label{subfred}
We say that $\mathcal D$ is of {\em Fredholm type} if there is a constant $C>0$
such that
\begin{align}
   \|\sigma\|_{\R_+} &\le C \|D\sigma\|_{\R_+} ,
  \quad\forall \sigma \in \mathcal L_{0,c}(\mathcal D) ,
  \label{fredt}
\intertext{and that $\mathcal D$ is {\em non-parabolic} if, for each $t>0$,
there is a constant $C>0$ such that}
  \|\sigma\|_{[0,t]} &\le C \|D\sigma\|_{\R_+} ,
  \quad\forall \sigma \in \mathcal L_{0,c}(\mathcal D) .
  \label{nonpa}
\end{align}
Obviously, if $\mathcal D$ is of Fredholm type, then it is non-parabolic.

In \pref{nopanopa} below (and the paragraph preceding it),
we will make the connection to the property
{\em non-parabolic at infinity} of Dirac operators
as considered in Definition \ref{nonpai0} in the introduction.

In Lemma 2.38 of \cite{disy},
we showed that $\mathcal D$ is non-parabolic if and only if, for each $t>0$,
there is a constant $C>0$ such that
\begin{equation}
   \|\sigma\|_{[0,t]}
   \le C \big(\|\sigma(0)\|_{\check H}^2 + \|D\sigma\|_{\R_+}^2 \big)^{1/2}
   =: C \|\sigma\|_W ,
  \label{nonpa2}
\end{equation}
for all $\sigma \in \mathcal L_{c}(\mathcal D)$.
If $\mathcal D$ is non-parabolic,
we let $W\subseteq L^2_{\loc}(\mathcal D)$
be the completion of $\mathcal L_{c}(\mathcal D)$
with respect to the norm $\|.\|_W$.
As we will see below,
the space $W$ is well suited for issues concerning the closedness
of the image of $D$, compare \eqref{imperp} and the text preceding it.

Assume now that $\mathcal D$ is non-parabolic.
Then, since $\|.\|_W$ is weaker than the graph norm of $D$,
\pref{dmax}.1 implies that $\dom D_{\max}\subseteq W$.
Furthermore, equality holds if and only if $\mathcal D$ is of Fredholm type.
Moreover, if $\varphi:\R_+\to\C$ is Lipschitz continuous with compact support
and $\sigma\in W$, then $\varphi\sigma\in\dom D_{\max}$.
In particular, the trace $\mathcal R$ is well defined and continuous on $W$
and takes values in $\check H=\check H(A_0)$.

By the non-parabolicity of $\mathcal D$
and the definition of $W$, $D$ extends to a bounded operator
$D_{\ext}:W\to L^2(\mathcal D)$.
For a boundary condition $B\subseteq\check H$,
we define $D_{B,\ext}$ to be the operator in $W$
with target $L^2(\mathcal D)$ and domain
\begin{equation}
  \dom D_{B,\ext} = \{ \sigma \in W : \sigma(0) \in B \} .
  \label{dombext}
\end{equation}
Obviously, $D_{B,\ext}$ is closed and extends $D_{B,\max}$,
and $D_{B,\ext}=D_{B,\max}$ if and only if $\mathcal D$ is of Fredholm type.

In Theorem 2.43 of \cite{disy} we showed that,
for $\mathcal D$ non-parabolic and $B$ regular,
$D_{B,\ext}$ has finite dimensional kernel and closed image with
\begin{equation}
  (\im D_{B,\ext})^\perp = \ker D_{B^a,\max} .
  \label{imperp}
\end{equation}
Thus, if $\mathcal D$ is non-parabolic and $B$ is elliptic,
then $D_{B,\ext}$ is a Fredholm operator and the $L^2$-index
\begin{equation}
  \ind_{L^2} D_{B,\max} := \dim \ker D_{B,\max} - \dim \ker D_{B^a,\max}
\end{equation}
of $D_{B,\max}$ is well defined and finite.

\begin{prop}\label{nonpar0}
Assume that, for some $a\ge0$,
\begin{equation*}
  (A_t\sigma,A_t\sigma)_t \ge \Re(A_t'\sigma,\sigma)_t + a \|\sigma\|_t^2 ,
\end{equation*}
for all $t\ge0$ and $\sigma\in H_A$.
Then $\mathcal D$ is non-parabolic and $D_{<0,\ext}$ is an isomorphism.
Moreover, if $a>0$, then $\mathcal D$ is of Fredholm type.
\end{prop}

\begin{proof}
Recall the Hardy inequality,
\begin{equation}
  \int_{\R_+} |\phi'|^2 \ge \int_{\R_+} \frac{|\phi|^2}{4t^2} ,
  \label{hardy}
\end{equation}
for any $C^1$ function $\phi$ on $\R_+$ with $\phi(0)=0$.
By \pref{proest},
\begin{equation*}
  \|D\sigma\|_{\R_+}^2 \ge \|\partial\sigma\|_{\R_+}^2 + a \|\sigma\|_{\R_+}^2  ,
\end{equation*}
for all $\sigma\in H^1_0(U_0,E)$.
Applying \eqref{odejac} and \eqref{hardy}
we get
\begin{align*}
  \|\partial\sigma\|_{\R_+}^2
  &= \int_N \int_0^\infty \| ( j^{1/2}\sigma)' \|^2 dtdx \\
  &\ge \int_N \int_0^\infty \frac{1}{4t^2} \|\sigma\|^2 j dtdx
  = \int_0^\infty \frac{1}{4t^2} \|\sigma\|_t^2 dt .
\end{align*}
It follows that $\mathcal D$ is non-parabolic.
Clearly, if $a>0$, then $\mathcal D$ is of Fredholm type.

Using the density of $\mathcal L_c(\mathcal D)$ in $W$,
\pref{proest} together with the assumed inequality implies that
\begin{equation*}
  \|\partial\sigma\|_{\R_+}^2 - (\sigma(0),A_0\sigma(0))_0 \le \|D\sigma\|_{\R_+} ,
\end{equation*}
for any $\sigma\in W$.
Hence $D\sigma=0$ and $\sigma(0)\in\check H_{<0}$
implies that $\partial\sigma=0$ and $\sigma(0)=0$.
That is, $\sigma$ solves
\begin{equation}
  \sigma' = - \frac{\kappa}{2}\sigma ,
  \label{partial0}
\end{equation}
with $\sigma(0)=0$, hence $\sigma=0$,
and therefore $\ker D_{<0,\ext}$ is trivial.

Conversely, the cokernel of $D_{<0,\ext}$ is isomorphic to $\ker D_{\le0,\max}$,
by what we said further up.
Now the same argument as above shows that any $\sigma\in\ker D_{\le0,\max}$
with $\sigma(0)\in\check H_{\le0}$
satisfies $\partial\sigma=0$.
It follows that $\sigma$ solves \eqref{partial0}
and hence, by \eqref{odejac}, that
\begin{equation*}
  \sigma(t,x) = j^{-1/2}(t,x) \sigma(0,x) ,
\end{equation*}
for all $t\in\R_+$ and $x\in N$.
Since the $L^2$-norm of $\sigma$ is finite, we conclude that $\sigma=0$.
Hence $\coker D_{<0,\ext}$ is trivial as well.
\end{proof}

By \cref{aprim2},
\begin{equation}
  c_0 := \sup_{t\in\R_+,\sigma\in H_A\setminus\{0\}}
  \frac{\|A_t'\sigma\|_t}{\|\sigma\|_t + \|A_t\sigma\|_t}
  \le C(C_R,C_R^E,C_W,n) < \infty .
  \label{c0}
\end{equation}

\begin{cor}\label{nonpar1}
Suppose that $\spec A_t\cap(-\lambda,\lambda)=\emptyset$,
for all $t\in\R_+$, where
\begin{equation*}
  2\lambda \ge c_0 + \sqrt{4c_0 + c_0^2} .
\end{equation*}
Then $\mathcal D$ is non-parabolic and $D_{<0,\ext}$ is an isomorphism.
Moreover, if the inequality is strict, then $\mathcal D$ is of Fredholm type.
\end{cor}

\begin{proof}
Choose $a\ge0$ with
\begin{equation*}
  2\lambda \ge c_0 + \sqrt{4c_0 + c_0^2 + 4a} .
\end{equation*}
Then we have, for all $t\ge0$ and $\sigma\in H_A$,
\begin{align*}
  \|A_t\sigma\|_t^2 - \Re(A_t'\sigma,\sigma)_t
  &\ge  \| A_t\sigma \|_t^2 - c_0( \| A_t\sigma\|_t + \|\sigma\|_t ) \|\sigma\|_t \\
  &\ge  (\lambda^2 - c_0\lambda - c_0 ) \|\sigma\|_t^2
  \ge a \|\sigma\|_t^2  ,
\end{align*}
by the definition of $c_0$, and hence \pref{nonpar0} applies.
\end{proof}

The following estimate relates boundary conditions to Fredholm properties of $D$,
as we will see further on.

\begin{lem}\label{lemest}
For all $\sigma\in H^1_c(U_0,E)$ and $w\in\R$,
\begin{multline*}
  \|\partial\sigma\|_{\R_+}^2 + \frac12 \|(A-w)\sigma\|_{\R_+}^2 \\
  \le \|(D-wT)\sigma\|_{\R_+}^2
  +  c_1 \|\sigma\|_{\R_+}^2
  + (\sigma_0,(A_0-w)\sigma_0)_0 ,
\end{multline*}
where $2c_1=c_0(c_0 + 2 + 2|w|)$.
\end{lem}

\begin{proof}
By \pref{proest} and the definition of $c_0$,
\begin{align*}
  \|\partial\sigma\|_{\R_+}^2 &+ \|(A-w)\sigma\|_{\R_+}^2
  - \|(D-wT)\sigma\|_{\R_+}^2 - (\sigma_0,(A_0-w)\sigma_0)_0 \\
  &= \Re (A'\sigma,\sigma)_{\R_+} \\
  &\le c_0( \|A\sigma\|_{\R_+} + \|\sigma\|_{\R_+} ) \|\sigma\|_{\R_+} \\
  &\le c_0 \big( \|(A-w)\sigma\|_{\R_+} + (1+|w|)\|\sigma\|_{\R_+} \big) \|\sigma\|_{\R_+} \\
  &\le \frac12  \|(A-w)\sigma\|_{\R_+}^2 + c_0(\frac{c_0}2 + 1 + |w| ) \|\sigma\|_{\R_+}^2 .
  \qedhere
\end{align*}
\end{proof}

\begin{prop}\label{nonpar2}
Assume that there are $\Lambda>\lambda\ge0$ such that
\begin{equation}
  (\Lambda - \lambda)^2 > 4c_0(c_0 + 2 + \lambda + \Lambda)
  \quad\text{and}\quad
  \spec A_t \cap (\lambda,\Lambda) = \emptyset ,
\end{equation}
for all $t\in\R_+$.
Suppose $w\in(\lambda,\Lambda)$ satisfies
\begin{equation}
  |\lambda-w|^2, |\Lambda-w|^2 > 2c_1 = c_0(c_0+2+2w) .
\end{equation}
Then we have:
\begin{enumerate}
\item
If $\sigma\in e^{-wt}L^2(\mathcal D)$ solves $D\sigma=0$
in the sense of distributions with
$\sigma(0)\in \check H_{<\Lambda}$, then $\sigma=0$.
\item
If $\sigma\in e^{wt}L^2(\mathcal D)$ solves $D\sigma=0$
in the sense of distributions with
$\sigma(0)\in \check H_{<-\lambda}$, then $\sigma=0$.
\item
$\mathcal D$ is non-parabolic
and $D_{<-\lambda,\ext}$ is injective.
\end{enumerate}
\end{prop}

The first assumption in \pref{nonpar2} implies
that the set of $w$ in $(\lambda,\Lambda)$ satisfying the required inequalities
is non-empty.
We also recall from \eqref{dsax3b}
that the spectrum of $A_t$, $t\in\R$, is symmetric about $0$
so that the second assumption
implies that $\spec A_t$ has empty intersection with $-(\lambda,\Lambda)$ as well.
We get that $H^s_{\le\lambda}=H^s_{<\Lambda}$
and that $H^s_{<-\lambda}=H^s_{\le-\Lambda}$, for all $s\in\R$.

\begin{proof}[Proof of \pref{nonpar2}]
Let $\sigma\in H^1_c(U_0,E)$, $v\in\R$,
and set $\tau=e^{vt}\sigma$.
Then
\begin{align*}
  \| e^{vt}D\sigma \|_{\R_+}^2
  & + (\tau_0,(A_0 - v)\tau_0)_0 - \| \partial\tau \|_{\R_+}^2 \\
  &= \| (D - vT)\tau \|_{\R_+}^2
  + (\tau_0,(A_0 - v)\tau_0)_0 - \| \partial\tau \|_{\R_+}^2 \\
  &\ge  \frac12\| (A - v) \tau \|_{\R_+}^2
  - \frac{c_0}{2}(c_0 + 2 + 2|v|) \| \tau \|_{\R_+}^2  ,
\end{align*}
by \lref{lemest}.
Suppose now that $w\in(\lambda,\Lambda)$ satisfies the required inequalities,
and choose $\varepsilon>0$ such that
\begin{equation}
  |\Lambda-w|^2, |\lambda-w|^2
  \ge c_0 (c_0 +2 + 2w ) + 2\varepsilon .
  \label{varep}
\end{equation}
Then, with $v=\pm w$, we continue the above computation and get
\begin{equation}
  \| (D - vT)\tau \|_{\R_+}^2 + (\tau_0,(A_0-v)\tau_0)_0
  \ge \| \partial\tau \|_{\R_+}^2 + \varepsilon \|\tau \|_{\R_+}^2 .
  \label{inequaw}
\end{equation}
By the density of $H^1_c(U_0,E)$ in $\dom D_{\max}$,
any $\tau\in\dom D_{\max}$ satisfies
\begin{equation}
  \| (D - vT)\tau \|_{\R_+}^2 + (\tau_0,(A_0-v)\tau_0)_0
  \ge \varepsilon \|\tau \|_{\R_+}^2 ,
  \label{inequaw2}
\end{equation}
where $v=\pm w$ and $\varepsilon$ are as above.
Now with $\sigma$ as in the first two assertions and $v=w$ and $v=-w$,
respectively, $\tau=e^{vt}\sigma$ is in $\dom D_{\max}$
and satisfies $D_{\max}\tau=vT\tau$.
The boundary condition for $\sigma$ implies that the boundary term
in \eqref{inequaw2} is non-positive, hence $\tau=0$, and hence $\sigma=0$.
This shows the two first assertions.
As for the last assertion, we note that
\begin{equation}\label{inequaw3}
\begin{split}
  \| D\sigma \|_{\R_+}^2 &+ (\sigma_0,(A_0+w)\sigma_0)_0 \\
  &\ge
  \| e^{-wt}D\sigma \|_{\R_+}^2  + (\sigma_0,(A_0+w)\sigma_0)_0 \\
  &\ge \varepsilon \| e^{-wt}\sigma \|_{\R_+}^2 ,
\end{split}
\end{equation}
for any $\sigma\in H^1_{c}(U_0,E)$.
\end{proof}

For later purposes we note that the computations in the above proof
also show that
\begin{equation}\label{inequaw4}
  \| e^{wt}D\sigma \|_{\R_+}^2  + (\sigma_0,(A_0-w)\sigma_0)_0 \\
  \ge \varepsilon \| e^{wt}\sigma \|_{\R_+}^2 ,
\end{equation}
for any $\sigma\in H^1_{c}(U_0,E)$,
where $w\in(\lambda,\Lambda)$ and $\varepsilon$ is as in \eqref{varep}.

Suppose now that the assumptions of \pref{nonpar2} are satisfied
and that $w\in(\lambda,\Lambda)$ satisfies the corresponding inequalities.
Then \eqref{inequaw3} and \eqref{inequaw4} lead us to consider
the {\em weighted Lebesgue spaces}
$L^2_{\pm w}(\mathcal D):=e^{\mp wt}L^2(\mathcal D)$,
with norm associated to the inner product
\begin{equation}
  ( \sigma,\tau )_{\pm w} := (e^{\pm wt}\sigma,e^{\pm wt}\tau)_{\R_+} ,
\end{equation}
and the {\em weighted Sobolev spaces} $H^1_{w,<\mu}(\mathcal D)$,
the completions of $H^1_{<\mu,c}(\mathcal D)$ with respect to the norms
\begin{equation}
  \| \sigma \|_{H^1_{w}(\mathcal D)} := \|\sigma\|_w + \|D\sigma\|_w .
\end{equation}

\begin{cor}\label{cisowei}
If the assumptions of \pref{nonpar2} hold and $w\in(\lambda,\Lambda)$
satisfies the corresponding inequalities, then the operators
\begin{align*}
  &D_{w,<\Lambda}: H^1_{w,<\Lambda}(\mathcal D) \to L^2_w(\mathcal D)
  \quad\text{and} \\
  &D_{-w,<-\lambda}: H^1_{-w,<-\lambda}(\mathcal D) \to L^2_{-w}(\mathcal D)
\end{align*}
are adjoints of each other and isomorphisms.
\end{cor}

We note that $D_{w}$ on $L^2_{w}(\mathcal D)$
is conjugate to the operator $D-wT.$ on $L^2(\mathcal D)$,
and similarly for $D_{-w}$.
Hence the operators $D_{\pm w}$ are Dirac-Schr\"odinger operators
in the sense of \cite{disy} (compare also Remark 2.27 of loc.cit.).

\begin{proof}[Proof of \cref{cisowei}]
The operators are adjoints of each other since
$\check H_{<-\lambda} = \check H_{\le-\Lambda}$,
by the assumptions of \pref{nonpar2}.
By \eqref{inequaw3} and \eqref{inequaw4},
 the images of the operators are closed.
The first two assertions of \pref{nonpar2} say that their kernels are trivial.
By integration by parts as in (5) of \pref{dmax},
we see that $\sigma\in L^2_{w}(\mathcal D)$ is in the orthogonal complement
of $D(H^1_{w,<\Lambda}(\mathcal D))$ if $\tau:=e^{2wt}\sigma$
solves $D\tau=0$ weakly with $\tau(0)\in H_{\le-\Lambda}=H_{<-\lambda}$.
Now $\tau\in e^{wt}L^2(\mathcal D)$, hence $\tau=0$,
by the second assertion of \pref{nonpar2}.
This shows that the first operator is an isomorphism.
The claim for the second follows in a similar fashion,
using the first assertion of \pref{nonpar2}.
\end{proof}

\begin{cor}\label{cindext}
If the assumptions of \pref{nonpar2} hold, then
\begin{equation*}
  \ind D_{<0,\ext} = \dim H_{[-\lambda,0)}
  - \dim \ker D_{\le\lambda,\max} .
\end{equation*}
In the super-symmetric case,
\begin{equation*}
  \ind D^+_{<0,\ext} = \dim H^+_{[-\lambda,0)}
  - \dim \ker D^-_{\le\lambda,\max} .
\end{equation*}
\end{cor}

\begin{proof}
By Theorem 3.14 of \cite{disy}, we have
\begin{equation*}
  \ind D_{<0,\ext} = \ind D_{<-\lambda,\ext} + \dim H_{[-\lambda,0)} .
\end{equation*}
By \pref{nonpar2}, $D_{<-\lambda,\ext}$ is injective.
On the other hand, the orthogonal complement of  $\im D_{<-\lambda,\ext}$
is given by the space of $\sigma$ in $L^2(\mathcal D)$ with $D\sigma=0$
and $\sigma(0)\in H_{\le\lambda}$.
This shows the first claim, and the proof of the second is similar.
\end{proof}

\section{Decomposition and Index}
\label{subdec}

We assume from now on that we have a decomposition $M=M_0\cup U_0$,
where $M_0$ and $U_0$ are domains in $M$
such that $M_0$ is compact and connected.
We will need later that there is a distance function which is defined
on a neighborhood of $U_0$.
However, up to and including \tref{nonpar5},
we only assume that $N:=M_0\cap U_0\ne\emptyset$
is a level surface of a $C^2$ distance function $f$
which is defined in some open neighborhood of $N$ in $M$.\footnote{
This seems to be the right setup for general applications.}
We assume that $T:=\grad f$ points into the direction of $U_0$,
set $A_0:=-D_N$ as in \eqref{defat} and get the associated Sobolev
spaces $H^s=H^s(A_0)$ as in Section \ref{susebv}.

\begin{lem}\label{h1m0}
There is a constant $C>1$ such that
\begin{equation*}
  \| \sigma \|_{H^1(M_0,E)}
  \le C \big( \| \sigma|_N \|_{H^{1/2}} + \| D\sigma \|_{L^2(M_0,E)} \big) ,
\end{equation*}
for all $\sigma\in H^1(M_0,E)$.
\end{lem}

\begin{proof}
Let $\mathcal R_0:H^1(M_0,E)\to H^{1/2}$ be restriction to $N$,
$\mathcal R_0\sigma:=\sigma|_N$,
and $\mathcal E_0:H^{1/2}\to H^1(M_0,E)$ be an extension operator
(for extension operators, see e.g. (1.36) and Lemma 1.37 in \cite{disy}).
Since $\mathcal E_0$ and $\mathcal R_0$ are continuous
and $H^1_0(M_0,E)$ is the kernel of $\mathcal R_0$,
\begin{equation*}
  H^1(M_0,E) \to H^{1/2} \times H^1_0(M_0,E) , \quad \sigma \mapsto
  (\mathcal R_0\sigma,\sigma - \mathcal E_0\mathcal R_0\sigma) ,
\end{equation*}
is a continuous bijection and therefore an isomorphism of topological vector spaces.
This reduces the discussion to the case where $\sigma$ belongs to $H^1_0(M_0,E)$,
and then $\sigma|_N=0$.

If a constant as asserted would not exist,
there would be a sequence $(\sigma_n)$ in $H^1_0(M_0,E)$
such that $\|\sigma_n\|_{L^2(M_0,E)}=1$ and $\| D\sigma_n\|_{L^2(M_0,E)}\to0$.
Extending $\sigma_n$ by $0$ to $U_0$,
we obtain an $H^1$-section $\tau_n$ of $E$ over $M$
with support in the compact domain $M_0$
such that $\|\tau_n\|_{L^2(M_0,E)}=1$.
Using partial integration as in \pref{dmax}.5 (or, more precisely,
as in Proposition 5.7.3 in \cite{disy}), we get that $D\tau_n=D\sigma_n$ on $M_0$
and that $D\tau_n=0$ on $U_0$.
In particular, $\| D\tau_n\|_{L^2(M,E)}\to0$.
Therefore, up to passing to a subsequence, $(\tau_n)$ converges in $L^2(M,E)$.
Furthermore, the limit $\tau$ vanishes in $U_0$, is of $L^2$-norm $1$,
and solves $D\tau=0$ weakly.
Hence $\tau$ is smooth and $D\tau=0$, by elliptic regularity.
Since the interior of $U_0$ is non-empty,
this is in contradiction to the unique continuation property for Dirac operators,
see Theorem 8.2 in \cite{BW}.
\end{proof}

It will be convenient to write $D_{M_0}$ for the restriction of $D$ to $M_0$,
and similarly in corresponding cases.

Consider the manifold $\tilde M$ which is the disjoint union of $M_0$ and $U_0$,
endowed with the Dirac bundle $\tilde E\to\tilde M$ induced by $E$.
We want to apply the results from Chapter 5 of \cite{disy}
to the Dirac operator $\tilde D$ of $\tilde E$ and,
therefore, need to check whether the requirements of Axiom VI there are satisfied.
The only requirement in that axiom which might be non-obvious
is dealt with in the following lemma.

\begin{lem}\label{ax6}
Let $\chi:\R\to\R$ be a smooth function with compact support
which is equal to $1$ close to $0$.
Then $(1-\chi\circ f)\sigma\in\dom D_{U_0,\min}$
for all $\sigma\in\dom D_{U_0,\max}$,
and similarly for $M_0$.
\end{lem}

\begin{proof}
We note first that $(1-\chi\circ f)\sigma$ is a section in $\dom D_{U_0,\max}$
which vanishes in a neighborhood of the boundary $N$ of $U_0$.
Hence the extension $\tilde\sigma$ of $(1-\chi\circ f)\sigma$ by $0$ to $M_0$
is in $\dom D_{\max}$.
Now we have $\dom D_{\max}=\dom D_{\min}$,
by Theorem II.5.7 of \cite{LM}.
Hence there is a sequence of smooth sections $\sigma_k\in C^\infty_c(M,E)$
such that $\sigma_k\to\tilde\sigma$ in $\dom D_{\min}$.
It follows that $(1-\tilde\chi\circ f)\sigma_k\to(1-\chi\circ f)\sigma$
in $\dom D_{U_0,\min}$,
where $\tilde\chi:\R\to\R$ is a smooth function with compact support
such that $(1-\tilde\chi)(1-\chi)=(1-\chi)$.
\end{proof}

Because $T$ is the exterior normal to $M_0$,
the space of boundary values of the maximal extension $D_{M_0,\max}$
of $D$ over $M_0$ is the hybrid Sobolev space
\begin{equation}
  \hat H
  = H^{-1/2}_{<0}\oplus H^{1/2}_{\ge0}
  = \check H(-A_0) .
  \label{hath}
\end{equation}

\begin{lem}\label{indprep}
For any $\lambda\ge0$, we have
\begin{equation*}
  \ind D_{M_0,\ge-\lambda} = \frac12 \dim H_{[-\lambda,\lambda]} .
\end{equation*}
\end{lem}

\begin{proof}
Since $D_{M_0,\ge-\lambda}$ is the adjoint operator of $D_{M_0,>\lambda}$,
we have
\begin{equation*}
   \ind D_{M_0,\ge-\lambda} = - \ind D_{M_0,>\lambda} .
\end{equation*}
On the other hand,
\begin{equation*}
  \ind D_{M_0,\ge-\lambda} - \ind D_{M_0,>\lambda}
  = \dim H_{[-\lambda, \lambda]} ,
\end{equation*}
by Theorem 5.16 in \cite{disy}.
\end{proof}

The same argument applies to $D_{U_0,\le\lambda,\max}$
if $D$ is of Fredholm type.

Specifying the data in the definition of non-parabolicity of the third named author,
compare \cite{Ca1}, we say that $D$ is {\em non-parabolic}
with respect to some subset $L\subseteq M$ if,
for any relatively compact open subset $K\subseteq M$,
there exists  a constant $C=C(K,L)$ such that
\begin{equation}
  \| \sigma \|_{L^2(K,E)} \le C \| D\sigma \|_{L^2(M,E)} ,
  \label{nonpam}
\end{equation}
for any smooth section $\sigma$ of $E$ with compact support
such that $\sigma|_L=0$.
Obviously, if $D$ is of Fredholm type,
then $D$ is non-parabolic with respect to any sufficiently large compact subset,
and if $D$ is non-parabolic with respect to some subset,
then also with respect to any larger subset.
Furthermore, if $M$ is connected,
then $D$ is non-parabolic with respect to any subset
whose complement is relatively compact,
by \lref{h1m0}.
Recalling Definition \ref{nonpai0},
we see that $D$ is {\em non-parabolic at infinity}
if $D$ is non-parabolic with respect to some compact subset of $M$.

\begin{prop}\label{nopanopa}
Suppose that the ends of $M$ are straight in the sense of Definition \ref{straight}
and let $\mathcal D$ be the Dirac system over $\R_+$
associated to $D$ over $U_0$ as in Section \ref{subdisy}.
Then $D$ is non-parabolic with respect to $M_0$
if and only if $\mathcal D$ is non-parabolic in the sense of Section \ref{subfred}.
In particular, if $\mathcal D$ satisfies the assumptions of \pref{nonpar2},
then $D$ is non-parabolic with respect to $M_0$.
\qed
\end{prop}

Assume from now on that $D$ is non-parabolic with respect to $M_0$.
Let $W(M,E)$ and $W(U_0,E)$ be the completion of $H^1_{c}(M,E)$
and $H^1_c(U_0,E)$ with respect to the norms associated to the inner products
\begin{equation}\label{wmu}
\begin{split}
  (\sigma,\tau)_{W(M,E)}
  &:= (\sigma,\tau)_{H^1(M_0,E)} + (D\sigma,D\tau)_{L^2(U_0,E)} , \\
  (\sigma,\tau)_{W(U_0,E)}
  &:= (\sigma,\tau)_{\check H} + (D\sigma,D\tau)_{L^2(U_0,E)} ,
\end{split}
\end{equation}
respectively.
We have
\begin{equation}\label{wmu2}
\begin{split}
  W(M,E)
  &= \{ (\sigma,\tau) \in H^1(M_0,E)\oplus W(U_0,E) :
  \sigma|_N = \tau|_N \} \\
  &\subseteq H^1_{\loc}(M,E) ,
\end{split}
\end{equation}
since the transmission condition $\sigma|_N = \tau|_N$
is a regular boundary condition for the manifold $\tilde M$ as above,
see Example 1.85 in \cite{disy}.
By definition, $D$ induces continuous operators
\begin{equation}\label{dext}
\begin{split}
  D_{\ext}: W(M,E) &\to L^2(M,E) , \\
  D_{U_0,\ext}: W(U_0,E) &\to L^2(U_0,E) .
\end{split}
\end{equation}
We arrive at the following version of Th\'eor\`eme 0.3 of \cite{Ca2}.

\begin{thm}\label{dwfred}
Suppose that $D$ is non-parabolic with respect to $M_0$.
Then $D_{\ext}:W(M,E) \to L^2(M,E)$ is a Fredholm operator with
\begin{equation*}
  (\im D_{\ext})^\perp = \ker D_{\max}
  = \{ \sigma\in L^2(M,E) : \text{$D\sigma = 0$ weakly} \} .
\end{equation*}
\end{thm}

\begin{proof}
Theorem 5.12 in \cite{disy} implies that the image of $D_{\ext}$ is closed
and that $\ker D_{\ext}$ is of finite dimension.
The last claim follows from the density of $H^1_c(M,E)$ in $W(M,E)$
and since $D$ is formally self-adjoint.
Finally, since $\ker D_{\max} \subseteq \ker D_{\ext}$
and the latter is of finite dimension, $D_{\ext}$ is a Fredholm operator.
\end{proof}

In the super-symmetric case $E=E^+\oplus E^-$, we get operators
\begin{equation}
  D_{\ext}^\pm: W(M,E^\pm) \to L^2(M,E^\mp) .
  \label{dextp}
\end{equation}
Since $D_{\ext}$ is a Fredholm operator,
the operators $D_{\ext}^\pm$ are Fredholm operators as well and
\begin{equation}
  \ind D_{\ext}^+
  = \dim \ker D_{\ext}^+ - \dim \ker D_{\max}^- ,
  \label{inddexpl}
\end{equation}
by \tref{dwfred} (and since $D$ is formally self-adjoint).

The transmission condition $\sigma|_N=\tau|_N$ as above is elliptic.
Therefore it can be decoupled into separate boundary conditions
for $M_0$ and $U_0$, respectively,
compare Theorems 3.24 and 5.12 in \cite{disy}.
This leads to the following index formulas.

\begin{thm}\label{nonpar5}
Suppose that $D$ is non-parabolic with respect to $M_0$.
Then we have, for any $\lambda \ge0$,
\begin{equation*}
  \ind D_{\ext}
  = \frac12 \dim H_{[-\lambda, \lambda]} + \ind D_{U_0,<-\lambda,\ext} .
\end{equation*}
In the super-symmetric case,
\begin{equation*}
  \ind D^+_{\ext} =   \ind D^+_{M_0,\ge0}
  + \dim H^+_{[-\lambda,0)} + \ind D^+_{U_0,<-\lambda,\ext} .
\end{equation*}
\end{thm}

\begin{proof}
The assertions are immediate consequences
of Theorems 3.24, 4.17, and 5.12 in \cite{disy}
and \lref{indprep} above.
\end{proof}

Suppose now that the ends of $M$ are straight
in the sense of Definition \ref{straight}.
We may then consider weighted Lebesgue and Sobolev spaces,
following the discussion just before and in \cref{cisowei}.
For $w\in\R$, let $L^2_{w}(M,E)$ be the space of measurable sections of $E$
which are square integrable over $M$ with respect to the weight
which is equal to $1$ over $M_0$ and equal to $e^{2wt}$ over $U_0$.
Endow $L^2_w(M,E)$ with the corresponding inner product
\begin{equation}
  (\sigma,\tau)_{L^2_{w}(M,E)}
  := (\sigma,\tau)_{L^2(M_0,E)} + (e^{wt}\sigma,e^{wt}\tau)_{L^2(U_0,E)} .
  \label{normv0}
\end{equation}
Furthermore, let $H_{w}^1(M,E)$ be the completion of $H^1_c(M,E)$
with respect to the norm associated to the inner product
\begin{equation}
  (\sigma,\tau)_{H_{w}^1(M,E)}
  := (\sigma,\tau)_{L^2_w(M,E)} + (D\sigma,D\tau)_{L^2_w(M,E)} .
\end{equation}
Assume from now on that the assumptions of \pref{nonpar2} are satisfied
and that $w\in\R$ satisfies the corresponding inequalities.
Then, by \eqref{inequaw3}, \eqref{inequaw4}, and \lref{h1m0},
the $H_{\pm w}^1(M,E)$-norm is equivalent to the norm
\begin{equation}
  \| \sigma \|_{\pm w} := \| \sigma|_N \|_{H^{1/2}}
  + \|D\sigma\|_{L^2(M_0,E)} + \|e^{\pm wt}D\sigma\|_{L^2(U_0,E)} .
  \label{normeq}
\end{equation}
Thus, by restriction to $M_0$ and $U_0$, respectively,
$H^1_{w}(M,E)$ is isomorphic to the space
of pairs $(\sigma,\tau)$ in  $H^1(M_0,E)\oplus H^1_w(U_0,E)$
satisfying the transmission condition $\sigma|_{N}=\tau|_N$.

\begin{thm}\label{dw}
Suppose that the Dirac system $\mathcal D$ over $\R_+$
associated to $E$ over $U_0$ satisfies the assumptions of \pref{nonpar2}
and that $w>0$ satisfies the corresponding inequalities.
Then
\begin{equation*}
  D_{-w}: H_{-w}^1(M,E) \to L^2_{-w}(M,E)
\end{equation*}
is a Fredholm operator with index
\begin{equation*}
  \ind D_{-w} = \frac12 \dim H_{[-\lambda,\lambda]} .
\end{equation*}
In the super-symmetric case,
\begin{equation*}
  \ind D^+_{-w}
  =  \ind D^+_{M_0,\ge0} + \dim H^+_{[-\lambda,0)} .
\end{equation*}
\end{thm}

\begin{proof}
By \eqref{inequaw3}, $D_{-w}$ as above is a Fredholm operator.
We also note that $D_{U_0,-w}$
is conjugate to the operator $D_{U_0}+w\grad f$,
where $f$ is the given distance function over $U_0$.
Hence the results of Section 3 in \cite{disy} apply
(compare also Remark 2.27 of loc.cit.) and show
that the Calder\'on projections associated to $L^2_{-w}$-solutions
of the equation $D\sigma=0$ over $M_0$ and $U_0$ are elliptic.
Hence, by Theorems 3.24 and 5.12 in \cite{disy},
$D$ as above has index
\begin{equation*}
  \ind D_{-w} = \ind D_{M_0,\ge-\lambda} + \ind D_{U_0,-w,< -\lambda} .
\end{equation*}
By \cref{cisowei}, $\ind D_{U_0,-w,< -\lambda}=0$,
hence the formula for $\ind D_{-w}$ follows from \lref{indprep}.
In the super-symmetric case,
\begin{align*}
  \ind D^+_{-w}
  &= \ind D^+_{M_0,\ge-\lambda} + \ind D^+_{U_0,-w,< -\lambda} \\
  &=  \ind D^+_{M_0,\ge-\lambda}
  =  \ind D^+_{M_0,\ge0} + \dim H^+_{[-\lambda,0)} .
  \qedhere
\end{align*}
\end{proof}

\begin{cor}\label{dwcor}
Under the assumptions of \tref{dw}, we have
\begin{equation*}
  \ind D_{\ext}^+ = \ind D_{-w} - \dim\ker D^-_{U_0,\le\lambda,\max} .
\end{equation*}
\end{cor}

\begin{proof}
By Theorems \ref{nonpar5} and \ref{dw}, we have
\begin{equation}\label{dwcor1}
  \ind D_{\ext}^+ = \ind D_{-w} + \ind D^+_{U_0,<-\lambda,\ext} .
\end{equation}
By \pref{nonpar2}.3, $D^+_{U_0,<-\lambda,\ext}$ is injective.
Hence, by \eqref{imperp},
\begin{equation}\label{dwcor2}
   \ind D^+_{U_0,<-\lambda,\ext}
   = - \dim\ker D^+_{U_0,\le\lambda,\max}
   \qedhere
\end{equation}
\end{proof}

In the case where the boundary $N=N_0$ of $M_0$ is smooth,
Theorem 3.1 in Atiyah-Patodi-Singer \cite{APS} applies and gives
\begin{equation}\label{aps}
\begin{split}
  \ind D^+_{M_0,\ge0}
  &= \int_{M_0} \omega_{D^+} + \int_{N_0} \tau_{D^+} \\
  &\hspace{1cm} + \frac12 \big(\eta(A_0^+) + \dim \ker A_0^+\big) ,
\end{split}
\end{equation}
where $\omega_{D^+}$ is the index form
and $\tau_{D^+}$ the transgression form.
We remark that $\omega_{D^+}$ is a universal polynomial
in the curvatures of $M$ and $E$
and that $\tau_{D^+}$ is a universal polynomial
in the curvature of $M$ and $E$
and the second fundamental form of $N$;
compare \cite{Gi1} and Section 3.10 in \cite{Gi}.
Now we may approximate $M_0$ by smooth domains
such that the second fundamental forms of their boundaries
approximate the second fundamental form of $N$.
Then the integrals of $\omega_{D^+}$ and $\tau_{D^+}$
over the approximating domains and their boundaries
converge to the integral of the corresponding forms
over $M_0$ and $N_0$, respectively.
On the other hand, the coefficients of $A_0^+$ are only $C^1$ in general,
and therefore the $\eta$-invariant of $A_0^+$ may not be well defined.
However, since the other terms on the right hand side of \eqref{aps}
are well defined, we may define $\eta(A_0^+)$ to be the number
such that  \eqref{aps} holds.
In \cite{Hi},
Michel Hilsum defined $\eta$-invariants for Lipschitz manifolds in a similar way,
and he showed that they enjoy many of the properties
of \lq\lq smooth" $\eta$-invariants.
We do not pursue this issue any further
since we apply the APS-formula only in the smooth case.

Assuming now that the ends of $M$ are smooth,
we may combine the index formula for $D^+$ in Theorems \ref{nonpar5}
and \ref{dw} with \eqref{aps}.
To that end, we continue to assume
that the assumptions of \pref{nonpar2} are satisfied.
Then the spectrum of $A_t$ has two parts,
the part consisting of eigenvalues of modulus at most $\lambda$
and the part consisting of those of modulus at least $\Lambda$.
Following a corresponding convention in \cite{Lo2},
we call the first the {\em low energy} and the second the {\em high energy} part
and get the corresponding spectral projections and spaces,
\begin{alignat}{3}
  &P_t := Q_{[-\lambda,\lambda]}(A^+_t) , & \quad
  &H^{\rm le}_t := P_t(H_t) , &\quad
  &A^{\rm le}_t := A_t|_{H^{\rm le}_t} ,
    \label{lopar} \\
  &Q_t := I - P_t  , &
  &H^{\rm he}_t := Q_t(H_t) , &
  &A^{\rm he}_t := A_t|_{H^{\rm he}_t} ,
  \label{hipar}
\end{alignat}
where we note that $H_t = H^{\rm le}_t \oplus H^{\rm he}_t$
is an orthogonal decomposition which is invariant under $A_t$.
In the super-symmetric case we get similar decompositions and call
\begin{equation}
  \eta(A^{\rm le,+}_t)
  \quad\text{and}\quad
  \eta(A^{\rm he,+}_t) ,
  \label{lohieta}
\end{equation}
the {\em low} and {\em high energy $\eta$}-invariant of $A_t^+$, respectively.
We have
\begin{equation}
  \eta(A^+_t) = \eta(A^{\rm le,+}_t) + \eta(A^{\rm he,+}_t) .
  \label{lohieta2}
\end{equation}

\begin{cor}\label{cindwex}
Assume that the ends of $M$ are smooth and straight
and that the Dirac system over $\R_+$ associated to $E$
over $U_0$ satisfies the assumptions of \pref{nonpar2}.
Then we have, in the super-symmetric case,
\begin{equation*}
  \ind D^+_{-w}
  = \int_{M_0} \omega_{D^+} + \int_{N_0} \tau_{D^+}
  + \frac12 \left(\dim H^+_{[-\lambda,\lambda]} + \eta(A^{\rm he,+}_0) \right) .
  \qed
\end{equation*}
\end{cor}

Since $\ind D^+_{-w}$ does not change when replacing the parameter $t$
along the ends by $t-t_0$, for any $t_0>0$,
it follows that $\ind D^+_{U_0,<-\lambda,\ext}$ is an asymptotic invariant
of $D$ (for $\lambda$ as in \pref{nonpar2}).
Compare also \cref{cindext}.

The formula in \cref {cindwex} can be used to define
high energy $\eta$-invariants in the case where the ends of $M$ are not smooth.
We expect that these enjoy nice properties because the family
of high energy operators $A_t^{\rm he}$ has no spectral flow.

We conclude this chapter by explaining the

\begin{proof}[Proof of \tref{pindfor}]
Since $M$ has only finitely many ends,
there is a decomposition $M=M_0\cup U_0$,
where $M_0$ and $U_0$ are domains in $M$
such that $M_0$ is compact,
such that the common boundary $N:=M_0\cap U_0$
of $M_0$ and $U_0$ is smooth,
such that each connected component of $N$
bounds exactly one connected component of $U_0$,
and such that the latter are in one to one correspondence
with the ends of $M$.

For each connected component $C$ of $N$,
let $A^+_{C}$ be the restriction of $A^+_0$ to sections of $E$
with support on $C$.
Then $A^+_0$ is the direct sum of the $A^+_{C}$
over the connected components $C$ of $N$.
Hence
\begin{equation*}
  \eta(A^+_0) = \sum\nolimits_{C} \eta(A^+_{C})
  \quad\text{and}\quad
  \dim\ker A^+_0 = \sum\nolimits_{C} \dim\ker A^+_{C} .
\end{equation*}
For the connected component  $\mathcal C$ of $U_0$
with $\partial\mathcal C=C$ we now set
\begin{multline*}
  {\rm Corr} (\mathcal C) :=
  \ind D^+_{\mathcal C,<0,\ext}
  - \int_{\mathcal C} \omega_{D^+} + \int_C \tau_{D^+} \\
  + \frac12 \left(\eta(A^+_C) + \dim \ker A^+_C \right) .
\end{multline*}
Then, by  \tref{nonpar5} and \eqref{aps},
\begin{equation*}
  \ind D^+_{\ext} = \int_M \omega_{D^+}
  + \sum\nolimits_{\mathcal C} {\rm Corr} (\mathcal C) .
\end{equation*}
By Theorem 3.24 of \cite{disy},
the terms ${\rm Corr}(\mathcal C)$ only depend on the ends of $M$
and not on the chosen decomposition of $M$ as above.
\end{proof}

\section{Manifolds with $\varepsilon$-Thin Ends}
\label{almflat}

Let $N$ be a closed and connected Riemannian manifold of dimension $n$.
We say that $N$ is {\em $\varepsilon$-flat} if
\begin{equation}
  \sqrt{K} \diam N \le \varepsilon ,
  \label{eflat}
\end{equation}
where $K$ is some upper bound of the modulus of the sectional curvature of $N$.
By Gromov's theorem on almost flat manifolds, there is a constant $\varepsilon(n)$
such that $N$ is an infra-nilmanifold if $N$ is $\varepsilon(n)$-flat \cite{Gr}.
In what follows we need some details from the proof of Gromov's theorem
from \cite{BK} and from Section 4 of Ruh's improvement
of Gromov's theorem in \cite{Ru}.
The estimates which we assert below hold if $\varepsilon(n)$
is chosen sufficiently small.
The arguments in the proofs of these assertions are elementary albeit intricate.

For any curve $c:[a,b]\to N$, denote by $L(c)$ the length of $c$
and by $h(c)$ parallel translation along $c$.
For orthogonal transformations $A$ and $B$ between equi-dimensional
Euclidean spaces $V$ and $W$,
we follow \cite{Ru} and let $d(A,B)$ be the maximal angle $\angle(Av,Bv)$,
where $v$ runs over non-zero vectors in $V$.
This is a non-smooth Finsler metric on the space of all orthogonal
transformations from $V$ to $W$,
invariant under precomposition and postcomposition by orthogonal
transformations of $V$ and $W$, respectively,
with injectivity radius and diameter $\pi$.

We begin with results from Chapters 2 and 3 in \cite{BK}.
Normalize the Riemannian metric of $N$ so that $\diam N=1$,
and assume, correspondingly, that $\sqrt{K}\le\varepsilon(n)$.
As in \cite{Ru}, let
\begin{equation}
  w = 2\cdot14^{\dim\SO(n)}
  \quad\text{and}\quad
  \rho \ge 10^4w .
  \label{wrho}
\end{equation}
Let $x$ and $y$ be points in $N$.
Then, if $c_0$ and $c_1$ are geodesics segments from $x$ to $y$
of length $<\rho$ such that $h(c_0)$ and $h(c_1)$ are $10^{-1}$-close,
then $h(c_0)$ and $h(c_1)$ are actually $10^{-5}$-close.
The relation $h(c_0)\sim h(c_1)$
iff $h(c_0)$ and $h(c_1)$ are $10^{-1}$-close
is an equivalence relation among the holonomies of geodesic segments
from $x$ to $y$ of length $<\rho$.
For each such equivalence class of holonomies,
there is a geodesic segment from $x$ to $y$
of length $<2\cdot10^{-4}\rho$
such that its holonomy belongs to the given equivalence class.

Let $c_0$ and $c_1$ be geodesic loops at $x$
such that $L(c_0)+L(c_1)<\rho$.
Then there is a unique geodesic loop $c_0*c_1$ at $x$ of length $<\rho$
homotopic to the concatenation of $c_0$ and $c_1$,
and $h(c_0*c_1)$ is $10^{-5}$-close to $h(c_1)\circ h(c_0)$.
This turns the set $H$ of equivalence classes of holonomies
along geodesic loops at $x$ of length $<\rho$ into a group,
and the order of $H$ is at most $w$.

Next we explain Ruh's construction of a flat metric connection on $N$ from \cite{Ru}.
Fix an orthonormal frame $F_0:\R^n\to T_xN$ to identify $T_xN$ with $\R^n$.
For each equivalence class $h\in H$ of holonomies along geodesic loops at $x$
of length $<\rho$, let $b_0(h)\in\O(T_xN)\simeq\O(n)$ be its barycenter.
This defines an almost homomorphism $b_0:H\to\O(n)$ in the sense of \cite{GKR}
and $b_0$ is $10^{-4}$-close to a homomorphism $b:H\to\O(n)$,
by Theorem 3.8 of \cite{GKR}.
It follows that $b$ is injective, and we use $b$ to identify $H$
with its image in $\O(n)$.

Let $c_0$ be a geodesic segment from $x$ to $y$ of length $<\rho$.
For each geodesic segment $c$ of length $<\rho$ from $x$ to $y$,
there is precisely one $h\in H$
such that $h(c)\circ h$ is $10^{-4}$-close to $h(c_0)$.
Enrich the equivalence class of $h(c_0)$ as above by all such $h(c)\circ h$.

Choose a smooth monotone function $\chi:\R\to\R$
with $\chi(r)=1$ for $r\le\rho/3$, $\chi(r)=0$ for $r\ge2\rho/3$,
and $|\chi'|\le10/\rho$.
For any enriched equivalence class $[h(c)\circ h]$ of holonomies as above,
let $b([h(c)\circ h])$ be its barycenter with respect to the weights $\chi(L(c))/\nu$,
where $\nu$ is the order of $H$ times the sum of the $\chi(L(c))$,
over all geodesic segments $c$ from $x$ to $y$ of length $<\rho$.
By the equivariance of barycenters with respect to orthogonal transformations,
the set of the barycenters $b([h(c)\circ h])$ is invariant under right multiplication
by elements from $H$,
and hence the frames $b\circ F_0$, where $b$ runs over the above barycenters,
define a reduction of the principal bundle of orthonormal frames of $N$
to a principal subbundle with structure group $H$.
In other words, we get a flat metric connection $\bar\nabla$ on $N$
with holonomy in $H$.

To estimate the norm of the difference between $\bar\nabla$
and the Levi-Civita connection $\nabla$ of $N$,
we go back one step and consider the situation before taking barycenters.
Let $v\in T_yN$ and $\sigma=\sigma(s)$ be a curve through $y$
with $s$-derivative $\dot\sigma(0)=v$.
Let $c_0:[0,1]\to N$ be a geodesic segment from $x$ to $y$ with $L(c_0)<\rho$.
There is a unique geodesic variation $c=c_s(t)$ of $c_0$
with $c_s(0)=x$ and $c_s(1)=\sigma(s)$,
and then $L(c_s)<\rho$ for all (sufficiently small) $s$.
Let $u\in T_xN$ and $X=X(s,t)$ be the vector field along $c$
such that $X(s,0)=u$ and
such that $X$ is parallel along the segments $c_s$.
Note that parallel translation along $\sigma$ with respect to $\bar\nabla$
corresponds to taking barycenters of such $X(s,1)$ along $\sigma$,
arising from geodesic segments from $x$ to $y$ of length $<\rho$.

We have $\nabla_t\nabla_sX=R(c',J)X$,
where the $s$-derivative $J:=\dot c$ of $c$
is a Jacobi field along each of the $c_s$
which vanishes at $t=0$ and is equal to $\dot\sigma(s)$ at $t=1$.
It follows that
\begin{equation}
  |(\nabla_t\nabla_sX)(0,t)| \le C_0K\rho t |v| |X| ,
\end{equation}
where $C_0$ is a universal constant.
Since taking barycenters depends smoothly on points and weights,
we conclude that
\begin{equation}
  | \bar\nabla_vX - \nabla_vX |
  \le C_1\big( K \rho + \frac{1}{\rho} \big) |v| |X| .
  \label{nabldiff}
\end{equation}
Now, for any given $\delta>0$, we may choose $\rho$ so large and,
accordingly, $\varepsilon=\varepsilon(n,\delta)$ so small,
that the right hand side of \eqref{nabldiff} is $<\delta |v| |X|$.
Hence, reversing the normalization of the diameter, we get that
 \begin{equation}
  | \bar\nabla - \nabla | \le \delta \diam N ,
  \label{dflatcon}
\end{equation}
where we recall that scaling does not change the Levi-Civita connection.
This finishes the exposition of results from \cite{BK} and \cite{Ru}.

\begin{proof}[Proof of \tref{mainth1}]
In the above constructions, it is understood, in the literature,
that the Riemannian manifold $N$ is smooth.
We want to apply it in our situation of straight ends,
where the Riemannian metric of the cross sections
$N_t\subseteq U\simeq[0,\infty)\times N$ is, in general, only $C^1$.
To overcome this technical difficulty,
we note that $f$ can be approximated,
locally uniformly in the $C^2$ topology,
by a sequence of smooth functions $f_k:U\to\R$.
Then, for any given cross section $N_t$,
the level sets $L_k=f_k^{-1}(t)$ approximate $N_t$ in the sense
that there is a $C^1$ diffeomorphism between them such that
Riemannian metric, Levi-Civita connection, and Weingarten operator
on $N_t$ are approximated by the corresponding objects on $L_k$.
In particular,  diameter and modulus of the sectional curvature
of the connected components of the levels $L_k$
are bounded from above by
\begin{equation*}
    d_t + \alpha
    \quad\text{and}\quad
    K = C_R + 2C_W^2 + \alpha ,
\end{equation*}
for any given $\alpha>0$ and all sufficiently large $k$,
where $d_t$ is an upper bound for the diameter of the connected
components of $N_t$ and
where we use the Gauss equation for the second estimate.
Thus the above constructions apply to $N_t$ if
\begin{equation*}
  \sqrt{K} d_t \le \varepsilon < \varepsilon(m-1,1) ,
\end{equation*}
where $K = C_R + 2C_W^2+1$
and $\varepsilon(m-1,1)=\varepsilon(n,\delta)$ is as in the discussion
of \eqref{dflatcon} above,
and they guarantee a flat connection $\bar\nabla^{N_t}$ on $N_t$
such that
\begin{equation*}
  |\bar\nabla^{N_t} - \nabla^{N_t}| \le d_t ,
\end{equation*}
where $\nabla^{N_t}$ denotes the Levi-Civita connection of $N_t$.

Suppose now that $E\to M$ is a Dirac bundle
of the type required in \tref{mainth1}.
Then the restrictions of $E$ to any given cross section $N_t$
is of the corresponding type and hence $\bar\nabla^{N_t}$
induces a flat Hermitian connection $\bar\nabla^{E_t}$
on the restriction $E_t=E|_{N_t}$.
Moreover, the holonomy of $\bar\nabla^{E_t}$
on each connected component of $N_t$ is of order at most $w$.

For convenience, assume now that $N$ is connected.
Decompose $E_t$ into holonomy irreducible components,
and let $F\to N_t$ be any such component.
Then $F$ has a twisted parallel orthonormal frame
\begin{equation}
  \Phi=(\sigma_1,\ldots,\sigma_k) ,
\end{equation}
that is, the sections $\sigma_i$ of $F$ are well defined and parallel on the
induced bundle with induced flat connection over the universal covering of $N_t$.
We think of them as sections of $E$ over $N_t$ which transform
according to the holonomy of $F$.
Approximating the Riemannian metric on $N_t$ by a smooth $\epsilon$-flat
Riemannian metric as above,
we see that we can apply the usual estimates for
the Rayleigh quotient of sections of $F$, that is,
the estimate of Li and Yau \cite{LY} in the case
where $F$ is the trivial complex line bundle
and the corresponding estimate in \cite{eiho} in the other cases.
The outcome is an estimate as follows:
If $\sigma$ is a section of $E$ over $N_t$ and $\sigma$ is orthogonal
to the globally $\bar\nabla^{E_t}$-parallel sections of $E$  over $N_t$,
then
\begin{equation}
  \| \bar\nabla^{E_t}\sigma \|_{N_t}^2
  \ge \frac{C(C_R,C_W,m)}{\varepsilon^2} \| \sigma \|_{N_t}^2 .
  \label{liyaplus}
\end{equation}
Here we use, in the twisted case, that the holonomy of $F$ is non-trivial
in the sense that, for each unit vector $v$ in $F$, there is a loop $c$ in $N_t$
(of length at most $\rho$) such that the angle between $v$ and $hv$ is
at least $\pi/2$, since otherwise the holonomy orbit of $v$ would
be contained in an open spherical ball of radius $\pi/2$ and would have a fixed point.
Hence, for each unit vector $v$ in $F$,
there is a loop $c$ in $N_t$ of length at most $2 d_t$  such that
the angle between $v$ and $h(v)$ is at least $\pi/2w$.

Now the estimate $|\bar\nabla^{N_t}-\nabla^{N_t}|\le d_t$ implies that
\begin{equation*}
  |\bar\nabla^{N_t} - \nabla|_{N_t} | \le d_t + C_W ,
\end{equation*}
where $\nabla$ denotes the Levi-Civita connections of $M$.
Hence
\begin{equation*}
  |\bar\nabla^{E_t} - \nabla^{E}|_{N_t} | \le C(d_t + C_W) ,
\end{equation*}
where $C$ is a constant which depends only on the type of $E$.
It follows that the difference between the Rayleigh quotients
for $\nabla^{E}|_{N_t}$ and $\bar\nabla^{E_t}$ is uniformly bounded.
We conclude that the assumptions of \pref{nonpar2} are satisfied.
\end{proof}

\section{Cuspidal Ends}
\label{secthi}

Assume from now that the ends of $M$ are cuspidal.
In the setup of Definition \ref{straight} and of Section \ref{secdifu},
denote by $\mathcal D$ the Dirac system associated to $E$ over $U$
as in Section \ref{subdisy}.
Clearly, for any $\epsilon>0$,
the cross sections $N_t$ are $\epsilon$-flat for all sufficiently large $t$
so that \tref{mainth1} applies.
On the other hand, in this chapter, we aim at more specific results.
In addition, we do not need to rely on the proof of Gromov's theorem
on almost flat manifolds.

\subsection{The Flat Connection}
\label{susthi}
Over $U$,
define a tensor field $\bar S$ of bilinear maps on $TM\oplus TM$
with values in $TM$ by
\begin{equation}
  \langle \bar S(u,v),w \rangle
  = - \int_s^\infty \langle R(J,T)X,Y \rangle (t,x)\,dt ,
  \label{barsm}
\end{equation}
where $u,v,w\in T_{(s,x)}M$,
$J$ is the $T$-Jacobi field along $\gamma_{(s,x)}:=F(s,x)$ with $J(s)=u$,
and $X,Y$ are the parallel vector fields along $\gamma_{(s,x)}$
with $X(s)=v$, $Y(s)=w$.
The integral converges uniformly,
by \eqref{boundrre} and since the ends are cuspidal.
Hence $\bar S$ is continuous and uniformly bounded.
We let $C_S$ be an upper bound for the operator norm of $\bar S$.

In the analogous way,
define a field $\bar S^E$ of bilinear maps on $TM\oplus E$ with values in $E$,
\begin{equation}
  \langle \bar S^E(u,v),w \rangle
  = - \int_s^\infty \langle R^E(J,T)\sigma_1,\sigma_2 \rangle (t,x)\,dt ,
  \label{barse}
\end{equation}
where now $v,w\in E_{(s,x)}$ and $\sigma_1,\sigma_2$ are the parallel sections
along $\gamma_{(s,x)}$ with $\sigma(s)=v$, $\tau(s)=w$.
Again, the integral converges uniformly,
by \eqref{boundrre} and since  the ends are cuspidal.
Hence $\bar S^E$ is also continuous and uniformly bounded.
We let $C_S^E$ be an upper bound for the operator norm of $\bar S^E$.

The arguments in Section 3 of \cite{BB2} carry over word by word
and show that the continuous metric connections
\begin{equation}
  \bar\nabla := \nabla - \bar S
  \quad\text{and}\quad
  \bar\nabla^E := \nabla^E - \bar S^E
  \label{barcon}
\end{equation}
on $TM$ and $E$ over $U$ are flat in the sense of the existence of
parallel $C^1$ frames over simply connected domains in $U$.
The difference to the situation in Section \ref{almflat} is
that we do not assume that $E$ is geometric
and that we have to pay for it by making stronger assumptions
on the smallness of the Riemannian metrics $g_t$ and by loosing control
on the holonomy of $\bar\nabla$ and $\bar\nabla^ E$.

It is easy to see that
\begin{equation}
  \bar S^E(X,Y\sigma) = \bar S(X,Y)\sigma + Y\bar S^E(X,\sigma) ,
  \label{barc0}
\end{equation}
hence the new connections are compatible with Clifford multiplication as well,
that is,
\begin{equation}
  \bar\nabla^E_X(Y\sigma)
  = (\bar\nabla_XY)\sigma + Y\bar\nabla^E_X\sigma
  \label{barc}
\end{equation}
By definition,
\begin{equation}
   \bar\nabla_T = \nabla_T , \quad
   \bar\nabla^E_T = \nabla^E_T ,
   \quad\text{and}\quad
   \bar\nabla T = 0 .
   \label{bart}
\end{equation}
For each $t\in\R_+$, the restriction of $\bar\nabla$ and $\bar S$ to $N_t$
will be denoted by $\bar\nabla_t$ and $\bar S_t$,
and similarly for $\bar\nabla^E$ and $\bar S^E$.
We also consider $\bar\nabla_t^E$
as a first order differential operator on $H^1(N_t,E)$
with values in $L^2(T^*N_t\otimes E)$.
The formal adjoint of $\bar\nabla_t^E$ is denoted $(\bar\nabla_t^E)^*$.

\begin{rem}\label{belekap}
The above construction of a flat connection is taken from \cite{BB2}
(where it is considered for a narrower class of bundles $E$).
In Appendix C of \cite{BeKa},
Igor Belegradek and Vitali Kapovitch remark that this connection
coincides with the flat connection introduced by Brian Bowditch in \cite{Bow}
(in the case of the tangent bundle of a simply connected, complete Riemannian
manifold with pinched negative sectional curvature),
who uses a kind of parallel translations through infinity
(which, in turn, coincides with the {\em horospherical translations}
in Section 2 of \cite{BrKa}):
Roughly speaking, two vectors $v,w\in E$ with footpoint on a common horosphere
are defined to be parallel if the distance between their parallel translates along
the unit speed geodesics to the center of the horosphere converges to $0$.
\end{rem}

\subsection{The Splitting}
\label{suslohi}
To keep the notation simple,
it will be convenient to assume in this section that $N$ is connected.
It will be obvious that, mutatis mutandis,
the results also apply in the case where $N$ is not connected.

For each $t\in\R_+$,
we let $H^{\rm c}_t$ be the space of $\bar\nabla^E$-parallel sections
of $E$ over $N_t$, that is,
$H^{\rm c}_t$ is the kernel of $\bar\nabla_t^E$.
Here the superscript {\em c} stands for {\em constant}.
We note that the spaces $H^{\rm c}_t$ are invariant
under Clifford multiplication by $T$, by \eqref{barc} and \eqref{bart}.
It is also clear that parallel translation in the $T$-direction identifies the
different spaces $H^{\rm c}_t$, $t\in\R_+$.
In particular, we may and will fix a family of $\bar\nabla^E$-parallel sections
$(\sigma_1, \ldots, \sigma_k)$ of $E$ over $U$ which are pointwise orthonormal
and whose restriction to $N_t$ forms an orthogonal basis of $H^{\rm c}_t$,
for all $t\in\R_+$ simultaneously.

We let $H^{\rm h}_t$ be the orthogonal complement
of $H^{\rm c}_t$ in $L^2(N_t,E)$.
Thus we obtain two families $\mathcal H^{\rm c}=(H^{\rm c}_t)$
and $\mathcal H^{\rm h}=(H^{\rm h}_t)$ of Hilbert spaces,
both of them invariant under Clifford multiplication by $T$.
Note, however,
that $\mathcal H^{\rm h}$ is not parallel in the $T$ direction
if $\mathcal H^{\rm c}$ is non-trivial and the volume density $j=j(t,x)$
as in Section \ref{secdifu} does not only depend on $t$,
but also on $x$, compare \eqref{projp2}.

As before, we use parallel translation to identify the spaces $H^{\rm c}_t$
with $H^{\rm c}_0$,
endowed with the inner products $(.,.)_t=(j_t.,.)_0$.
Since $T$ is parallel in the $T$ direction,
Clifford multiplication by $T$ does not depend on $t$
after this identification.

Let $\bar P_t$ and $\bar Q_t:=I-\bar P_t$ be the orthogonal projections in $H_t$
onto $H^{\rm c}_t$ and $H_t^{\rm h}$, respectively.
By definition,
\begin{equation}\label{projp}
  \bar P_t\sigma
  = \frac1{\vol N_t} \sum\nolimits_{1\le i\le k} (\sigma_i,\sigma)_t \sigma_i .
\end{equation}
For any function $\psi=\psi(t,x)$ on $U$
we denote by $\bar\psi=\bar\psi(t)$ the function
which associates to $t\in\R_+$ the mean of $\psi$
over the cross section $N_t$.
By \eqref{odejac} and \eqref{projp}, we have
\begin{equation}
  (\nabla_T\bar P)\sigma = \bar P(\kappa\sigma) - \bar\kappa\bar P\sigma .
  \label{projp2}
\end{equation}
Associated to the projections $\bar P$ and $\bar Q$, we consider the operators
\begin{equation}
  D^{\rm c} := \bar PD\bar P ,
  \quad D^{\rm h} := \bar QD\bar Q ,
  \quad D^{\rm ch} := \bar PD\bar Q ,
  \quad D^{\rm hc} := \bar QD\bar P .
  \label{lehe}
\end{equation}
We use corresponding notations and conventions in other cases.

\begin{prop}\label{susys}
The family
\begin{equation*}
  \mathcal D^{\rm c}
  :=(\mathcal H^{\rm c},\mathcal A^{\rm c},T)
\end{equation*}
is a Dirac system in the sense of Section \ref{preds} with finite rank and
\begin{equation*}
  \partial^{\rm c} = \frac{d}{dt} + \frac{\bar\kappa}2
  \quad\text{and}\quad
  D^{\rm c} = T(\partial^{\rm c} + A^{\rm c}) .
\end{equation*}
\end{prop}

\begin{proof}
The sections $\sigma_1,\ldots,\sigma_k$ as above are $C^1$,
so that the image $H^{\rm c}_t$ of $\bar P_t$
consists of $C^1$ sections of $E$ over $N_t$.
Hence $H^{\rm c}_t$ is contained in $H_A$, for all $t\in\R_+$.
Furthermore, $A^{\rm c}_t= \bar P_tA_t\bar P_t$
is a bounded and symmetric operator on $H^{\rm c}_t$.
Clearly, for $\sigma_1,\sigma_2\in H^{\rm c}_0$,
\begin{equation*}
  | (\bar P_tA_t\bar P_t\sigma_1,\sigma_2)_t
  - (\bar P_sA_s\bar P_s\sigma_1,\sigma_2)_s|
  = | (A_t\sigma_1,\sigma_2)_t - (A_s\sigma_1,\sigma_2)_s| .
  \qedhere
\end{equation*}
\end{proof}

Associated to the decomposition into constant sections
and sections perpendicular to them,
we get an orthogonal splitting
\begin{equation}
  L^2(\mathcal D) = L^2(\mathcal H)
  = L^{2,\rm c}(\mathcal H) \oplus L^{2,\rm h}(\mathcal H) .
  \label{l2lehe}
\end{equation}
where
\begin{equation}\label{l2lehe2}
\begin{split}
  L^{2,\rm c}(\mathcal H)
  &:= L^2(\mathcal H^{\rm c}) \quad\text{and} \\
  L^{2,\rm h}(\mathcal H)
  &:= \{ \sigma\in L^2(\mathcal H) : \bar P\sigma = 0 \} .
\end{split}
\end{equation}
We use corresponding notations for other spaces of sections.

\begin{lem}\label{h1lehe}
The projections $\bar P$ and $\bar Q$
are continuous on $H^1(\mathcal D)$.
In particular, as topological vector spaces,
\begin{align*}
  H^1(\mathcal D)
  &= H^{1,\rm c}(\mathcal D) \oplus H^{1,\rm h}(\mathcal D) , \\
  H^1_{\loc}(\mathcal D)
  &= H^{1,\rm c}_{\loc}(\mathcal D) \oplus H^{1,\rm c}_{\loc}(\mathcal D) .
\end{align*}
\end{lem}

\begin{proof}
Since $\sigma_1,\ldots,\sigma_k$ and $\vol N_t$ are $C^1$,
we conclude that
\begin{equation*}
  \bar P(H^1(\mathcal D)) \subseteq H^1(\mathcal D)
  \quad\text{and}\quad
  \bar Q(H^1(\mathcal D)) \subseteq H^1(\mathcal D) ,
\end{equation*}
by \eqref{projp}.
Hence $\bar P$ and $\bar Q=I-\bar P$ are continuous
with respect to the $H^1$-norm,
by the closed graph theorem.
\end{proof}

\begin{lem}\label{ray}
The Rayleigh quotients
\begin{align}
  \bar\rho_t &:= \inf\{ \|\bar\nabla^E_t\sigma\|_t^2/\|\sigma\|_t^2 :
  \sigma\in H^{\rm h}_t \cap H_A , \sigma\ne0 \} , \tag{1}\label{ray1} \\
  \rho_t &:= \inf\{ \|\nabla^E_t\sigma\|_t^2/\|\sigma\|_t^2 :
  \sigma\in H^{\rm h}_t \cap H_A , \sigma\ne0 \} \tag{2}\label{ray2}
\end{align}
tend to infinity as $t$ tends to infinity.
Here $\nabla^E_t$ and $\bar\nabla^E_t$ denote the restrictions
of $\nabla^E$ and $\bar\nabla^E$ to $N_t$.
\end{lem}

\begin{proof}
We discuss the Rayleigh quotients associated to $\bar\nabla^ E$ first.
Split $H^{\rm h}_t \cap H_A=U_t\oplus V_t$,
where $U_t$ consists of sections in $H^1(N_t,E)$
which are linear combinations $\sum\varphi_i\sigma_i$
of the basis $(\sigma_1, \ldots, \sigma_k)$ as above
and where $V_t$ consists of sections in $H^1(N_t,E)$
which are pointwise perpendicular to $\sigma_1,\ldots,\sigma_k$.
Note that $U_t$ and $V_t$ are invariant under $\bar\nabla_t^E$
and perpendicular to each other,
and thus it suffices to consider them separately.

Let  $\sigma=\sum\varphi_i\sigma_i\in U_t$, $\sigma\ne0$.
To be perpendicular to $H^{\rm c}_t$ in $L^2(N_t,E)$
means that the coefficient functions $\varphi_i$ integrate to $0$.
Moreover, the Rayleigh quotient of $\sigma$ is given by the sum of the
Rayleigh quotients corresponding to the Laplace operator on functions on $N_t$.
Hence
\begin{equation*}
  \frac{\|\bar\nabla^E_t\sigma\|_t^2}{\|\sigma\|_t^2}
  = \frac{\sum\| \grad\varphi_i \|_t^2}{\sum\|\varphi_i\|_t^2}
  \ge ce^{ct} ,
\end{equation*}
for some constant $c>0$, by Theorem 7 in \cite{LY}.

Now we consider $V_t$.
Perpendicular to $(\sigma_1,\ldots,\sigma_k)$,
the holonomy of $\bar\nabla$ does not have non-trivial invariant vectors.
Since loops  in $N=N_0$ of length at most $2\diam N$
generate the fundamental group of $N$,
there is a constant $\alpha>$ such that,
for each vector $u$ in some fiber of $E$ over $N$,
there is a loop $c$ in $N$ of length at most $2\diam N$
such that the holonomy $h_c$ of $\bar\nabla$ along $c$
satisfies $| h_cu-u | \ge \alpha | u |$.
For each $t\ge0$, the $\bar\nabla$-holonomy about the curve $c$
shifted to $N_t$ is the same.
We conclude that,
for each $t\in\R_+$ and vector $u$ in some fiber of $E$ over $N_t$,
there is a loop $c$ in $N_t$ of length at most $2\varphi(t)\diam N$
such that the holonomy $h_c$ of $\bar\nabla$ along $c$ satisfies
the same inequality,
\begin{equation*}
  | h_cu - u | \ge \alpha | u | .
\end{equation*}
Hence Theorem 5 in \cite{eiho} applies and shows that the Rayleigh
quotient of $\bar\nabla^E_t$ on $V_t$ tends uniformly to infinity as $t$ tends to infinity.
This shows the first claim.
As for the Rayleigh quotients associated to $\nabla^E_t$,
we recall that the difference $\| \bar\nabla_t^E - \nabla^E_t \| \le C_S^E$.
\end{proof}

\begin{thm}\label{nonparc}
There are constants $\lambda_0,\Lambda_t\ge0$
with $\lim_{t\to\infty}\Lambda_t=\infty$
such that $\spec A_t\cap(\lambda_0,\Lambda_t)=\emptyset$
or, more precisely, such that
\begin{align}
  \|D_t\sigma\|_t &\le \lambda_0\|\sigma\|_t
  \quad\text{for all $\sigma\in H^{\rm c}_t$} ,
   \tag{1}\label{nopale} \\
  \|D_t\sigma\|_t &\ge \Lambda_t\|\sigma\|_t
  \quad\text{for all $\sigma\in H^{\rm h}_t$} .
  \tag{2}\label{nopahe}
\end{align}
In particular, for all sufficiently large $t$,
\begin{enumerate}
\item
$\mathcal D_t$ satisfies the hypothesis of \pref{nonpar2},
\item
$D$ is non-parabolic with respect to $M_t:=M\setminus\{f>t\}$,
\end{enumerate}
where $f$ is the distance function as in Definition \ref{straight}.
\end{thm}

\begin{proof}
By \eqref{licht} and \eqref{lic4},
\begin{equation*}
  \big| \|D_t\sigma\|_t^2 - \|\nabla^t\sigma\|_t^2 \big|
  \le C_K \| \sigma \|_t^2 .
  \qedhere
\end{equation*}
\end{proof}

Assume now that the ends of $M$ are smooth, that is,
the associated distance function $f$ on $U$ is smooth.
Since the ends of $M$ are cuspidal and the curvatures of $M$ and $E$
and the second fundamental forms of the cross sections are uniformly bounded,
we have, with $M_t=M\setminus\{f>t\}$ as above,
\begin{equation}
  \lim_{t\to\infty} \int_{M_t}\omega_{D^+} = \int_{M}\omega_{D^+}
  \quad\text{and}\quad
  \lim_{t\to\infty} \int_{N_t}\tau_{D^+} = 0 ,
  \label{limint}
\end{equation}
compare \eqref{aps}.
By \tref{nonparc}, we may fix the starting time $t=0$ such that the condition
\begin{equation}
  (\Lambda_t-\lambda_0)^2 > 4c_0(c_0 + 2 + \lambda_0 + \Lambda_t)
  \label{recall}
\end{equation}
of \pref{nonpar2} is satisfied for all $t\in\R_+$,
where $\lambda$ and $\Lambda$ there correspond to $\lambda_0$
and $\Lambda_t$ here.

\begin{prop}\label{weinfin}
If $w>0$ satisfies $(w-\lambda_0)^2 > c_0(c_0+2+2w)$, then
\begin{align*}
  \ind D_{-w}^+ &= \int_M \omega_{D^+}
  + \frac12 \big( \dim H^+_{[-\lambda_0,\lambda_0]}(A_0^+)
  + \lim_{t\to\infty} \eta(A_t^{\rm he,+}) \big) .
\end{align*}
\end{prop}

\begin{proof}
Observing that  $\dim H^+_{[-\lambda_0,\lambda_0]}(A_t)$
is independent of $t\in\R_+$,
the assertion follows immediately from  \cref{cindwex} and \eqref{limint}.
\end{proof}

As for the index of $D_{\ext}^+$, we refer to \cref{dwcor}.

\section{Explicit Index Formulas}
\label{susifo}
It is certainly desirable to get more explicit index formulas than the one in \pref{weinfin}.
To that end one might hope that the decomposition into low and high energy parts
along the cusps should lead to a decomposition of the corresponding
Dirac system $\mathcal D$ into low and high energy Dirac systems.
Whereas the low energy part leads to the Dirac system $\mathcal D^c$ as in \pref{susys},
the high energy part does not seem to define a Dirac system in general.
If it did, we would want to apply \pref{nonpar0} to show that the high energy part
would not contribute to the index.
However, it is also a problem to control the mixed parts of the Dirac operator.
Under additional technical assumptions on $\mathcal D$ and with quite some effort,
we can solve these problems.
However, at this point this part of our work is rather technical
and does not seem to lead to interesting applications.
For the applications in this article, we can restrict to a more special situation.

Since we want to discuss the contribution of each of the ends of $M$
to the extended index of $D$ as in \tref{pindfor} separately,
we will also consider from now on the objects introduced so far
for the different ends separately.
We let $\mathcal C$ be an end of $M$ and keep the notation as before,
except for adding an index $\mathcal C$ when appropriate.
For better readability, we will not decorate the open sets $U$ and $U_t$,
the cross sections $N_t$,
and the distance function $f$ on $U$ with an index $\mathcal C$.

We assume that $\mathcal C$ is smooth and cuspidal and that, in addition,
\begin{equation}
  \kappa_{\mathcal C} = \bar\kappa_{\mathcal C}
  \quad\text{and}\quad
  A^{\rm le}_{\mathcal C} \sigma_i = \sum\nolimits_j a_i^j\sigma_j ,
  \label{hass}
\end{equation}
for some fixed Hermitian matrix $(a_i^j)\in\Gl(k,\C)$.
Here $(\sigma_1,\ldots,\sigma_k)$ is a family of $\bar\nabla^E_{\mathcal C}$-parallel
sections of $E$ over $U$ which are pointwise orthonormal
and whose restriction to $N_{t}$ forms an orthogonal basis
of $H^c_{\mathcal C,t}$, for all $t\in\R_+$, as further up.
The above conditions hold for homogeneous cusps as discussed
in Chapters \ref{sechc} -- \ref{secle}.

The second condition of \eqref{hass} requires that the space $H^c_{\mathcal C,t}$
of constant sections in $H_{\mathcal C,t}$ is invariant under $A_{\mathcal C,t}$.
By \tref{nonparc}, we get that
\begin{equation}
\begin{split}\label{lecheh}
  H^{\rm le}_{\mathcal C,t}
  &= H^{\rm c}_{\mathcal C,t}
  = H_{\mathcal C,[-\lambda_0,\lambda_0]}(A_{\mathcal C,t}) , \\
  H^{\rm he}_{\mathcal C,t}
  &= H^{\rm h}_{\mathcal C,t}
  = H_{\mathcal C,\R\setminus[-\lambda_0,\lambda_0]}(A_{\mathcal C,t}) ,
\end{split}
\end{equation}
compare \eqref{lopar} and \eqref{hipar}.
The additional assumption $\kappa_{\mathcal C} = \bar\kappa_{\mathcal C}$ implies
that the high energy family
$\mathcal H^{\rm he}_{\mathcal C}=(H^{\rm he}_{\mathcal C,t})$
is invariant under parallel translation
so that it defines a Dirac subsystem $\mathcal D^{\rm he}_{\mathcal C}$
of $\mathcal D_{\mathcal C}$,
as in the case of the low energy system
$\mathcal D^{\rm le}_{\mathcal C}:=\mathcal D^{\rm c}_{\mathcal C}$;
compare \eqref{projp2} and \pref{susys}.
Along $\mathcal C$,
we obtain corresponding low and high energy Dirac operators $D^{\rm le}$
and $D^{\rm he}$, decomposing the original Dirac operator $D$.
Since $A^{\rm le}_{\mathcal C,t}$ does not depend on $t$
and the spectral gap of $A_{\mathcal C,t}^{\rm he}$ tends to infinity as $t\to\infty$,
we may suppose that the assumptions on the spectral gap in \pref{nonpar2}
are satisfied, for all $t\in\R_+$.

\begin{lem}\label{heiso}
Under the above assumptions,
\begin{equation*}
  D^{\rm he}_{U_t,<\lambda,\ext}
  = D^{\rm he}_{U_t,\le \lambda,\ext}
  =   D^{\rm he}_{U_t,<\Lambda_t,\ext} .
\end{equation*}
and $D^{\rm he}_{U_t,<\lambda,\ext}$
and $D^{\rm he,\pm}_{U_t,<\lambda,\ext}$ are isomorphisms,
for all $t\ge0$ and $-\Lambda_t<\lambda<\Lambda_t$.
In particular, for all such $t$ and $\lambda$,
\begin{equation*}
  \ind D^+_{U_t,<\lambda,\ext}
  = \ind D^{\rm le,+}_{U_t,<\lambda,\ext} .
\end{equation*}
\end{lem}

\begin{proof}
The first assertion is clear since the spectrum of $A^{\rm he}_{\mathcal C,t}$
does not intersect the interval $(-\Lambda_t,\Lambda_t)$.
Furthermore, $D^{\rm he}_{U_t,<0,\ext}$ is injective, by \cref{nonpar1}.
Now $D^{\rm he}_{U_t,<\lambda,\ext}$
and $D^{\rm he}_{U_t,\le-\lambda,\ext}$ are adjoints of each other,
hence $D^{\rm he}_{U_t,\le0,\ext}$ is surjective.
\end{proof}

Since $\kappa_{\mathcal C}$ depends (at most) on $t$
and $j_{\mathcal C}$ solves the initial value problem
$j'_{\mathcal C}=\kappa_{\mathcal C} j_{\mathcal C}$ with $j_{\mathcal C,0}=1$,
we conclude that $j_{\mathcal C}=j_{\mathcal C}(t,x)$ depends only on $t$ as well.
Then the linear map
\begin{equation}
  \Phi: L^2(\R_+,\R^k) \to L^2(\mathcal H^{\rm le})_{\mathcal C} , \quad
  \Phi(\varphi) = j_{\mathcal C}^{-1/2}\sum\nolimits_i\varphi^i\sigma_i ,
  \label{letrafo}
\end{equation}
is a unitary isomorphism such that
\begin{equation}
  \Phi^{-1} D^{\rm le} \Phi = \bar T \left(\frac{d}{dt} + A^{\rm le}_{\mathcal C,0} \right) ,
  \label{lesys}
\end{equation}
where $\bar T=\Phi^{-1}T\Phi$.
This is a finite rank Dirac system with constant coefficients.
In the super-symmetric case, we get a system of the form
\begin{equation}
  \Phi^{-1} D^{\rm le} \Phi
  = \begin{pmatrix} 0 & -1 \\ 1 & \phantom{-}0 \end{pmatrix}
  \left( \frac{d}{dt}
  + \begin{pmatrix} A^{\rm le,+}_{\mathcal C,0} & 0 \\
  0 & A^{\rm le,-}_{\mathcal C,0} \end{pmatrix}
  \right) ,
  \label{lesys2}
\end{equation}
where $A^{\rm le,-}_{\mathcal C,0}=-\bar A^{\rm le,+}_{\mathcal C,0}$.
As for high energy,
\begin{equation}
  \Psi: L^2(\R_+,H^{\rm he}_{\mathcal C,0})
  \to L^2(\mathcal H^{\rm he}_{\mathcal C}) , \quad
  \Phi(\varphi) = j^{-1/2}_{\mathcal C} P^{/\!\!/}_{\mathcal C} \varphi ,
  \label{hetrafo}
\end{equation}
is also a unitary isomorphism, and we have
\begin{equation}
  \Psi^{-1} D^{\rm he} \Psi = \bar T \left(\frac{d}{dt}
  + A_{\mathcal C,t}^{\rm he} \right) ,
  \label{hesys}
\end{equation}
where $\bar T=\Psi^{-1}T\Psi$.
There is a corresponding formula in the super-symmetric case,
substituting $\bar A_{\mathcal C,t}^{\rm he,\pm}$
for $\bar A_{\mathcal C,0}^{\pm}$ on the right in \eqref{lesys2}.

\begin{prop}\label{exinle}
Under the above assumptions, $D^{\rm le}_{<0,\ext}$
and $D^{\rm le,+}_{<0,\ext}$ are isomorphisms.
\end{prop}

\begin{proof}
The Dirac system \ref{lesys} does not have extended
or $L^2$-solutions $\sigma$ with $\sigma(0)$ in $H^{\rm le}_{<0}$
or $H^{\rm le}_{\le0}$, respectively.
\end{proof}

\begin{thm}\label{exinfin}
If all ends of $M$ are smooth and cuspidal and satisfy \eqref{hass}, then
\begin{equation*}
  \ind D_{\ext}^+ = \int_M \omega_{D^+}
  + \frac12\sum_{\mathcal C} \left(
  \lim_{t\to\infty} \eta(A^{\rm he,+}_{\mathcal C,t})
  + \eta(A^{\rm le,+}_{\mathcal C,0}) + \dim\ker A^{\rm le,+}_{\mathcal C,0} \right) .
\end{equation*}
\end{thm}

\begin{proof}
By \tref{nonpar5}, \lref{heiso}, and \pref{exinle}, we have
\begin{equation*}
  \ind D_{\ext}^+ = \ind D^+_{M_t,0} ,
\end{equation*}
for all $t\in\R_+$.
Hence the assertion follows from \eqref{aps} and \eqref{limint},
where we separate $\eta$-invariant and $\dim\ker A_t^+$
according to cusp $\mathcal C$ and low and high energy,
observing that $\dim\ker A^{\rm he,+}_{\mathcal C,t}=0$
and that $A^{\rm le,+}_{\mathcal C,t}=A^{\rm le,+}_{\mathcal C,0}$,
for all $t\in\R_+$.
\end{proof}

Recall the quantities $h^\pm_\infty:=\dim\ker D^\pm_{\ext}-\dim\ker D^\pm_{\max}$
from \eqref{hinfty}.

\begin{thm}\label{maxinf}
If all ends of $M$ are smooth and cuspidal and satisfy \eqref{hass}, then
\begin{equation*}
  \ind_{L^2} D^+ = \int_M \omega_{D^+}
  + \frac12\sum_{\mathcal C} \left(
  \lim_{t\to\infty} \eta(A^{\rm he,+}_{\mathcal C,t})
  + \eta(A^{\rm le,+}_{\mathcal C,0}) \right)
  - \frac12 \big( h^+_{\infty} - h^-_{\infty} \big) .
\end{equation*}
\end{thm}

\begin{proof}
Since $D$ is formally self-adjoint, the $L^2$-index of $D$ vanishes
and therefore
\begin{equation*}
  \ind D_{\ext} = \ind D^+_{\ext} + \ind D^-_{\ext} = h^+_{\infty} + h^-_{\infty} .
\end{equation*}
On the other hand, we have
\begin{equation*}
  \omega_{D^-} = -\omega_{D^+}
  \quad\text{and}\quad
  A^-_{\mathcal C,t} = - A^+_{\mathcal C,t} ,
\end{equation*}
for all $t\in\R_+$ and cusps $\mathcal C$ of $M$.
Therefore, applying \tref{exinfin} to $D^+$ and $D^-$, we obtain that
\begin{align*}
  \ind D_{\ext} &= \ind D^+_{\ext} + \ind D^-_{\ext} \\
  &= \frac12\sum_{\mathcal C} \big(
  \dim\ker A^{\rm le,+}_{\mathcal C,0} + \dim\ker A^{\rm le,-}_{\mathcal C,0} \big)
  = \dim\ker A^{\rm le,+}_{\mathcal C,0}
\end{align*}
since the integral and $\eta$ terms for $D^+$ and $D^-$ cancel each other.
We conclude that
\begin{equation}\label{maxinf5}
  h^+_{\infty} + h^-_{\infty} = \dim\ker\bar A^{\rm le,+}_{\mathcal C,0}
\end{equation}
and hence that
\begin{equation}
  \ind D^+_{\ext} - \ind_{L^2} D^+ = h^+_{\infty}
  = \frac12 \big( h^+_{\infty} - h^-_{\infty} + \dim\ker A^{\rm le,+}_{\mathcal C,0} \big) .
  \qedhere
\end{equation}
\end{proof}

\begin{rems}\label{remhpm}
1) It is immediate from \eqref{lesys2} that $D^+$ is of Fredholm type
if and only if the kernel of $\bar A^+$ vanishes or,
equivalently, if and only if $h^+_{\infty}=h^-_{\infty}=0$.

2) Theorems \ref{exinfin} and \ref{maxinf} generalize
Theorems \ref{intexinfin} and \ref{intmaxinf} from the introduction.
In fact, the discussion in the first part of Chapter \ref{sechc}
implies that homogeneous cusps are smooth and satisfy \eqref{hass}.
\end{rems}

\section{Homogeneous Cusps}
\label{sechc}

Let $N$ be a simply connected nilpotent Lie group
with Lie algebra $\mathfrak n$.
Fix a left-invariant Riemannian metric $g$ on $N$,
and let $W$ be a negative definite and symmetric derivation of $\mathfrak n$.
Then $(\exp(-tW))_{t\in\R}$
is a one-parameter group of automorphisms
of $\mathfrak n$ which induces a one-parameter group $(\Phi_t)_{t\in\R}$
of automorphisms of $N$.
The associated semi-direct product $S:=\R\ltimes N$, where
\begin{equation}
  (s,x) (t,y) := (s+t,x\Phi_{s}(y)) ,
  \label{spro}
\end{equation}
is a simply connected solvable Lie group containing $N\cong\{0\}\times N$
as a subgroup of codimension one.
The vector field $T:=\partial/\partial t$ on $S$ is left-invariant,
and the Lie algebra $\mathfrak s$ of $S$ extends $\mathfrak n$ by
\begin{equation}
  [T,X] = - WX ,
  \label{lies}
\end{equation}
where $X\in\mathfrak n$.
For later use, we note that left translation, right translation,
and conjugation with $(t,e)\in S$ are given by
\begin{equation}\label{ltrt}
\begin{split}
  L_{(t,e)}(s,x) &= (s+t,\Phi_t(x)) , \\
  R_{(t,e)}(s,x) &= (s+t,x) , \\
  (t,e)(s,x)(-t,e) &= (s,\Phi_t(x)) ,
\end{split}
\end{equation}
respectively.
In particular,
the shift by $t$ along the $T$-lines is obtained by right translation with $(t,e)$.
Moreover, for $X\in\mathfrak n\subseteq\mathfrak s$,
\begin{equation}\label{dshif}
\begin{split}
  R_{(t,e)*}X_{(s,x)}
  &= L_{(s,x)*}L_{(t,e)*}L_{(-t,e)*}R_{(t,e)*}X \\
  &= L_{(s+t,x)*}(\Ad_{(t,e)}^{-1}X)
  = L_{(s+t,x)*} (\exp(tW)X) ,
\end{split}
\end{equation}
where we recall that $(\Phi_t)$ is the one-parameter group of automorphism
of $N$ associated to $-W$
(and where we identify $\mathfrak s\ni X=X_e\in T_eS$).

Endow $S$ with the left-invariant Riemannian metric
which agrees with $g$ along $N$
and such that $\R$ and $\mathfrak n$ are pairwise perpendicular with $|T|=1$.
Note that $T$ is a unit normal field along the cross sections $N_t:=\{t\}\times N$
and that the $T$-lines are unit speed geodesics.
In particular,
\begin{equation}
  f: S \to \R, \quad f(t,x):=t ,
  \label{difus}
\end{equation}
is a smooth distance function on $S$ such that $\grad f=T$
and such that the associated diffeomorphism $F$
is the identity on $S=\R\times N$.
By the Koszul formula and the symmetry of $W$,
\begin{equation}
  \nabla_TX = 0 ,
  \label{dtx}
\end{equation}
for any $X\in\mathfrak s$.
For any $X\in\mathfrak n\subseteq\mathfrak s$,
\begin{equation}
  \nabla_XT = WX ,
  \label{dxt}
\end{equation}
by \eqref{lies} and \eqref{dtx};
that is, except for the compactness of the cross sections,
we are in the situation of Section \ref{secdifu}.
By \eqref{dtx} and \eqref{dxt},
\begin{equation}
  R(T,X)Y = \nabla_{[X,T]}Y = \nabla_{WX}Y ,
  \label{rxt}
\end{equation}
for all $X\in\mathfrak n$ and $Y\in\mathfrak s$.
In  particular,
\begin{equation}
  R(X,T)T = - W^2X ,
  \label{rxtt}
\end{equation}
and hence the sectional curvature of tangential $2$-planes of $S$
containing $T$ is strictly negative.

Let $\Gamma\subseteq N$ be a discrete subgroup
such that the quotient $\Gamma\backslash N$ is compact.
Since $\Gamma\subseteq N$,
the distance function $f$ as in \eqref{difus}
is well defined on $\Gamma\backslash S$.
We keep the notation $f$ and $T=\grad f$ on the quotient.
The cross sections of $f$ are given by $\{t\}\times\Gamma\backslash N$,
and right translation by $(t,e)$ induces the shift $F_t$
from $\Gamma\backslash N$ to $\{t\}\times\Gamma\backslash N$, see \eqref{ltrt}.
By \eqref{dshif}, $F_t$ has derivative $F_{t*}=\exp(tW)$.
The Jacobian of $F_t$ is given by $j(t)=\exp(\kappa t)$,
where $\kappa=\tr W$ as in Section \ref{secdifu}.
It only depends on $t$ and not on $x\in\Gamma\backslash N$.
Moreover, since $W$ is negative definite, $F_t$ is contracting for $t>0$:
If we order the eigenvalues of $W$,
\begin{equation}
  \kappa_2 \le \ldots \le \kappa_m < 0 ,
\end{equation}
then any part $[t_0,\infty)\times N$ of $\R\times N$
models cuspidal ends as in Definition \ref{defthin}
with $c=-2\kappa_m$ and $C=1$.
We call such ends {\em homogeneous cusps}.

If $X_i\in\mathfrak n$ is a unit eigenvector of $W$ for the eigenvalue $\kappa_i$,
then $\exp(\kappa_it)X_i$ is a Jacobi field along each $T$-line and
\begin{equation}
  \langle \nabla_XY,Z \rangle
  = - \int_0^\infty \langle R(T,e^{\kappa_it}X_i)Y,Z \rangle ,
\end{equation}
for all $Y,Z\in\mathfrak s$.
It follows that the flat connection $\bar\nabla$ associated to the cusp
as in Section \ref{susthi} defines left-invariant vector fields on $S$ or, rather,
their image in $\Gamma\backslash S$ to be $\bar\nabla$-parallel.

Let $K_0$ be a connected Lie subgroup
of the orthogonal group $\SO(\mathfrak s)$
which contains the holonomy group of $S$ at $e$.
Denote the Lie algebra of $K_0$ by $\mathfrak k$.
Consider the principal bundle $\mathcal P_0:=S\times K_0$ over $S$,
with structure group $K_0$,
where we view $p=(s,k)\in\mathcal P_0$
as representing the frame $L_s\circ k:T_eS \to T_sS$ of $S$,
where $L_s$ denotes left-translation by $s$ (and its derivative).
This interpretation corresponds to an embedding of $\mathcal P_0$
into the principal bundle of orthonormal frames of $S$.
The group $S$ acts on $\mathcal P_0$ by left translation,
$s(s',k):=(ss',k)$, and the orbits of this action
are the left-invariant frames $F_k:=\{(s,k)\mid s\in S\}$ over $S$.

\begin{lem}\label{holred}
The Levi-Civita connection $\nabla$ and flat connection $\bar\nabla$
of $S$ reduce to $\mathcal P_0$.
That is, if $c:I\to S$ is a smooth curve and $F$ is a parallel frame
along $c$ with respect to $\nabla$ or $\bar\nabla$
such that $F(t_0)\in\mathcal P_0$ for some $t_0\in I$,
then $F(t)\in\mathcal P_0$ for all $t\in I$.
\end{lem}

\begin{proof}
Let $F$ be an orthonormal frame along $c$,
and write $F(t)=L_{c(t)}f(t)$, where $f:I\to\O(\mathfrak s)$.
Then the covariant derivative of $F$ along $c$
with respect to $\nabla$ is given by
\begin{equation}
  F'(t) = L_{c(t)}(f'(t) + A_{c'(t)}f(t)) ,
  \label{frapar}
\end{equation}
where
\begin{equation}\label{axy}
  A_{X} =
 \begin{cases}
  R(T,W^{-1}X) &\text{for $X\in\mathfrak n$} , \\
  0  &\text{for $X=T$} ,
 \end{cases}
\end{equation}
by \eqref{dtx} and \eqref{rxt}.
By \eqref{frapar}, $F$ is $\nabla$-parallel if $f' + A_{c}f = 0$.

Now $R(Y,Z)$ is in the Lie algebra of the holonomy group of $S$ at $e$,
for all $Y,Z\in\mathfrak s$, hence also $A_{c'(t)}$, for all $t\in I$.
Since $K_0$ contains the holonomy group of $S$ at $e$,
we get that $A_{c'(t)}\in\mathfrak k$, for all $t\in I$.
It follows that a solution of $f' + A_{c}f = 0$ is contained in $K_0$
if $f(t_0)$ is in $K_0$, for some $t_0\in I$.
This proves the assertion for $\nabla$.

By what we said above, a frame is $\bar\nabla$-parallel
if and only if it is left-invariant under $S$.
Hence the $\bar\nabla$-parallel frames along $c$
are of the form $F(t)=L_{c(t)}k$, $t\in I$,
where $k\in\O(\mathfrak s)$.
Hence, if $F(t_0)\in\mathcal P_0$ for some $t_0\in I$,
then $k\in K_0$, and then $F=F_k$ is contained in $\mathcal P_0$.
\end{proof}

Let $K\to K_0$ be a covering homomorphism,
where $K$ is a connected Lie group,
and let $\mathcal P:=S\times K$
be the corresponding covering space of $\mathcal P_0$,
a principal bundle over $S$ with structure group $K$.
Via the projection $K\to K_0$,
identify the Lie algebra of $K$
with the Lie algebra $\mathfrak k$ of $K_0$.
As in the case of $\mathcal P_0$,
$S$ acts by left translations on $\mathcal P$,
and we have the corresponding orbits $F_k$, $k\in K$.
Moreover, since $\mathcal P\to\mathcal P_0$ is a covering projection,
Levi-Civita and flat connection lift from $\mathcal P_0$ to $\mathcal P$.

Denote by $\hat\alpha_*:\mathfrak k\to\mathfrak u(\Sigma_{\mathfrak s})$
the composition of the differential
of $\alpha:K\to K_0\subseteq\SO(\mathfrak s)$
with the differential of the spinor representation $\Sigma_{\mathfrak s}$
of $\mathfrak{so}(\mathfrak s)\simeq\mathfrak{spin}(\mathfrak s)$.
Let $V$ be a finite dimensional Hermitian vector space
and $\pi_*:\mathfrak k\to\mathfrak u(V)$ be a unitary representation.
Suppose that there is a unitary representation
$\beta:K\to\Sigma_{\mathfrak s}\otimes V$ with
\begin{equation}
  \hat\alpha_*\otimes\id + \id\otimes\pi_* = \beta_* ,
  \label{hdb}
\end{equation}
and let $E=\mathcal P\times_\beta(\Sigma_{\mathfrak s}\otimes V)$
be the associated Hermitian vector bundle over $S$.
Levi-Civita and flat connection on $\mathcal P$
induce Hermitian connections $\nabla^E$ and $\bar\nabla^E$ on $E$,
respectively.
We extend Clifford multiplication to $\Sigma_{\mathfrak s}\otimes V$ by
\begin{equation}
  X \cdot (u\otimes v) := (X \cdot u) \otimes v ,
  \label{clixuv}
\end{equation}
where $X\in\mathfrak s$, $u\in\Sigma_{\mathfrak s}$, and $v\in V$.
By \eqref{hdb} and since $K$ is connected,
Clifford multiplication commutes with $\beta$, that is
\begin{equation}
  \beta(k)(Xw) = X(\beta(k)w) ,
  \label{clibet}
\end{equation}
for all $k\in K$, $X\in\mathfrak s$, and $w\in\Sigma_{\mathfrak s}\otimes V$.
Hence \eqref{clixuv} induces a Clifford multiplication on $E$
which turns $E$ into a Dirac bundle over $S$.
The canonical action of $S$ on $E$ preserves the Dirac data of $E$;
we say that $E$ is a {\em homogeneous} Dirac bundle over $S$.

Using the left-invariant orbit $F_e$ in $\mathcal P$,
we view sections of $E$ as smooth maps $\sigma:S\to\Sigma_m\otimes V$.
In this interpretation,
covariant derivatives and Dirac operator are given by
\begin{equation}
  \nabla^E_X\sigma = X(\sigma) + \beta_*(A_X) \sigma ,
  \quad
  \bar\nabla^E_X\sigma = X(\sigma) ,
  \label{dxf}
\end{equation}
and
\begin{equation}
  D\sigma
  = \sum\nolimits_{j} X_j\cdot (X_j(\sigma) + \beta_*(A_{X_j}) \sigma) ,
  \label{dif}
\end{equation}
where $X$ is a vector field on $S$,
$(X_1,\ldots,X_m)$ is an orthonormal frame of $S$,
and $A_X$ is as in \eqref{axy}.
In particular,
$\sigma$ is $\bar\nabla^E$-parallel if and only if  $\sigma$ is constant.

Let $\tau$ be a unitary representation of $\Gamma$ on $V$,
the {\em twist}, and assume that $\tau$ and $\pi_*$ commute, that is,
\begin{equation}
  \tau(\gamma) \pi_*(Y) = \pi_*(Y) \tau(\gamma) ,
  \label{gampi}
\end{equation}
for all $\gamma\in\Gamma$ and $Y\in\mathfrak k$.

\begin{lem}\label{tbc}
Extend $\tau$ by the trivial representation on $\Sigma_{\mathfrak s}$
to $\Sigma_{\mathfrak s}\otimes V$.
Then $\tau$ commutes with $\beta$ and Clifford multiplication,
\begin{align*}
  \tau(\gamma)(\beta(k)w) &= \beta(k)(\tau(\gamma)w) , \\
  \tau(\gamma)(Xw) &= X(\tau(\gamma)w) ,
\end{align*}
for all $\gamma\in\Gamma$, $k\in K$,
$X\in\mathfrak s$, and $w\in\Sigma_{\mathfrak s}\otimes V$.
\end{lem}

\begin{proof}
Since $K$ is connected,
the first assertion follows from the corresponding
infinitesimal properties in \eqref{hdb} and \eqref{gampi}.
As for the second assertion,
we note that $\tau$ acts trivially on the first and Clifford multiplication
trivially on the second factor of $\Sigma_{\mathfrak s}\otimes V$.
\end{proof}

By \lref{tbc}, $\tau$ induces a Hermitian bundle $E_\tau$
over $\Gamma\backslash S$
such that sections of $E_\tau$ correspond to maps
$\sigma:S\to\Sigma_{\mathfrak s}\otimes V$ which satisfy
\begin{equation}
  \sigma(\gamma s) = \tau(\gamma) \sigma(s) ,
  \label{twist}
\end{equation}
for all $s\in S$.
The connections $\nabla^E$ and $\bar\nabla^E$ on $E$
descend to Hermitian connections on $E_\tau$,
also denoted by $\nabla^E$ and $\bar\nabla^E$, respectively.
Moreover, $E_\tau$ inherits Clifford multiplication from $E$
and thus turns into a Dirac bundle over $\Gamma\backslash S$.

\begin{exas}\label{exdibu}
1) (Spinor bundles) Since $S$ is contractible,
spin structures over $\Gamma\backslash S$ are determined
by homomorphisms $\tau:\Gamma\to\{+1,-1\}$.
In our setup, the corresponding spinor bundles
over $\Gamma\backslash S$ can be given by the data:
$K_0=\SO(\mathfrak s)$ and $K=\Spin(\mathfrak s)$,
$\alpha:\Spin(\mathfrak s)\to\SO(\mathfrak s)$ the canonical covering map,
$V=\C$, $\pi_*=0$, $\beta$ the spinor representation,
extended trivially to the factor $\C$ of $\Sigma_{\mathfrak s}\otimes\C$,
and finally the twist defined by $\tau$,
where $\gamma$ acts by multiplication
with $\tau(\gamma)=\pm1$ on $\C$.

2) (Clifford bundle) If $m$ is even,
then $\cl(\mathfrak s)=\Sigma_{\mathfrak s}\otimes\Sigma_{\mathfrak s}$.
Thus, to obtain the Clifford bundle over $\Gamma\backslash S$,
we may take $K_0=K=\SO(\mathfrak s)$, $\alpha=\id$,
$V=\Sigma_{\mathfrak s}$, $\beta_*$ the differential of the spinor representation,
and $\tau$ the trivial representation of $\Gamma$ on $\Sigma_{\mathfrak s}$.
\end{exas}

If the dimension $m$ of $S$ is even,
then the $\pm1$-eigenspaces $\Sigma_{\mathfrak s}^\pm\otimes V$
of multiplication by the complex volume form (compare Section \ref{decsig})
are invariant under $\beta$, by \eqref{clibet}.
By \lref{tbc}, they are also invariant under $\tau$.
Thus the complex volume form yields
the super-symmetry $E = E^+ \oplus E^-$ with
\begin{equation}
  E^\pm = \mathcal P\times_{\beta}(\Sigma_{\mathfrak s}^\pm\otimes V) .
\end{equation}
In the case of the Clifford bundle,
there is another natural super-symmetry,
namely the even-odd decomposition.
Our methods also allow for a discussion of the latter,
but here and below we concentrate on the decomposition
given by the complex volume form.

We now pass to the Dirac system associated
to the distance function $f$ and the Dirac bundle $E_\tau$
over  $\Gamma\backslash S$.
We identify sections of $E_\tau$ over $\{t\}\times\Gamma\backslash N$
with maps $\sigma:N\to\Sigma_{\mathfrak s}\otimes V$ satisfying \eqref{twist}.
Under this identification,
parallel translation along the $T$-lines is the identity,
and the Hilbert space $L^2(\{t\}\times\Gamma\backslash N,E_\tau)$
corresponds to the Hilbert space of measurable maps
$N\to\Sigma_{\mathfrak s}\otimes V$ satisfying \eqref{twist}
which are square integrable over a fundamental domain of $\Gamma$.
In the notation of \eqref{corrd},
\begin{equation}
  A_t\sigma
  = - \sum_{2\le j\le m} e^{-\kappa_jt}TX_j \cdot X_j(\sigma)
  - \sum_{2\le j\le m} TX_j \cdot \beta_*(A_{X_j})\sigma
  - \frac{\kappa}{2}\sigma ,
  \label{ats}
\end{equation}
where $(X_2,\ldots,X_m)$ is an orthonormal basis of $\mathfrak n$
consisting of eigenvectors of $W$, $WX_i=\kappa_iX_i$.

We may also have a different view on $E_\tau$
over $\{t\}\times\Gamma\backslash N$:
$L_{(t,e)}$ is an isometry of $S$ which maps $N$ to $\{t\}\times N$
and which leaves the normal field $T$ to the cross sections $\{t\}\times N$ invariant.
Suppressing the coordinate $t$ in $\{t\}\times N$,
$L_{(t,e)}$ corresponds to $\Phi_t$, by \eqref{ltrt}.
That is, $E_\tau$ over $\{t\}\times\Gamma\backslash N$
corresponds to $E_{\Phi_t\tau\Phi_t^{-1}}$ over $\Phi_t(\Gamma)\backslash N$,
where $N$ is endowed with the fixed left-invariant metric $g$.
Under this correspondence,
the exponential factors in the expression for $A_t$ in \eqref{ats} disappear.
More precisely, $-A_t$ corresponds to the Dirac operator
\begin{equation}
  D_t\sigma
  = \sum_{2\le j\le m} TX_j \cdot X_j(\sigma)
  + \sum_{2\le j\le m} TX_j \cdot \beta_*(A_{X_j})\sigma
  + \frac{\kappa}{2}\sigma ,
  \label{dts}
\end{equation}
where $\sigma$ satisfies the twist data with respect to $\Phi_t\tau\Phi_t^{-1}$.
In particular, the local data for the different operators $D_t$
coincide under the correspondence.

\subsection{Asymptotic $\eta$-Invariants}
\label{susasy}
Let $L^{2,\pm}(t)$ be the Hilbert space of measurable maps
$N\to\Sigma_{\mathfrak s}^\pm\otimes V$
satisfying \eqref{twist} with respect to $\Phi_t\tau\Phi_t^{-1}$
which are square integrable over a fundamental domain of $\Phi_t(\Gamma)$.
Then $D^\pm_t=-A^\pm_t$ is an unbounded self-adjoint operator on $L^{2,\pm}(t)$.

For the computation of the asymptotic high energy $\eta$-invariant of $D^{+}_t$,
it will be useful to consider the {\em flat Dirac operator} $\bar D^{+}_t$,
defined by
\begin{equation}
  \bar D^+_t\sigma =
  \sum\nolimits_{2\le j\le m} TX_j \cdot X_j(\sigma) .
  \label{tilded}
\end{equation}
We note that $\bar D^+_t$ is a formally self-adjoint operator
and that $D^+_t-\bar D^+_t$ is left-invariant of order zero.
In particular, the principal symbols of $D^+_t$ and $\bar D^+_t$ are the same.
We have
\begin{equation}\label{tilded2}
\begin{split}
  (\bar D^+_t)^2\sigma
  &= \sum_{2\le j,k\le m} TX_j\cdot X_j(TX_k\cdot X_k (\sigma)) \\
  &= - \sum_{2\le j\le m} X_j(X_j (\sigma))
  + \sum_{2\le j\ne k\le m} X_jX_k\cdot X_j(X_k(\sigma)) , \\
  &= \Delta\sigma + \sum_{2\le j<k\le m} X_jX_k\cdot [X_j,X_k](\sigma) ,
\end{split}
\end{equation}
where $\Delta$ denotes the standard Laplace operator of $N$, here acting
on maps with values in the vector space $\Sigma_{\mathfrak s}^\pm\otimes V$.
If $\mathfrak n$ is two-step nilpotent,
then the Lie brackets $[X_j,X_k]$ in the second term on the right of \eqref{tilded2}
are in the center of $\mathfrak n$,
and then the operator defined by the second term commutes with $\Delta$.

The idea to consider $\bar D^+_t$ is taken from \cite{DS}.
The proof of our main result in this direction, \tref{higeta} below,
is a variation of arguments in \S 5 of \cite{DS}.
This line of reasoning was also used by Cheeger and Gromov
in order to show that their $\rho$-invariant is the limit of the (signature)
$\eta$-invariant under a collapse of the corresponding manifold
with bounded covering geometry \cite{CG}.

\begin{thm}\label{higeta}
For $D^+_t$ and $\bar D^+_t$ as above, we have
\begin{equation*}
  \lim_{t\to\infty} \eta(D^{\rm he,+}_t)
  = \lim_{t\to\infty} \eta(\bar D^+_t) ,
\end{equation*}
\end{thm}

\begin{proof}
Left-invariant sections in $L^{2,+}(t)$ are in the kernel of $\bar D^+_t$,
for all $t$.
By what we said above, sections in $L^{2,+}(t)$ are $\bar\nabla$-parallel
if and only if they are left-invariant.
Since the difference $D^+_t-\bar D^+_t$ is left-invariant, hence uniformly bounded,
\tref{nonparc} implies that the kernel of the operator $\bar D^+_t$
consists precisely of the left-invariant sections in $L^{2,+}(t)$,
for all sufficiently large $t$.

Let $P_t:L^{2,+}(t)\to L^{2,+}(t)$ be the orthogonal projection onto the space
of left-invariant sections in $L^{2,+}(t)$.
Then $P_t$ commutes with $\bar D^+_t$ and $D^+_{t,c}$,
where we write $D^+_{t,c}=D^+_t-\bar D^+_t$.
For fixed $t$, consider the family of operators
\begin{equation*}
  D^+_{t,u} := \bar D^+_t + u (I-P_t)D^+_{t,c}(I-P_t) + P_t ,
  \quad 0 \le u \le 1.
\end{equation*}
By definition,
\begin{align*}
   \eta(D^+_{t,1}) &= \eta(D^{\rm he,+}_t) + \dim\im P_t , \\
   \eta(D^+_{t,0}) &= \eta(\bar D^+_t) + \dim\im P_t .
\end{align*}
The non-zero eigenvalues of $\bar D^+_t$ tend to infinity as $t$ tends to $\infty$,
whereas $D^+_{t,c}$ is uniformly bounded independently of $t$.
It follows that $D^+_{t,u}$ is invertible, for all sufficiently large $t$.
Now by Proposition 2.12 in \cite{APS3} and the invertibility of $D^+_{t,u}$,
\begin{equation*}
  \frac{d}{du}\eta(D^+_{t,u})
\end{equation*}
is a local invariant\footnote{In \cite{APS3}
this assertion is only stated for the $\eta$-invariant modulo $\Z$.
However,
as is clear from the remarks preceding Proposition 2.12 in \cite{APS3},
this is only because of the possibility of eigenvalues crossing $0$,
which is excluded by invertibility.},
given by an explicit integral formula constructed out of the complete symbols
of $D^+_{t,u}$ and $(I-P_t)D^+_{t,c}(I-P_t)$.
On the other hand, $P_t$ is (infinitely) smoothing,
and hence the complete symbol of $D^+_{t,u}$ and  $(I-P_t)D^+_{t,c}(I-P_t)$
are the same as those of
\begin{equation*}
  L_{t,u} := \bar D^+_t + uD^+_{t,c}
  \quad\text{and}\quad
  D^+_{t,c} .
  \end{equation*}
Now the symbols of $L_{t,u}$ and $D^+_{t,c}$ do not depend on $t$,
by \eqref{dts} and \eqref{tilded}.
It follows that the local invariant for $d\eta(D^+_{t,u})/du$ is bounded in modulus
by a continuous function $b=b(u)$ which does not depend on $t$.
Therefore we have
\begin{align*}
  | \eta(D^{\rm he,+}_t) - \eta(\bar D^+_t) |
  &= | \eta(D^+_{t,1}) - \eta(D^+_{t,0}) | \\
  &\le \text{const}\cdot \vol(\Phi_t(\Gamma)\backslash N) \to 0.
  \qedhere
\end{align*}
\end{proof}

The fact that the high energy $\eta$-invariant has no spectral flow is perhaps
an indication that its limit deserves to be investigated along the lines
of the discussion of the $\rho$-invariant in \cite{CG}.

\subsection{Vanishing of $\eta$-Invariants}
\label{susevan}
Let $Z$ belong to the center of $\mathfrak n$.
Then $Z$ commutes with $\Gamma$ and pushes forward
to a vector field on the quotient $\Gamma\backslash N$, also denoted by $Z$.

\begin{lem}\label{lemcom}
Clifford multiplication with $Z$ commutes with $(\bar D^+_t)^2$.
\end{lem}

\begin{proof}
We can assume that $Z$ has norm one.
Choosing $X_2=Z$,
then, in the second sum on the right in \eqref{tilded2} above,
the terms with $i=2$ vanish since $Z$ commutes with all the $X_i$, $i>2$.
\end{proof}

\begin{thm}\label{vaneta}
If the center of $N$ has dimension at least two,
then the spectrum of $\bar D^+_t$, including multiplicities,
is symmetric about zero.
In other words, the eta function of $\bar D^+_t$ vanishes identically.
\end{thm}

\begin{proof}
Choose orthonormal vector fields $Z$ and $Z'$ in the center of $\mathfrak n$
and let $W_\pm$ be the eigenspaces of the involution $iZ$
in $\Sigma_{\mathfrak s}^+\otimes V$ for the eigenvalues $\pm1$.
Since $(\bar D^+_t)^2$ commutes with $iZ$, see \lref{lemcom},
it leaves the spaces of sections
with values in $W_+$ and $W_-$ invariant.
In particular, if $\lambda>0$ is an eigenvalue of $(\bar D^+_t)^2$
and $\mathcal S(\lambda)$ denotes the corresponding eigenspace of sections,
then
\[
  \mathcal S(\lambda)
  = \mathcal S_+(\lambda) \oplus \mathcal S_-(\lambda) ,
\]
where $\mathcal S_+(\lambda)$ and $\mathcal S_-(\lambda)$
consist of eigensections in $\mathcal S(\lambda)$
with values in $W_+$ and $W_-$, respectively.

We note that $\mathcal S(\lambda)$ is invariant under $\bar D^+_t$ and
that $\bar D^+_t$ has eigenvalues $\pm\sqrt\lambda$ on $\mathcal S(\lambda)$.
Furthermore, the multiplicities of $\sqrt\lambda$ and $-\sqrt\lambda$
as eigenvalue of $\bar D^+_t$ coincide if and only if the trace of $\bar D^+_t$
on $\mathcal S(\lambda)$ vanishes.

We let $X_2=Z$.
Then $X_iW_+=W_-$ and $X_iW_-=W_+$ for $3\le i\le m$,
and hence the corresponding terms of $\bar D^+_t$
do not contribute to the trace of $\bar D^+_t$  on $\mathcal S(\lambda)$.
Now the remaining term $X_2\cdot X_2(\sigma)=Z\cdot Z(\sigma)$
of $\bar D^+_t\sigma$ leaves $\mathcal S(\lambda)$ invariant,
and its trace on $\mathcal S(\lambda)$ is equal to the trace of $\bar D^+_t$
on $\mathcal S(\lambda)$, by what we just said.

Clifford multiplication with $Z'$ leaves $\mathcal S(\lambda)$
invariant, by \lref{lemcom}.
On the other hand,
\[
  Z\cdot Z(Z'\cdot\sigma)
  = Z\cdot(Z'\cdot Z(\sigma))
  = - Z'\cdot(Z\cdot Z(\sigma)) ,
\]
that is, the involution $iZ'$ anticommutes with the operator
which sends $\sigma$ to $Z\cdot Z(\sigma)$.
It follows that the trace of $\bar D^+_t$ on $\mathcal S(\lambda)$ vanishes.
\end{proof}

\begin{cor}\label{vaneta2}
If the center of $N$ has dimension at least two,
then the asymptotic high energy $\eta$-invariant
$\lim_{t\to\infty}\eta(A_t^{\rm he,+})=0$.
\end{cor}

\begin{proof}
Recall that $A_t^+=-D_t^+$
and apply Theorems \ref{higeta} and \ref{vaneta}.
\end{proof}

\section{Heisenberg Manifolds}
\label{sushei}

The only simply connected two-step nilpotent Lie groups not
covered by \tref{vaneta} are the standard Heisenberg groups $N=G_n$,
where here $m-1 = \dim N = 2n+1$.
We represent $G_n$ as $\R^n\times\R^n\times\R$ with typical element $(x,y,z)$
and multiplication given by
\begin{equation}\label{aheimul}
  (x,y,z)(x',y',z') = (x+x',y+y',z+z'+ xy') .
\end{equation}
The left-invariant vector fields
\begin{equation}\label{aheibas}
 X_j := \frac{\partial}{\partial x_j} , \quad
 Y_j := \frac{\partial}{\partial y_j} + x_j\frac{\partial}{\partial z} ,
 \quad\text{and}\quad
 Z := \frac{\partial}{\partial z}
\end{equation}
form a basis of the Lie algebra of $G_n$.
They commute pairwise, except for the $n$ Lie brackets $[X_j,Y_j]=Z$.

\subsection{Lattices in Heisenberg groups}\label{sechei}

Lattices in $G_n$ are classified in \cite[Section 2]{GW}:
Let $D_n$ be the set of $n$-tupels $d=(d_1,\dots,d_n)$
of natural numbers such that $d_{i}$ divides $d_{i+1}$, $1\le i<n$.
Then, for any $d\in D_n$,
\begin{equation}\label{aheilat}
   \Gamma_{d}
   := \{ (x,y,z) \mid x,y\in\Z^n,z\in\Z,\text{$d_i$ divides $x_i$}\}
\end{equation}
is a lattice in $G_n$.
The isomorphism type of $\Gamma_d$ is determined by $d$
and, up automorphism of $G_n$,
any lattice $\Gamma$ in $G_n$ is equal to some $\Gamma_d$, $d\in D_n$.
Following the notation in \cite{GW}, we then set
\begin{equation}
  | \Gamma | := d_1 \cdots d_n .
  \label{type}
\end{equation}
Fix $d\in D_n$.
The $2n+1$ elements
\begin{equation}\label{aheilat2}
   \phi_j := (d_je_j,0,0) , \quad \psi_j := (0,e_j,0) , \quad \zeta := (0,0,1)
\end{equation}
generate $\Gamma_d$.
They commute pairwise, except for the $n$ relations
\begin{equation}\label{aheilat4}
   \phi_j\psi_j\phi_j^{-1}\psi_j^{-1} = \zeta^{d_j} = (0,0,d_j) .
\end{equation}

Let $\tau$ be an irreducible unitary representation of $\Gamma_d$
on a finite dimensional Hermitian vector space $V$.
Since $\tau$ is irreducible and $\zeta$ is central,
there is a number $c\in[0,1)$ with
\begin{equation}
  \tau(\zeta) = e^{2\pi ic} I .
  \label{twistpara}
\end{equation}
Let  $A_j:=\tau(\phi_j)$ and $B_j:=\tau(\psi_j)$, for $1\le j\le n$.
Then, if $\lambda$ is an eigenvalue of $B_j$,
for some $j$ and some eigenvector $v\in V$,
then
\begin{equation}
  B_j(A_jv)
  = e^{-2\pi icd_j}(A_jB_jA_j^{-1})(A_jv)
  = e^{-2\pi icd_j}\lambda A_jv ,
\end{equation}
and hence $e^{-2\pi icd_j}\lambda$ is an eigenvalue of $B_j$ as well.
It follows that $c$ is rational, by the finite dimensionality of $V$.

Let $m_j$ be the denominator of $cd_j$.
Consider the sublattice $\Gamma_{md}\subseteq\Gamma_d$,
where $md:=(m_1d_1,...,m_nd_n)$.
The order of $\Gamma_{md}$ in $\Gamma_d$ is
\begin{equation}
  | \Gamma_d/\Gamma_{md} | = m_1\cdots m_n ,
\end{equation}
and $\tau$ restricts to an Abelian representation on $\Gamma_{md}$.
By irreducibility, $\tau$ is induced from a one-dimensional representation
of $\Gamma_{md}$.
That is, there are real numbers $\alpha_1,\beta_1,\dots,\alpha_n,\beta_n$
such that $\phi_j^{m_j}$ and $\psi_j$ act on $\C$
by multiplication with $e^{2\pi i\alpha_j}$ and $e^{2\pi i\beta_j}$, respectively,
and $\tau$ is induced from this representation of $\Gamma_{md}$.
In particular,
\begin{equation}
  \dim V = m_1\cdots m_n .
  \label{dimu}
\end{equation}
For any $n$-tuple
\begin{equation}
  b = (b_1,\dots,b_n)
  = (\beta_1+l_1cd_1,\ldots,\beta_n+l_ncd_n) \in \R^n/\Z^n ,
\end{equation}
where $(l_1,\dots,l_n)\in\Z^n$, we let $V_{b}$ be the subspace of $V$
on which $\psi_j$ acts by $e^{2\pi ib_j}$.
We note that these subspaces $V_{b}$ are one-dimensional
and pairwise orthogonal and that they span $V$.

\subsection{Twisted Right Regular Representation}\label{repmul}
The set
\begin{equation}\label{heifunx}
  F := \{ (x,y,z) \in G_n \mid x\in P, (y,z)\in Q \} ,
\end{equation}
where
\begin{equation}\label{heifuny}
\begin{split}
  P &:= \{ x \in \R^n \mid 0 \le x_j \le d_j \} , \\
  Q &:= \{ (y,z) \in \R^n\times\R \mid 0 \le y_j,z \le 1 \} ,
\end{split}
\end{equation}
is a fundamental domain of the action of $\Gamma_d$ on $G_n$
by left translations.
Observe that, by \eqref{aheibas}, the standard Lebesgue
measure with respect to the $(x,y,z)$-coordinates is left-invariant,
hence bi-invariant, on $G_n$.

Fix an irreducible unitary representation $\tau$ of $\Gamma_d$
on a finite dimensional Hermitian vector space $V$ as above.
and consider the Hilbert space $L^2(\tau)$
of maps $\sigma:G_n\to V$ such that
\begin{equation}\label{tautran}
  \sigma(\gamma g) = \tau(\gamma)\sigma(g)
\end{equation}
for all $\gamma\in\Gamma_d$ and $g\in G_n$
which are square integrable over $F$.
The {\em right regular representation} $\rho$ of $G_n$
acts unitarily on $L^2(\tau)$ by
\begin{equation}\label{arigreg}
  (\rho(g)\sigma)(x,y,z) = \sigma((x,y,z)g) ,
\end{equation}
and our next aim is to determine the multiplicities of the irreducible
unitary representations of $G_n$ in $L^2(\tau)$.
Here we recall that irreducible unitary representations of the
Heisenberg group $G_n$ correspond to coadjoint orbits of $G_n$,
by the classical theorem of Stone and von Neumann
(or by the more general Kirillov theory, respectively).
This correspondence will show up in the following discussion.

Let $\sigma\in L^2(\tau)$. Then
\begin{equation}\label{tautranz}
  e^{2\pi ic} \sigma(x,y,z)
  = \tau(\zeta) \sigma(x,y,z)
  = \sigma(x,y,z+1) .
\end{equation}
Let $\sigma_b$ be the component of $\sigma$ in $V_b$.
Then
\begin{equation}\label{tautrany}
  e^{2\pi ib_j} \sigma_b(x,y,z)
  = B_j \sigma_b(x,y,z)
  = \sigma_b(x,y+e_j,z) .
\end{equation}
The transformation rule with respect to $A_j$ is more complicated,
\begin{equation}\label{tautranyx}
  A_j \sigma_b(x,y,z)
  = \sigma_{b-cd_je_j}(x+d_je_j,y,z+d_jy_j) .
\end{equation}
By \eqref{tautranz} and \eqref{tautrany},
we can develop $\sigma_b$ in a Fourier series,
\begin{equation}\label{taufou}
  \sigma_b(x,y,z)
  = \sum_{\substack{v\equiv b \\ w \equiv c}}
  e^{2\pi i(vy + wz)} \sigma_{v,w}(x)
\end{equation}
where $\equiv$ indicates congruence modulo $\Z^n$.
Fix a $w$ congruent to $c$ and consider the space $L^2(\tau,w)$
of $\sigma\in L^2(\tau)$ with
\begin{equation}\label{tautranw}
  \sigma(x,y,z + t) = e^{2\pi iwt} \sigma(x,y,z) ,
\end{equation}
that is, in the above Fourier development of the components
$\sigma_b$ of $\sigma$, only the terms with the given $w$ occur.
We obtain an orthogonal decomposition
\begin{equation}
  L^2(\tau) = \oplus_{w\equiv c} L^2(\tau,w) .
\end{equation}
Now the spaces $L^2(\tau,w)$ are $\rho$-invariant and,
therefore, it remains to investigate $\rho$ on them.
For $\sigma\in L^2(\tau,w)$, we have
\begin{equation}\label{foutran}
\begin{split}
  \sum_{u\equiv b+cd_je_j} &e^{2\pi i(uy+wz)} A_j\sigma_{u,w}(x)
  = A_j \sigma_{b+cd_je_j}(x,y,z) \\
  &= \sigma_{b}(x+d_je_j,y,z+d_jy_j) \\
  &= e^{2\pi iwd_jy_j} \sigma_{b}(x+d_je_j,y,z) \\
  &= \sum_{v\equiv b}e^{2\pi i((v+wd_je_j)y+wz)} \sigma_{v,w}(x+d_je_j) .
\end{split}
\end{equation}
We conclude that, for any $v\equiv b$ and $x\in\R^n$,
\begin{equation}\label{foutran2}
  \sigma_{v+wd_je_j,w}(x) = A_j^{-1} \sigma_{v,w}(x+d_je_j) .
\end{equation}
There are two cases, $w=0$ and $w\ne0$, respectively.

If $w=0$, then $w=c=0$ and $\dim V=1$.
By \eqref{foutran2}, the Fourier coefficients $\sigma_{v,0}$
are $d_j$-periodic in $x_j$ up to the twists
by the complex numbers $A_j$ of norm one.

Suppose now that $w\ne0$.
Then, by \eqref{foutran2},
the Fourier coefficients $\sigma_{u,w}$ with
\begin{equation}\label{founum}
  u = b + k_1e_1 + \dots + k_ne_n , \quad
  0 \le k_j < |w|d_j ,
\end{equation}
determine all the Fourier coefficients of $\sigma$.
We also get
\begin{equation}\label{founor}
  || \sigma ||_{L^2(\tau,w)}^2
  = \sum || \sigma_{u,w} ||_{L^2(\R^n,V_b)}^2 ,
\end{equation}
where the sum is over all $u$ as in \eqref{founum}.
Here we recall that, on the left hand side,
the $L^2$-norms are given by the corresponding integrals
over the fundamental domain $F$ of $\Gamma_d$
as in \eqref{heifunx},
whereas the integrals on the right hand side are over
Euclidean $x$-space.
We obtain
\begin{equation}
  L^2(\tau,w) \cong \oplus L^2(\R^n,V_b) ,
\end{equation}
where we have $m_1d_1\cdots m_nd_n|w|^n$ summands
$L^2(\R^n,V_b)$ on the right hand side,
namely one for each $u$ as in \eqref{founum}.

To identify $\rho$ on $\oplus L^2(\R^n,V_b)\cong L^2(\tau,w)$,
let $g=(x',y',z')\in G_n$ and recall \eqref{aheimul} and \eqref{arigreg}.
We compute
\begin{multline}
  e^{2\pi i(u(y+y') + w(z+z'+xy'))} \sigma_{u,w}(x+x') \\
  = e^{2\pi i(uy + wz)} e^{2\pi i(uy' + w(z'+xy'))} \sigma_{u,w}(x+x') ,
\end{multline}
hence $g$ acts on $\sigma_{u,w}\in L^2(\R^n,V_b)$ by
\begin{equation}
  (\rho(g)\sigma_{u,w})(x)
  = e^{2\pi i(uy' + w(z'+xy'))} \sigma_{u,w}(x+x') .
\end{equation}
Via a unitary  identification $V_b\cong\C$
and the substitution $x+u/c$ for $x$,
we see that $\rho$ on $L^2(\R^n,V_b)$ is unitarily equivalent to the
irreducible unitary representation $\rho_w$ of $G_n$ on $L^2(\R^n,\C)$
with
\begin{equation}\label{rhow}
  (\rho_w(g)f)(x) = e^{2\pi iw(z'+xy')}f(x+x') .
\end{equation}
This is the standard representation of $G_n$ associated to the
coadjoint orbit of linear functionals on the Lie algebra of $G_n$
which send $Z$ to $w$.
Hence $L^2(\tau,w)$ is a corresponding isotypical component
of $\rho_w$ in $L^2(\tau)$.
By \eqref{founum} and \eqref{founor},
the multiplicity of $\rho_w$ in $L^2(\tau)$ and $L^2(\tau,w)$ is
\begin{equation}\label{mulmul}
  m_1|w|d_1\cdots m_n|w|d_n = |\Gamma| \dim V |w|^n .
\end{equation}

\subsection{Spectrum of Twisted Laplacians}\label{replap}
Let $w\ne0$.
To determine the spectrum of the Laplacian $\Delta_w$
of a given left-invariant Riemannian metric on $L^2(\R^n,\C)$
with respect to the representation $\rho_w$ as in \eqref{rhow},
we follow the discussion in the proof of Theorem 3.3 of \cite{GW}:
With respect to the given metric,
there is an orthonormal basis
\begin{equation}\label{heibaso}
  X_1', \dots , X_n' , Y_1' , \dots Y_n' , Z'
\end{equation}
of the Lie algebra of $G_n$ with $Z'=rZ$, $r=1/|Z|>0$, such that
\begin{equation}\label{heibaso2}
  [X_j',X_k'] = [X_j',Y_k'] = [Y_j',Y_k'] = 0
\end{equation}
for all $j\ne k$ and such that there are numbers $r_j>0$ with
\begin{equation}\label{heibaso4}
  [X_j',Y_j'] = r_j^2Z .
\end{equation}
The pull back of the metric under the automorphism $\Phi$
of $G_n$ with
\begin{equation}
  \Phi_*(r_jX_j) = X_j' ,  \quad
  \Phi_*(r_jY_j) = Y_j' ,  \quad
  \Phi_*(Z) =Z
\end{equation}
is the left-invariant Riemannian metric on $G_n$ for which the fields
\begin{equation}\label{normet}
  r_1X_1 , \dots , r_nX_n , r_1Y_1 , \dots , r_nY_n , rZ
\end{equation}
are orthonormal.
Since $\Phi_*(Z)=Z$,
$\rho_w\circ\Phi$ is still an irreducible unitary representation of $G_n$
associated to the coadjoint orbit of linear functions on the Lie algebra
of $G_n$ which send $Z$ to $w$,
hence $\rho_w\circ\Phi$ is unitarily equivalent to $\rho_w$.
In other words,
we can assume without loss of generality that the given left-invariant
Riemannian metric on $G_n$ has an orthonormal basis as in \eqref{normet}.
As for the Laplacian  on $L^2(\R^n,\C)$ with respect to $\rho_w$,
we obtain
\begin{equation}\label{laprhow}
  \Delta_w
  = -\sum_{1\le j\le n} r_j^2\frac{\partial^2}{\partial x_j^2}
  + 4\pi^2w^2 \big(r^2 + \sum_{1\le j\le n} x_j^2r_j^2 \big) ,
\end{equation}
by \eqref{rhow}.
Now the Hermite functions
\begin{equation}
  h_p(x) = \exp(x^2/2) \frac{\partial^{p_1+\cdots+p_n}}
  {\partial x_1^{p_1}\cdots\partial x_n^{p_n}}\exp(-x^2) ,
\end{equation}
where $p=(p_1,\dots,p_n)$ runs over all $n$-tuples of non-negative integers,
form an orthogonal basis of $L^2(\R^n,\C)$ and satisfy
\begin{equation}
  x_j^2h_p - \frac{\partial^2h_p}{\partial x_j^2} = (2p_j+1) h_p .
\end{equation}
It follows that the functions $f_p(x)=h_p(\sqrt{2\pi|w|}x)$ are an orthogonal
basis of $L^2(\R^n,\C)$ and that they satisfy
\begin{equation}
  \Delta_w f_p =\lambda(w,p) f_p ,
\end{equation}
where
\begin{equation}\label{wlapeig}
  \lambda(w,p)
  := 4\pi^2w^2r^2 + 2\pi |w|
  \sum_{1\le j\le n} (2p_j + 1) r_j^2 .
\end{equation}
Thus, by \eqref{mulmul},
the multiplicity of $\lambda(w,p)$ in $L^2(\tau,w)$ is equal to
\begin{equation}\label{mulmul2}
  d_1\cdots d_n m_1 \cdots m_n |w|^n = |\Gamma| \dim V |w|^n ,
\end{equation}
when counted according to the $n$-tuples $p$.

In our application of the above in the proof of \tref{etahei},
we will vary the parameter $r=1/|Z|$ of the metric, keeping
\begin{equation*}
  r_1X_1, \dots , r_nX_n , r_1Y_1, \dots , r_nY_n
\end{equation*}
orthonormal and perpendicular to $Z$.
Then the above functions $f_p$ remain eigenfunctions
of $\Delta_w$ in $L^2(\R^n,\C)$
and the corresponding eigenvalues vary according
to \eqref{wlapeig}.
Hence the eigensections in $L^2(\tau,w)$
corresponding to the above eigenfunctions $f_p$
remain the same during this variation of the metric
and the corresponding eigenvalues vary according
to \eqref{wlapeig} as well.

\subsection{$\eta$-Invariants for Heisenberg Manifolds}
\label{heieta}
In this section,
we study the $\eta$-invariant of the operator $\bar D_t^+$ as in \eqref{tilded}.
The solvable extension $S$ of $N=G_n$ as in Chapter \ref{sechc}
and the connection $\nabla^E$ do not enter in this discussion.
We recall though that $\Sigma_{\mathfrak s}^+\simeq\Sigma_{\mathfrak n}$,
where $\mathfrak n$ denotes the Lie algebra of $G_n$
and where Clifford multiplication with $X$ in  $\Sigma_{\mathfrak n}$
corresponds to Clifford multiplication with $TX$ in  $\Sigma_{\mathfrak s}^+$,
for all $X\in\mathfrak n$.
This should be kept in mind, see e.g. \eqref{bard}.

Let $\Gamma$ be a lattice in $G_n$ of type $d$.
It is clear from \eqref{aheilat} that there is a smallest $s>0$
such that $\zeta:=\exp(s^2Z)$ is contained in $\Gamma$
and that $\zeta$ is a generator of the center of $\Gamma$.
The automorphism $\Phi(x,y,z)=(sx,sy,s^2z)$ of $G_n$
maps $\exp Z$ to $\zeta$, and, therefore, we may assume that
\begin{equation}\label{norgam}
  \zeta = \exp Z
\end{equation}
generates the center of $\Gamma$.
For any left-invariant Riemannian metric on $G_n$,
$N=\Gamma\backslash G_n$ is a Riemannian submersion
over a flat torus with closed geodesics as fibers,
given as orbits of the one-parameter group generated by $Z$.
By our normalization \eqref{norgam},
the length of the fibers is given by $|Z|$.

Let $\tau$ be an irreducible unitary representation of $\Gamma_d$
on a finite dimensional Hermitian vector space $V$
as in Section \ref{sechei} and extend $\tau$ by the trivial representation
on $\Sigma_{\mathfrak n}$ to $\Sigma_{\mathfrak n}\otimes V$
as in Chapter \ref{sechc}.
Recall from Section \ref{sechei} that $\zeta$ acts by
multiplication with $\exp(2\pi ic)$ for some $c=c(\tau)\in[0,1)\cap\Q$
and that
\begin{equation}
  \dim V = \delta(c,d) := m_1 \cdots m_n ,
\end{equation}
where $d=d(\Gamma)$ and $m_j$ is the denominator of $cd_j$.
In the notation of this chapter,
and in terms of an orthonormal frame $(E_j)$ of $G_n$,
we study the unbounded operator
\begin{equation}
  \bar D\sigma = \sum E_j \cdot E_j(\sigma) ,
  \label{bard}
\end{equation}
in the Hilbert space $L^2(\tau)$ of measurable maps
$G_n\to\Sigma_{\mathfrak n}\otimes V$ satisfying \eqref{twist}
which are square integrable over a fundamental domain
of $\Gamma=\Gamma_d$ in the Heisenberg group $G_n$.

Before stating the next result,
we recall the definition of the {\em Hurwitz zeta function},
for $c>0$ and $\Re s>1$ given by the infinite sum
\begin{equation}\label{zetaw}
  \zeta_c(s) = \zeta(s,c) := \sum\nolimits_{k\ge0}(k+c)^{-s} .
\end{equation}
We have $\zeta_1=\zeta$, the {\em Riemann zeta function}.
We also set $\zeta_0:=\zeta$.
For each $c\ge0$, $\zeta_c$ can be extended to a meromorphic
function on the complex plane, defined for all $s\ne1$,
and with a simple pole at $s=1$,
where the residue is equal to $1$.

It is maybe interesting to note that, for $0<c<1$,
\begin{equation}\label{etazeta}
  \zeta_c(s)-\zeta_{1-c}(s)
  \quad\text{and}\quad
  \zeta_c(2s)+\zeta_{1-c}(2s)
\end{equation}
are the eta and zeta function
of the operator $id/dt$ and $-d^2/dt^2$, respectively,
on the Hermitian line bundle over $\R/2\pi\Z$
with twist $e^{-2\pi ic}$.

\begin{thm}\label{etahei}
Endow $G_n$ with a left-invariant Riemannian metric,
let $\Gamma$ be a lattice in $G_n$ such that $\zeta=\exp Z$
generates the center of $\Gamma$,
and set $r:=1/|Z|$.
Consider a Clifford module $\Sigma_{\mathfrak n}\otimes V$ as above
and let $c=c(\tau)$.
Then we have, for all $s\in\C$ with sufficiently large real part,
\begin{align}
  \eta(\bar D,s)
  &= | \Gamma | \dim V (2\pi r)^{-s}
  (\zeta_c(s-n) - \zeta_{1-c}(s-n))
  \tag{1}
\intertext{if $n$ is even}
  \eta(\bar D,s)
  &= - | \Gamma | \dim V (2\pi r)^{-s}
  (\zeta_c(s-n) + \zeta_{1-c}(s-n))
  \tag{2}
\end{align}
if $n$ is odd.
\end{thm}

We conclude that, under the assumptions of the above theorem,
the eta function of $\bar D$ is holomorphic if $n$ is even and is
meromorphic with a simple pole at $s=n+1$ if $n$ is odd.
We also see that the $\eta$-invariant $\eta(\bar D)=\eta(\bar D,0)$
of $\bar D$ only depends on $n$, the type of $\Gamma$, and $c$.

\begin{proof}[Proof of \tref{etahei}]
The main argument in the proof is modeled
along the lines of the proof of Proposition 4.1 of \cite{DS}.
For $w\equiv c$ modulo integers, we let
\begin{equation}
  L^2(\tau,w) := \{ \sigma \in L^2(\tau) :
  \sigma(x,y,z+t) = e^{2\pi iwt}\sigma(x,y,z) \}
\end{equation}
and get an orthogonal decomposition
\begin{equation}
  L^2(\tau) = \oplus_{w\equiv c} L^2(\tau,w) ,
\end{equation}
where $L^2(\tau)$ is the Hilbert space of measurable maps
$G_n\to\Sigma_{\mathfrak n}\otimes V$ satisfying \eqref{twist}
which are square integrable over a fundamental domain
of $\Gamma=\Gamma_d$ in $G_n$ as above.
Since the spaces $L^2(\tau,w)$ are invariant under $\bar D$,
the eta function of $\bar D$ is the sum of the eta functions
of the restrictions of $\bar D$ to the different $L^2(\tau,w)$.
Thus we can consider the latter separately.

There are two cases, $w=0$ and $w\ne 0$.
As for $w=0$,
we note that $Z(\sigma)=0$ for any $\sigma\in L^2(\tau,0)$.
Hence the unitary involution $\omega_0$ of $L^2(\tau,0)$
given by Clifford multiplication with $irZ$ anti-commutes with $\bar D$.
Hence the spectrum of $\bar D$ is symmetric about $0$,
and, therefore, the eta function of $\bar D$ on $L^2(\tau,0)$
vanishes identically.

Suppose now that $w\ne0$.
We want to apply the results from the beginning of this capter
and note, to that end, that the spaces $L^2(\tau)$ and $L^2(\tau,w)$ here
are isomorphic to the corresponding spaces in Section \ref{repmul},
tensored  with $\Sigma_{\mathfrak n}$.

It follows from the discussion in Section \ref{replap} that,
except for the determination of multiplicities,
the particular lattice does not enter into the discussion.
By what we explain in Section \ref{replap},
we can assume that
\begin{equation}
  r_1X_1, r_1Y_1 , \dots , r_nX_n, r_nY_n , rZ
\end{equation}
is an orthonormal basis of the given left-invariant metric on $G_n$.
Then \eqref{bard} turns into
\begin{align}\label{bard2ei}
  \bar D\sigma &= \sum_{1\le j\le n}
  r_j^2 \big(X_j\cdot  X_j(\sigma) + Y_j\cdot Y_j(\sigma) \big) + r^2 Z\cdot Z(\sigma) ,
\intertext{and \eqref{tilded2} turns into}
  \bar D^2\sigma
  &= \Delta\sigma + \sum_{1\le j\le n} r_j^4X_jY_j \cdot Z(\sigma) ,
\end{align}
where $\sigma\in L^2(\tau,w)$ is smooth.

We let $\omega_j$, $1\le j\le n$, be the unitary involutions
on $\Sigma_{\mathfrak n}\otimes V$ and $L^2(\tau,w)$
given by Clifford multiplication with $ir_j^2X_jY_j$,
respectively.
Then
\begin{equation}
  \Sigma_{\mathfrak n}
  = \oplus_{\varepsilon\in\{1,-1\}^n} \Sigma_{\varepsilon} ,
\end{equation}
where
\begin{equation}
  \Sigma_\varepsilon
  = \{ \sigma \in \Sigma_{\mathfrak n} :
  \text{$\omega_j\sigma = \varepsilon_j\sigma$ for all $1 \le j \le n$} \} .
\end{equation}
Now the unitary involutions $\omega_j$ commute with $\Delta$.
Thus on
\begin{equation}
  L^2(\tau,w,\varepsilon)
  := \{ \sigma \in L(\tau,w) :
  \text{$\sigma$ has values in $\Sigma_\varepsilon\otimes V$} \} ,
\end{equation}
$\bar D^2$ has eigenvalues
\begin{equation}\label{heifd2ei}
\begin{split}
  \lambda(w,p,\varepsilon)
  &= \lambda(w,p)
  + 2\pi w(r_1^2\varepsilon_1+\dots + r_n^2\varepsilon_n) \\
  &= 4\pi^2w^2r^2
  + 2\pi|w|\sum_{1\le j\le n}(2p_j+1+\varepsilon_j\sign w)r_j^2
\end{split}
\end{equation}
with multiplicity $2^nm_1d_1\cdots m_nd_n|w|^n$,
where $p$ runs over all $n$-tuples of non-negative integers,
by \eqref{wlapeig} and \eqref{mulmul2}.
For all $p$, we have
\begin{equation}
  \lambda(w,p,\varepsilon)\ge 4\pi^2w^2r^2>0 .
\end{equation}

Let $W$ be an eigenspace of $\bar D^2$ in $L^2(\tau,w)$
for the eigenvalue $\lambda$,
and recall from Section \ref{replap} that $W$ is independent
of the parameter $r$ of the metric.
Since $\bar D^2$ commutes with the involutions $\omega_j$,
$W$ has an orthonormal basis consisting of eigensections of $\bar D^2$
such that each of them belongs to some $L^2(\tau,w,\varepsilon)$,
where $p$ and $\varepsilon$ satisfy
\begin{equation}
  S := 2\pi|w| \sum_{1\le j\le n}(2p_j+1+\varepsilon_j\sign w) r_j^2
  = \lambda - 4\pi^2w^2r^2 ,
\end{equation}
by \eqref{heifd2ei}.
Now Clifford multiplication by the unit vector $rZ$ commutes with $\bar D^2$
and leaves the subspaces $L^2(\tau,w,\varepsilon)$ invariant,
whereas Clifford multiplication by the unit vectors $r_jX_j$
and $r_jY_j$ maps $L^2(\tau,w,\varepsilon)$ to $L^2(\tau,w,\delta)$
for $\delta\ne\varepsilon$.
Hence using an orthonormal basis of eigensections of $W$ as above,
we see that the trace of $\bar D$ on $W$ is equal to an integral
multiple $k2\pi wr$ of $2\pi wr$.
On the other hand, the trace of $\bar D$ on $W$
is also equal to $l\sqrt\lambda$ for some integer $l$.
Now $0$ is not an eigenvalue of $\bar D^2$ on $W$,
independently of $r>0$.
Hence $k$ and $l$ do not depend on $r$,
and we get an equality of functions of $r\in(0,\infty)$,
\begin{equation}
  k^24\pi^2w^2r^2 = l^2(4\pi^2w^2r^2 + S)^2 .
\end{equation}
If $l=0$, then the eigenvalues $\pm\sqrt\lambda$ of $\bar D$
occur with equal multiplicity in $W$ and, therefore, their contributions
to the eta function of $\bar D$ on $L^2(\tau,w)$ cancel.
If $l\ne0$, then $S=0$, since $S$ does not depend on $r$.
But then, since $w\ne0$, $p_j\ge0$, and $\varepsilon_j=\pm1$ for all $j$,
we conclude that $\lambda(w,p,\varepsilon)=4\pi^2w^2r^2$ and that
\begin{equation}\label{heifdei6}
  p_1 = \dots = p_n = 0
  \quad\text{and}\quad
  \varepsilon_1 = \dots = \varepsilon_n = - \sign w
\end{equation}
for $1\le j\le n$.
This will be denoted by $p=0$ and $\varepsilon=-\sign w$.

To determine the contribution of the corresponding eigenspaces,
we note that,
by our identification $\Sigma_{\mathfrak n}=\Sigma_{\mathfrak s}^+$,
Clifford multiplication by $irZ\omega_1\cdots\omega_n$
is equal to the identity on $\Sigma_{\mathfrak n}$.
Since Clifford multiplication with $irZ$ commutes
with Clifford multiplication with the $\omega_j$,
it leaves the subspaces $\Sigma_{\varepsilon}$ invariant
and acts by multiplication with $\varepsilon_1\cdots\varepsilon_n$
on them.
Now  $Z(\sigma)=2\pi iw\sigma$ for any $\sigma$ in $L^2(\tau,w)$.
Hence the eigenspace for $\bar D^2$ in $L^2(\tau,w)$ with
eigenvalue $\lambda(w,0,-\sign w)=4\pi^2w^2r^2$
is an eigenspace of $\bar D$ with eigenvalue
\begin{equation}
  \begin{split}
  2\pi wr  &\quad\text{if $n$ is even} , \\
  - 2\pi |w| r  &\quad\text{if $n$ is odd} ,
  \end{split}
\end{equation}
and dimension $m_1\dots m_nd_1\dots d_n|w|^n=|\Gamma|\dim V$.
Thus, for all $s\in\C$ with sufficiently large real part,
\begin{equation}\label{heifdetae}
  \eta(\bar D,s)
  = |\Gamma | \dim V (2\pi r)^{-s}
  \sum_{w\equiv c ,\, w \ne 0} \sign(w)|w|^{n-s}
\end{equation}
if $n$ is even and
\begin{equation}\label{heifdetao}
  \eta(\bar D,s)
  = - | \Gamma | \dim V (2\pi r)^{-s}
  \sum_{w\equiv c ,\, w \ne 0} |w|^{n-s}
\end{equation}
if $n$ is odd.
\end{proof}

We apply the results of this chapter to Dirac operators on
homogeneous vector bundles over complex hyperbolic cusps
of complex dimension $n$.
Such cusps are homogeneous in the sense of Chapter \ref{sechc},
where the nilpotent Lie group is given by the Heisenberg group $N=G_{n-1}$
of dimension $2n-1$ and $\Gamma\subseteq G_{n-1}$ is a lattice.
In our formulas above we therefore need to substitute $n$ by $n-1$.

\begin{cor}\label{etacor}
In the sense of Chapter \ref{sechc}, suppose that a complex hyperbolic cusp
is determined by a lattice $\Gamma\subseteq G_{n-1}$ and that
the homogeneous Dirac bundle over the cusp is given
by unitary representations $\pi_*$ of $\mathfrak u(n)$ and $\tau$ of $\Gamma$
on a Hermitian vector space $V$.
Assume that $V$ is irreducible as a joint $\mathfrak u(n)$ and $\Gamma$ module.
Then the twist parameter $c$ of $\tau$ is well defined and
\begin{equation*}
  \lim_{t\to\infty} \eta(A_t^{\rm he,+}) =
  (-1)^n | \Gamma | \dim V
  \big(\zeta_c(1-n) + (-1)^n \zeta_{1-c}(1-n)\big) .
\end{equation*}
\end{cor}

\begin{proof}
We recall that $A_t=-D_t$, see \eqref{defat}.
By \tref{higeta},
we have $\lim_{t\to\infty} \eta(D_t^{\rm he,+})=\lim_{t\to\infty} \eta(\bar D_t^+)$.
Now the operator $\bar D_t^+$ corresponds to the operator $\bar D$
considered above, where the left-invariant metric on $G_n$ comes from
the cross section $\{t\}\times N$ in $S$.
Since $V$ is irreducible as a joint $\mathfrak u(n)$ and $\Gamma$ module,
it is a direct sum of isotypical irreducible representations of $\Gamma$
as used in \tref{etahei} so that the number $c$ is the same for each summand.
Hence \tref{etahei} applies and shows
that $\eta(\bar D_t^+)$ does not depend on $t$
and that it is given by the formula in \cref{etacor}
\end{proof}

\begin{exa}\label{exaspin}
Spinor bundles as in Example \ref{exdibu}  are given by the
trivial representation of $\mathfrak u(n)$
and classified by twists $\tau:\Gamma_d\to\{+1,-1\}$.
Since $\tau(\zeta)=\pm 1$, we have $c=0$ or $c=1/2$.
Hence the asymptotic high energy $\eta$-invariant of $A_t^+$
vanishes identically if $n$ is odd.
If $n$ is even and $c=0$, then
\begin{equation}\label{heifdeta8}
  \lim_{t\to\infty} \eta(A_t^{\rm he,+}) = 2 |\Gamma | \zeta(1-n) ,
\end{equation}
which agrees with Proposition 4.1 in \cite{DS}
in the case $\Gamma=\Gamma_{(1,\dots,1)}$ considered there
(with a different choice of orientation).
If $n$ is even and $c=1/2$, then
\begin{equation}\label{heifdeta9}
  \lim_{t\to\infty} \eta(A_t^{\rm he,+}) = 2(2^{1-n} - 1) |\Gamma | \zeta(1-n) ,
\end{equation}
where we use that $\zeta_{1/2}(s)=(2^s-1)\zeta(s)$.
Recall also that
\begin{equation}
  \zeta(1-n) = - B_n/n ,
\end{equation}
where $B_n$ denotes the $n$-th Bernoulli number.
\end{exa}

\section{Low Energy $\eta$-Invariants}
\label{secle}

\subsection{General Remarks and Computations}
\label{susgrc}
We return to the situation and notation considered in Chapter \ref{sechc}
and let $E=\mathcal P\times_\beta(\Sigma_{\mathfrak s}\otimes V)$
be a homogeneous Dirac bundle over $S$.
As in Chapter \ref{sushei},
we view sections of $E^+$ as smooth maps
$\sigma:S\to\Sigma_{\mathfrak n}\otimes V$.

The vector field $T$ is left-invariant and a global unit normal field
along the hypersurfaces $N_t:=\{t\}\times N$ of $S$.
In accordance with this,
we choose frames $(X_1,\ldots,X_m)$ of $S$ to be left-invariant and
orthonormal with $X_1=T$.
Then $X_2,\ldots,X_m$ are tangent to the hypersurfaces $N_t$.

Let $\Gamma\subseteq N$ be a lattice,
$\tau:\Gamma\to V$ be a unitary representation,
and $E_\tau$ be the induced Dirac bundle
over $\Gamma\backslash S=\R\times(\Gamma\backslash N)$.
Then we have, for any $t\in\R$, the orthogonal decomposition
\begin{equation}
  L^2(N_t,E_\tau^+) = H^{\rm le,+}(A_t) \oplus H^{\rm he,+}(A_t) ,
\end{equation}
where $H^{\rm le,+}(A_t)$ is the space of constant maps
$N_t\to\Sigma_{\mathfrak n}\otimes V$,
compare Chapter \ref{secthi} and, in particular, \eqref{lecheh}.

\begin{prop}\label{triviata}
If $V$ is irreducible as a joint $\mathfrak k$ and $\Gamma$ module
and $\tau$ is non-trivial,
then the low energy spaces $H_t^{\rm le,+}(A_t)$ are trivial and,
therefore, $\eta(A_t^{\rm he,+})=0$, for all $t\in\R$.
\qed
\end{prop}

Thus the low energy $\eta$-invariant can only be non-trivial
when $\tau$ is trivial.
We refer to this as the {\em untwisted case}
and assume for the rest of this section that we are in this case,
whether $V$ is irreducible as a $\mathfrak k$ module or not.
Then the space $H^{\rm le,+}(A_t)$ is isomorphic
to $\Sigma_{\mathfrak n}\otimes V$,
by identifying constant maps with their respective values.

For $\sigma\in H^{\rm le,+}(A_t)$ and with $A_X$ as in \eqref{axy},
we have
\begin{equation}
  D_t^+\sigma = \sum_{2\le j\le m} TX_j \cdot \beta_*(A_{X_j})\sigma
  + \frac\kappa2 \sigma ,
  \label{hdbdt}
\end{equation}
by \eqref{dts},
where we recall our convention $X_1=T$.
Our objective in this chapter is the $\eta$-invariant of $D_t^+$
on $H^{\rm le}(A_t^+)$.
We view elements of $H^{\rm le,+}(A_t)$ as constant maps on $S$.
Then $H^{\rm le,+}(A_t)$ becomes independent of $\Gamma$ and $t$.
By \eqref{hdbdt}, $D_t^+$ does not depend on $t$ either.
As a shorthand, we will write
\begin{equation}
  \text{$H^{\rm le}_N$ for $H^{\rm le,+}(A_t)\quad$
  and $\quad D_N^{\rm le}$ for $D_t^+$ on $H^{\rm le}_N$.}
  \label{short}
\end{equation}
Recall that $\beta_*=\hat\alpha_*\otimes\id + \id\otimes\pi_*$,
by \eqref{hdb}, and that
\begin{equation}
  \hat\alpha_{*}(A_X) = \frac12\sum_{1\le j<k\le m}
  \langle\nabla_XX_j,X_k\rangle X_jX_k .
  \label{dspin}
\end{equation}
where $X_jX_k$ stands for Clifford multiplication by $X_jX_k$.
With our convention $X_1=T$, \eqref{dspin} turns into
\begin{equation}
  \hat\alpha_{*}(A_X)
  = \frac12T\nabla_XT + \frac12\sum\nolimits_{2\le j<k\le m}
  \langle\nabla_XX_j,X_k\rangle X_jX_k .
  \label{dspin2}
\end{equation}
It follows that \eqref{tast} and \eqref{nast} define
the Dirac structure on $E$ associated to the Riemannian metric of $N$.
As for the first term on the right hand side of \eqref{dspin2}, we note that
\begin{equation}
  \sum_{j\ge2}T\nabla_{X_j}T = - \kappa .
  \label{cancelkap}
\end{equation}
Choose an orthonormal frame $(X_2,\ldots,X_m)$ of $\mathfrak n$
such that $[X_j,X_k]$ is contained in the linear hull of the $X_l$
with $l<\min\{j,k\}$.
On $H^{\rm le}_N$, we then have
\begin{equation}
  D^{\rm le}_N - \sum_{j\ge2} TX_j\otimes\pi_*(A_{X_j})
  =  \frac12 \sum_{j\ge2\le k<l}TX_j
  \langle\nabla_{X_j}X_k,X_l\rangle X_kX_l ,
  \label{4dn1}
\end{equation}
by \eqref{hdbdt}, \eqref{dspin2}, and \eqref{cancelkap}.
By the Koszul formula and since $[X_j,X_k]$ is perpendicular to $X_l$ for $k<l$,
the right hand side is equal to
\begin{equation}
  -\frac14 \sum_{j\ge2\le k<l}TX_j
  \big(\langle X_j,[X_k,X_l]\rangle + \langle X_k,[X_j,X_l]\rangle \big) X_kX_l .
  \label{4dn2}
\end{equation}
Now $[X_j,X_l]$ is a linear combination of the $X_k$ for $k<l$,
hence the terms in \eqref{4dn2} can be rewritten as
\begin{equation}
  - \frac14 \sum_{2\le k<l}T[X_k,X_l]X_kX_l
  -\frac14 \sum_{2\le j,l}TX_j[X_j,X_l]X_l .
  \label{4dn3}
\end{equation}
Since $X_2$ is central, these terms are equal to
\begin{equation}
  \frac14 \sum_{2<j<k}T[X_j,X_k]X_jX_k
  = \frac 18 \sum_{j,k>2}T[X_j,X_k]X_jX_k .
  \label{4dn4}
\end{equation}
In conclusion,
\begin{equation}
  D^{\rm le}_N
  = \frac18 \sum_{j,k>2}T[X_j,X_k]X_jX_k
  + \sum_{j\ge2} TX_j\otimes\pi_*(A_{X_j}) ,
  \label{4dn}
\end{equation}
our main formula regarding $D^{\rm le}_N$.

For general nilpotent Lie groups $N$, it will probably be hard to get an explicit
formula for the low energy $\eta$-invariant.
We expect that a simple explicit formula does exist for groups of Heisenberg type
as in \cite{Ka}.
So far, we have been able to achieve this for the standard Heisenberg groups
as in Chapter \ref{sushei};
their low energy $\eta$-invariant will be the topic in Section \ref{suschc} below.
Recall that cusps of complex hyperbolic manifolds correspond to the case
where $N$ is the standard Heisenberg group
and that cusps of quaternionic and Cayley hyperbolic manifolds also give rise
to groups of Heisenberg type.

Suppose now that $N$ is of Heisenberg type.
That is, we are given an orthogonal decomposition
\begin{equation}
  \mathfrak n = \mathfrak z + \mathfrak x ,
  \label{hety}
\end{equation}
where $\mathfrak z$ is contained in the center of $\mathfrak n$,
and a linear map $J$ from $\mathfrak z$ into the space of skew-symmetric
endomorphisms of $\mathfrak x$ such that the Clifford relations hold,
\begin{equation}
  J_{Z_1}J_{Z_2} + J_{Z_2}J_{Z_1} + 2 \langle Z_1,Z_2\rangle = 0 ,
  \label{hety2}
\end{equation}
for all $Z_1,Z_2\in\mathfrak z$.
Moreover, the Lie brackets of vectors in $\mathfrak x$ are contained
in $\mathfrak z$ and satisfy, by definition,
\begin{equation}
   \langle [X_1,X_2],Z\rangle
   = 2c \langle J_ZX_1,X_2\rangle ,
   \label{hety4}
\end{equation}
for all $X_1,X_2\in\mathfrak x$ and $Z\in\mathfrak z$,
where $c>0$ is a fixed constant.
Compare with \eqref{chj}--\eqref{chht},
where we have the case $\dim\mathfrak z=1$ and $c=1$.

The constant $c>0$ in \eqref{hety4} is arbitrary, and the derivation $W$
is defined to have $\mathfrak x$ and $\mathfrak z$
as eigenspaces with  $-c$  and $-2c$ as respective eigenvalues.
This normalization has the following amazing formula as a consequence.

\begin{lem}\label{amaz}
For all $Z\in\mathfrak z$ and $X\in\mathfrak x$, we have
\begin{equation*}
  R(Z,X) = R(J_ZX,T) .
\end{equation*}
\end{lem}

\begin{rems}
1) In Lemma \ref{amaz}, we consider $R(Z,X)$ as an endomorphism of $\mathfrak s$,
the tangent space of $S$ at the neutral element,
but note that $R(Z,X)$ acts also on associated bundles.

2) If $N$ is the standard Heisenberg group of dimension $2n+1$,
then $S$ is isometric to the complex hyperbolic space $\C H^{n+1}$
of dimension $2n+2$ with sectional curvature in $[-4c^2,-c^2]$
and complex structure $J$ with $JT=Z$
and such that $J$ coincides with $J_Z$ on $N$.
In this case, the equation in \lref{amaz} is a special case of the more
general $R(JU,V)=-R(U,JV)$ which says that the curvature tensor
of $\C H^{n+1}$ is a differential form of type $(1,1)$.
\end{rems}

\begin{proof}[Proof of \lref{amaz}]
By straightforward computations,
using \eqref{dxt}, \eqref{rxt}, \eqref{hety2}, and \eqref{hety4},
\end{proof}

Let $Z\in\mathfrak z$ with $|Z|=1$.
Then $J_Z$ is an orthogonal complex structure on $\mathfrak x$.
In particular, the dimension of $\mathfrak x$ is even,
and we denote it by $2n$.
Moreover, there is an orthonormal basis $(X_1,\ldots,X_{2n})$
of $\mathfrak x$ such that $J_ZX_{2j-1}=X_{2j}$, for $1\le j\le n$.
Given any such basis, set
\begin{equation}
  D_Z := \frac{c}2 \sum\nolimits_{1\le j\le n}
  TZX_{2j-1}X_{2j} + TZ\otimes\pi_*(A_Z) .
  \label{dz}
\end{equation}
Observe that, for any orthonormal basis $(Y_1,\ldots,Y_{2n})$
of $\mathfrak x$,
\begin{equation}\label{dz2}
\begin{split}
  D_Z &= \frac18 \sum\nolimits_{j,k}
  \langle[Y_j,Y_k],Z\rangle TZY_{j}Y_{k} + TZ\otimes\pi_*(A_Z) \\
  &=  \frac{c}4 \sum\nolimits_{j,k}
  \langle J_ZY_j,Y_k\rangle TZY_{j}Y_{k} + TZ\otimes\pi_*(A_Z) \\
  &=  \frac{c}2 \sum\nolimits_{j<k}
  \langle J_ZY_j,Y_k\rangle TZY_{j}Y_{k} + TZ\otimes\pi_*(A_Z)
\end{split}
\end{equation}
In what follows, let $\{A,B\}:=AB+BA$.

\begin{lem}\label{anco}
For any $X\in\mathfrak x$, we have
\begin{equation*}
  \{ D_Z, TX\otimes\pi_*(A_X) + TJ_ZX\otimes\pi_*(A_{J_ZX}) \} = 0 .
\end{equation*}
\end{lem}

\begin{proof}
We have $cA_X=R(X,T)$, by \eqref{rxt}, and hence
\begin{equation}
  c\{TZXJ_ZX, TX\otimes\pi_*(A_X) \}
  = 2ZJ_ZX\otimes\pi_*(R(X,T)) .
  \label{anco1}
\end{equation}
By substituting $J_ZX$ for $X$ in \eqref{anco1}, we obtain
\begin{equation}
  c\{TZXJ_ZX, TJ_ZX\otimes\pi_*(A_{J_ZX}) \}
  = -2ZX\otimes\pi_*(R(J_ZX,T)) .
  \label{anco1j}
\end{equation}
We also have $[Z,X]=0$, hence $[A_Z,A_X] = R(Z,X)$.
Furthermore, $R(Z,X) = R(J_ZX,T)$, by \lref{amaz},
hence
\begin{equation}
  \{TZ\otimes\pi_*(A_Z), TX\otimes\pi_*(A_X) \}
  = ZX\otimes\pi_*(R(J_ZX,T)) .
  \label{anco2}
\end{equation}
By substituting $J_ZX$ for $X$ in \eqref{anco2},
we obtain
\begin{equation}
    \{TZ\otimes\pi_*(A_Z), TJ_ZX\otimes\pi_*(A_{J_ZX}) \}
  = -ZJ_ZX\otimes\pi_*(R(X,T)) .
  \label{anco2j}
\end{equation}
Moreover, we have
\begin{multline}
  \{TZYJ_ZY, TX\otimes\pi_*(A_{X}) \} \\
  = \{TZYJ_ZY, TJ_ZX\otimes\pi_*(A_{J_ZX}) \} = 0 ,
  \label{anco3}
\end{multline}
for all $Y\in\mathfrak x$ perpendicular to $X$ and $J_ZX$.
Now we may assume that $X$ is of norm $1$.
Then there is an orthonormal basis $(X_1,\ldots,X_{2n})$
of $\mathfrak x$ such that $J_ZX_{2j-1}=X_{2j}$, for $1\le j\le n$,
and such that $X=X_1$.
By \eqref{anco3}, the terms of $D_Z$ involving $TZX_{2j-1}X_{2j}$,
$j\ge2$, do not contribute to the anti-commutator
$\{D_Z,TX\otimes\pi_*(A_X) + TJ_ZX\otimes\pi_*(A_{J_ZX})\}$.
The four remaining terms cancel pairwise,
by \eqref{anco1}--\eqref{anco2j}.
\end{proof}

For an orthonormal basis $(X_1,\ldots,X_{2n})$ of $\mathfrak x$, set
\begin{equation}
  D_{\mathfrak x}
  := TX_1\otimes\pi_*(A_{X_1}) +\dots+ TX_{2n}\otimes\pi_*(A_{X_{2n}}) ,
  \label{dxx}
\end{equation}
and note that $D_{\mathfrak x}$ does not depend on the choice
of $(X_1,\ldots,X_{2n})$.

\begin{rem}\label{abel}
If $\mathfrak z=0$, then $\mathfrak n$ is Abelian and we are in the case
of real hyperbolic spaces or cusps, respectively,
and we get $D^{\rm le}_N=D_{\mathfrak x}$ on $H^{\rm le}_N$.
The contribution of cusps in the case $\dim N=1$
follows easily from the more general discussion in \cite{BB1}.
If $\dim N\ge2$, then the arguments in the proof of \tref{vaneta} apply
and show that the low energy $\eta$-invariant vanishes.
\end{rem}

\begin{lem}\label{anco5}
For any unit vectors $Z\in\mathfrak z$,
\begin{equation*}
  \{D_Z,D_{\mathfrak x}\} = 0 .
\end{equation*}
\end{lem}

\begin{proof}
Apply \lref{anco},
using an orthonormal basis $(X_1,\ldots,X_{2n})$ of $\mathfrak x$
with $J_ZX_{2j-1}=X_{2j}$, for $1\le j\le n$.
\end{proof}

Assume from now on that $\mathfrak z\ne0$, compare Remark \ref{abel}.
For an orthonormal basis $(Z_1,\ldots,Z_\ell)$ of $\mathfrak z$, set
\begin{equation}
  D_{\mathfrak z} := D_{Z_1} +\dots+ D_{Z_\ell} ,
  \label{dzz}
\end{equation}
and note that $D_{\mathfrak z}$ does not depend on the choice
of $(Z_1,\ldots,Z_\ell)$.

\begin{cor}\label{anco7}
On $H^{\rm le}_N$, we have
\begin{equation*}
  D^{\rm le}_N =  D_{\mathfrak z} + D_{\mathfrak x}
  \quad\text{and}\quad
  \{D_{\mathfrak z},D_{\mathfrak x}\} = 0 .
  \qed
\end{equation*}
\end{cor}

\begin{prop}\label{etadd}
On $H^{\rm le}_N$, we have
\begin{align}
  \ker(D^{\rm le}_N) &= \ker D_{\mathfrak z}\cap\ker D_{\mathfrak x} ,
  \tag{1} \label{etadd1}  \\
  \eta(D^{\rm le}_N) &= \eta(D_{\mathfrak z})
  = \eta (D_{\mathfrak z}|_{\ker D_{\mathfrak x}}) .
  \tag{2} \label{etadd2}
\end{align}
\end{prop}

\begin{proof}
By \cref{anco7}, \eqref{etadd1} is clear and
\begin{align*}
  \eta(D^{\rm le}_N)
  &= \eta (D_{\mathfrak z}|_{\ker D_{\mathfrak x}})
  + \eta (D_{\mathfrak x}|_{\ker D_{\mathfrak z}}) , \\
  \eta (D_{\mathfrak z}|_{\ker D_{\mathfrak x}}) &=
  \eta(D_{\mathfrak z}) , \\
  \eta (D_{\mathfrak x}|_{\ker D_{\mathfrak z}})
  &=\eta (D_{\mathfrak x}) .
\end{align*}
Now $D_{\mathfrak x}$ anticommutes with the involution $TZ_1$
of $\Sigma^+\otimes V$,
hence $\eta(D_{\mathfrak x})=0$, hence \eqref{etadd2}.
\end{proof}

\subsection{Contribution of Complex Hyperbolic Cusps}
\label{suschc}
As in Section \ref{susech}, we represent complex hyperbolic space $\C H^n$
by the symmetric pair $(G,K)$,
where $G=\SU(1,n)$ and $K=\Se(\U(1)\times\U(n))\cong\U(n)$,
and recall the decomposition $\mathfrak{su}(1,n)=\mathfrak k+\mathfrak p$,
where $\mathfrak k\cong\mathfrak u(n)$.
We decompose matrices $X\in\mathfrak{su}(1,n)$ correspondingly
and write
\begin{equation}
  X = X^{\mathfrak k}+X^{\mathfrak p}
  \quad\text{with $X^{\mathfrak k}\in\mathfrak k$
  and $X^{\mathfrak p}\in\mathfrak p$.}
  \label{decokp}
\end{equation}
We recall that,
after identifying $\mathfrak p$  with the tangent space of $\C H^n$
at the point fixed by $\U(n)$ as usual, we have
\begin{equation}
  R(X,Y)Z = - [[X,Y], Z] ,
  \label{rxyz}
\end{equation}
for all $X,Y,Z\in\mathfrak p$, where the Lie brackets on the right hand side
are taken in $\mathfrak{su}(1,n)$.

Let $\mathfrak n\subset\mathfrak s\subset\mathfrak{su}(1,n)$
be as in Section \ref{susech}.
Recall that we identify $\C H^n$ in \eqref{chom} also with
the solvable Lie subgroup $S$ of $SU(1,n)$ corresponding to $\mathfrak s$,
endowed with the left-invariant Riemannian metric
such that the isomorphism $\mathfrak s\to\mathfrak p$, $X\mapsto X^{\mathfrak p}$,
preserves scalar products.
Furthermore, the nilpotent Lie subgroup $N$ of $S$
corresponding to $\mathfrak n$ is isomorphic to the standard Heisenberg group.
Thus we are in the situation of Section \ref{susgrc},
where $N$ is the standard Heisenberg group and where $c=1$ in \eqref{hety4}.

Let $X\in\mathfrak n$.
By \eqref{chtz} and \eqref{chtx}, we have $[T,X]=-WX$
and hence
\begin{equation}
  [T,X^{\mathfrak p}] = - (WX)^{\mathfrak k} ,
  \label{txp}
\end{equation}
where we use the notation from \eqref{decokp}.
Using \eqref{rxt} and \eqref{rxyz}, we obtain therefore that
\begin{equation}
  A_{WX}Y = R(T,X^{\mathfrak p})Y^{\mathfrak p}
  = - [[T,X^{\mathfrak p}], Y^{\mathfrak p}]
  = [(WX)^{\mathfrak k}, Y^{\mathfrak p}] .
\end{equation}
Since $W$ is invertible,
we conclude that, for any $X\in\mathfrak n$,
\begin{equation}
  A_X = X^{\mathfrak k} .
  \label{axxk}
\end{equation}

With $\alpha$ as in $\eqref{chir}$,
we let $\hat\alpha_*:\mathfrak u(n)\to\mathfrak u(\Sigma)$
be the composition of the differential $\alpha_*$ of $\alpha$
with the differential of the spinor representation
of $\mathfrak{so}(\mathfrak p)\simeq\mathfrak{spin}(\mathfrak p)$
on $\Sigma:=\Sigma_{\mathfrak p}$.
Following Chapter \ref{sechc}, we choose $K=\U(n)$
and let $\pi_*$ be a unitary representation of $\mathfrak u(n)$
on a Hermitian vector space $V$.
We assume that there exists a unitary representation
$\beta$ of $K=\U(n)$ on $\Sigma\otimes V$ satisfying \eqref{hdb}
and get the associated Dirac bundle $E$ over $\C H^n$.
Clifford multiplication by the complex volume form $\omega_{\C}$
determines a super-symmetry $E=E^+\oplus E^-$,
and this super-symmetry is induced by the corresponding
decomposition $\Sigma=\Sigma^+\oplus\Sigma^-$.

To distinguish it from multiplication with $i$ in $\C^n\simeq\mathfrak p$,
we denote the complex structure in $\cl(\mathfrak p)$ by $\sqrt{-1}$.
With the corresponding changes in notation,
we follow Section \ref{decsig} and set
\begin{equation}
  \omega_j := \sqrt{-1}X_j^{\mathfrak p}Y_j^{\mathfrak p} ,
  \quad 1\le j\le n ,
  \label{invol}
\end{equation}
where $X_1=T,Y_1=Z,X_2,Y_2,\dots,X_n,Y_n$ are as in \eqref{heibas}.
By the discussion in Section \ref{decsig}, we have
\begin{equation}
  \Sigma^+ = \oplus_{\epsilon \in \{-1,1\}^{n-1}} \Sigma_\epsilon^+ ,
  \label{decosi2}
\end{equation}
where
\begin{equation}
  \Sigma_\epsilon^+ := \{ \sigma\in \Sigma^+ :
  \text{$\omega_j\sigma=\epsilon_j\sigma$ for $2\le j\le n$} \} .
  \label{decosi4}
\end{equation}
Since $X_j$ commutes with $\omega_k$ for $k\ne j$
and anti-commutes with $\omega_j$,
all the subspaces $\Sigma_\epsilon$ are isomorphic.
In particular, for all $\epsilon \in \{-1,1\}^{n-1}$,
\begin{equation}
  \dim \Sigma_\epsilon^+
  = \dim\Sigma^+/{\rm card} \{-1,1\}^{n-1} = 1 .
  \label{decosi5}
\end{equation}
For any $\epsilon\in\{-1,1\}^{n-1}$,
let $\nu(\epsilon)\in\{0,\dots,n-1\}$ be the number of $j$ with $\epsilon_j=-1$,
for $2\le j\le n$.
Then
\begin{equation}
  \Sigma^+ = \oplus_k \Sigma^+_k ,
  \quad\text{where}\quad
  \Sigma^+_k := \oplus_{\nu(\epsilon)=k} \Sigma_{\epsilon}^+ .
  \label{decosi6}
\end{equation}
By definition, $\omega_\C$ acts as identity on $\Sigma^+$,
hence $\omega_1=\omega_2\cdots\omega_n$ on $\Sigma^+$.
Therefore
\begin{equation}
  \Sigma^+_{\rm even} =\oplus_{\text{$k$ even}} \Sigma_k^+
  \quad\text{and}\quad
  \Sigma^+_{\rm odd} = \oplus_{\text{$k$ odd}} \Sigma_k^+
    \label{decosi7}
\end{equation}
are the eigenspaces of $\omega_1$ for the eigenvalues $1$ and $-1$,
respectively.
In passing we note that the left side of \eqref{decosi6} gives the decomposition
of $\Sigma^+$ into irreducible representations of the stabilizer of $T$ in $\U(n)$,
by work of Camporesi and Pedon, see \cite[Lemma 3.1]{CP}.

We recall that the complexification of $\mathfrak u(n)$ is $\mathfrak{gl}(n,\C)$,
where the complex structure of $\mathfrak{gl}(n,\C)$
is given by multiplication of matrix coefficients with $i$.
The space $\mathfrak h\subseteq\mathfrak{gl}(n,\C)$ of diagonal matrices
is a Cartan subalgebra of $\mathfrak{gl}(n,\C)$, and the roots
\begin{equation}
  \rho_j({\rm diag}(h_1,...,h_n)) := h_j
\end{equation}
constitute a basis of $\mathfrak h^*$.
The associated Weyl group $\mathcal W$ of automorphisms of $\mathfrak h$
leaves the set $\{\rho_1,\dots,\rho_n\}$ invariant
and acts on it as the (complete) group of permutations.
As usual, we choose
\begin{equation}
  \{ h = {\rm diag}(h_1,...,h_n) : h_j \in \R, h_1 > h_2> \dots >h_n \}
  \label{uweyl}
\end{equation}
as positive Weyl chamber.
The corresponding set of positive roots of $\mathfrak{gl}(n,\C)$ is given by
\begin{equation}
  \Delta^+ = \{\rho_j - \rho_k : 1\le j < k \le n \} .
  \label{uproot}
\end{equation}
Irreducible complex representations of $\mathfrak u(n)$ are classified
by their {\em highest weight} $\lambda= \sum\lambda_j \rho_j$,
where $\lambda$ is {\em dominant},
that is, $\lambda_1 \ge \lambda_2 \ge \dots \ge \lambda_n$,
and algebraically integral, that is, $\lambda_i-\lambda_j\in \Z$ for all $i,j$.
The dimension of the corresponding representation space $V_\lambda$ is
\begin{equation}
  \dim V_\lambda = \prod_{j<k}\frac{k-j-\lambda_k+\lambda_j}{k-j} ,
  \label{udimrep}
\end{equation}
by the Weyl dimension formula.
The irreducible representation with highest weight $\lambda$
is induced by a representation of $\U(n)$ if all the $\lambda_j$ are integral.
The representation $\alpha$ as above is the irreducible representation
of $\U(n)$ with highest weight $(2,1,\dots,1)$ (and complex dimension $n$).

For the discussion of $\hat\alpha_*$,
we identify $\mathfrak p=\R^{2n}$ and $\Sigma=\Sigma_{2n}$.
We let $(e_1,\dots,e_{2n})$ be the standard basis of $\R^{2n}\simeq\C^n$
with $e_{2j}=ie_{2j-1}$, $1\le j\le n$,
and denote the complex structure of $\Sigma_{2n}$ by $\sqrt{-1}$ as above.
For $h=(it_1,\dots,it_n)$ in $\mathfrak h\cap\mathfrak u(n)$, we get
\begin{align}
  \hat\alpha_*(h)
  &= \frac14 \sum_{1\le j\le 2n} e_j \big( it_je_j + (it_1+\dots+it_n)e_j \big)
  \notag \\
  &= - \frac12\sqrt{-1} \sum_{1\le j\le n} (t_1+\dots+2t_j+\dots+t_n)\omega_j .
  \label{uspin}
\end{align}
by the Parthasarathy formula \cite[Lemma 2.1]{Pa},
where $e_j$ and $\omega_j$ stand for Clifford multiplication
by $e_j$ and $\omega_j$, respectively.
Hence the subspaces $\Sigma_\varepsilon$ of $\Sigma_{2n}$
as in Section \ref{decsig} are weight spaces.
For $0\le l\le n$,
we let $V_l$ be the sum over all $\Sigma_\varepsilon$
such that $l$ is the number of $j$ with $\varepsilon_j=-1$,
that is, $\varepsilon_1+\dots+\varepsilon_n=n-2l$.
Then $V_l$ is the irreducible representation of $\mathfrak u(n)$
with heighest weight
\begin{equation}
  \lambda_1 = \dots = \lambda_l = l - \frac{n -1}2
  > \lambda_{l+1} = \dots = \lambda_n = l - \frac{n +1}2 ,
  \label{uspinl}
\end{equation}
and is of dimension $\binom{n}{l}$, in agreement with Weyl's dimension formula.

As an example, we discuss differential forms.
Since $\alpha$ is the irreducible representation
with maximal weight $(2,1,\dots,1)$,
the bundles of differential forms of type $(p,0)$ and $(0,q)$ are associated
to the irreducible representations $\beta$ of $\U(n)$ with maximal weights
\begin{align}
  \lambda_1 = \dots = \lambda_{n-p} = - p
  &> \lambda_{n-p+1} = \dots = \lambda_n = - (p + 1)
  \label{up0}
\intertext{and}
  \lambda_1 = \dots = \lambda_{q} = q+1
  &> \lambda_{q+1} = \dots = \lambda_n = q ,
  \label{u0q}
\end{align}
respectively.
We see that the sum of the bundles of differential forms of type $(0,q)$,
$0\le q\le n$, is given by $\Sigma\otimes V_n$,
where $V_n$ is as above.
That is, $\pi_*$ is the one-dimensional irreducible representation
of $\mathfrak u(n)$ with highest weight $\lambda_j=(n+1)/2$, $1\le j\le n$.

\begin{rem}\label{remlift}
 From \eqref{uspinl}, we see that $\hat\alpha_*$ comes from
a representation of $\U(n)$ if $n$ is odd,
and then the spinor bundle of $\C H^n$ descends
to quotients of $\C H^n$ by discrete subgroups of $\SU(1,n)$.
On the other hand,
if  $\widetilde\SU(1,n)$ denotes the non-trivial twofold cover of $\SU(1,n)$,
then $\hat\alpha_*$ comes from a representation of the corresponding
twofold cover $\tilde\U(n)$ of $\U(n)$, for all $n$.
Hence, if the discrete subgroup of $\SU(n)$ under consideration
admits a lift into $\tilde\SU(1,n)$,
then the spinor bundle also decends to the corresponding quotient of $\C H^n$.
A similar remark applies to $\beta_*$.
\end{rem}

We note that $D_{\mathfrak x}$ is an odd operator with respect to the grading
\begin{equation}
  \Sigma^+\otimes V_\pi
  = ( \Sigma^+_{\rm even}\otimes V_\pi )
  \oplus ( \Sigma^+_{\rm odd}\otimes V_\pi ) ,
  \label{decosi8}
\end{equation}
whereas $D_{\mathfrak z}=D_Z$ is an even operator.

\begin{thm}\label{leech}
With $H^k:=\ker D_{\mathfrak x}\cap (\Sigma^+_k\otimes V_\pi)$
and $b_k:=\dim H^k(\pi)$, for $0\le k\le n-1$, we have
\begin{align}
  &\ker D_{\mathfrak x} = \oplus H^k , \tag{1} \\
  &b_k =
  (n-1)! \dim V_\pi\prod_{\substack{1\le j \le n \\ j\ne k+1}}
  |\lambda_j-\lambda_{k+1}+k+1-j|^{-1} ,  \tag{2} \\
  &D_Z|_{H^k}
  = (-1)^k ( 2k - 2\lambda_{k+1} - n + 1 )/2 . \tag{3}
\end{align}
\end{thm}

\begin{proof}
Our proof relies on Kostant's theorem,
see \cite{Kos} or Theorem 4.139 in \cite{KV}.
We start by describing an explicit model of $\Sigma$,
compare Chapter 5 of \cite{Wu}.
For $1\le j\le n$, let
\begin{equation}
  F_j := \frac12(X_j^{\mathfrak p} - \sqrt{-1}Y_j^{\mathfrak p})
  \quad\text{and}\quad
  \bar F_j := \frac12(X_j^{\mathfrak p} + \sqrt{-1}Y_j^{\mathfrak p}) .
  \label{sigma}
\end{equation}
As elements of $\cl(\mathfrak p)$, they satisfy
\begin{equation}
  F_jF_j = \bar F_j\bar F_j = 0 , \quad
  \bar F_jF_j = - F_j\bar F_j - 1 ,
  \label{sigmaj}
\end{equation}
for $1\le j\le n$, and
\begin{equation}
  F_jF_k = - F_kF_j , \quad
  \bar F_j\bar F_k = - \bar F_k\bar F_j , \quad
  F_j\bar F_k = - \bar F_kF_j ,
  \label{sigmajk}
\end{equation}
for $1\le j\ne k \le n$.
We identify $\Sigma$ with the left ideal in the Clifford algebra
generated by $\bar F=\bar F_1\cdots\bar F_n$.
Then the monomials $F_I\bar F$ over all $0\le k\le n$
and multi-indices $I=\{i_1,\cdots,i_k\}$ with $i_1<\dots<i_k$
constitute a basis of $\Sigma$.
The relations \eqref{sigmaj} and \eqref{sigmajk} determine an isomorphism
$\Sigma\simeq\Lambda(\C^n)$,
where $\C^n$ is spanned by $F_1,\dots,F_n$.
We have
\begin{equation}
  \omega_j\cdot F_I\bar F
  = \begin{cases}
  \phantom{-} F_{I}\bar F &\text{if $j\notin I$} , \\
  - F_{I}\bar F &\text{if $j\in I$} .
  \end{cases}
  \label{involsi}
\end{equation}
so that, under the identification $\Sigma\simeq\Lambda(\C^n)$,
\begin{equation}
  \Sigma^+_k
  \simeq \begin{cases}
  \Lambda^{k}(\C^{n-1}) &\text{if $k$ is even} , \\
  F_1\wedge\Lambda^{k}(\C^{n-1})
  \simeq \Lambda^{k}(\C^{n-1}) &\text{if $k$ is odd} ,
  \end{cases}
  \label{siglam}
\end{equation}
where $\C^{n-1}$ is spanned by $F_2,\dots,F_n$.

Recall that, by complexification,
$\pi_*$ induces a representation of $\mathfrak{gl}(n,\C)$.
Following the exposition in \cite[\S IV.8]{LM}, we set
\begin{align}
  \mathcal D_{\mathfrak x}
  &:= \frac12 \sum_{2\le j\le n}
  T(X_j^{\mathfrak p} - \sqrt{-1}Y_j^{\mathfrak p})
  \otimes \pi_*(X_j^{\mathfrak k} + iY_j^{\mathfrak k}) \\
  &\phantom{:}= 2 \sum_{2\le j\le n} TF_j \otimes \pi_*(E_{1j}) ,
  \notag \\
  \bar{\mathcal D}_{\mathfrak x}
  &:= \frac12 \sum_{2\le j\le n}
  T(X_j^{\mathfrak p} + \sqrt{-1}Y_j^{\mathfrak p})
  \otimes \pi_*(X_j^{\mathfrak k} - iY_j^{\mathfrak k}) \\
  &\phantom{:}= - 2 \sum_{2\le j\le n} T\bar F_j \otimes \pi_*(E_{j1}) ,
  \notag
\end{align}
where we note that factors $\sqrt{-1}$ on the left and $i$ on the right of $\otimes$
multiply to $-1$ in the tensor product.
Using \eqref{axxk}, we have
\begin{equation}
  D_{\mathfrak x} = \mathcal D_{\mathfrak x} + \bar{\mathcal D}_{\mathfrak x} ,
  \quad
  \bar{\mathcal D}_{\mathfrak x} = \mathcal D_{\mathfrak x}^* ,
  \quad\text{and}\quad
  \mathcal D_{\mathfrak x}\mathcal D_{\mathfrak x}
  = \bar{\mathcal D}_{\mathfrak x}\bar{\mathcal D}_{\mathfrak x} = 0 .
\end{equation}
Moreover,
\begin{equation}
  {\mathcal D}_{\mathfrak x}(\Sigma^+_k\otimes V_\pi)
  \subset \Sigma^+_{k+1}\otimes V_\pi
  \quad\text{and}\quad
  \bar{\mathcal D}_{\mathfrak x}(\Sigma^+_k\otimes V_\pi)
  \subset \Sigma^+_{k-1}\otimes V_\pi .
\end{equation}
Hence $\ker D_{\mathfrak x}$ is equal to the space of
$\mathcal D_{\mathfrak x}$-harmonic cocycles of the cochain complex
\begin{equation}
  \cdots \xrightarrow{{\mathcal D}_{\mathfrak x}} \Sigma^+_{k-1} \otimes V_\pi
  \xrightarrow{{\mathcal D}_{\mathfrak x}}\Sigma^+_k\otimes V_\pi
  \xrightarrow{{\mathcal D}_{\mathfrak x}}\Sigma^+_{k+1}\otimes V_\pi
  \xrightarrow{{\mathcal D}_{\mathfrak x}} \cdots
\end{equation}
This shows the first assertion of the theorem and
that $\ker D_{\mathfrak x}$ is isomorphic to the cohomology of the complex.
Moreover, under the above identification $\Sigma^+=\Lambda(\C^{n-1})$,
we have
\begin{equation}
  \mathcal D_{\mathfrak x}(\omega\otimes v)
  = 2  \sum_{2\le j\le n} (F_j\wedge\omega) \otimes \pi_*(E_{1j})v .
  \label{coho2}
\end{equation}
Following the notation in \cite{KV},
we consider the subalgebras $\mathfrak u$ and $\mathfrak l$
of $\mathfrak{gl}(n,\C)$,
where $\mathfrak u$ is spanned by the  $E_{1j}$, $2\le j\le n$,
and
\begin{equation}
  \mathfrak l :=
  \left\{ \begin{pmatrix} x & 0 \\ 0 & B \end{pmatrix} :
  \text{$x\in\C$ and $B\in \mathfrak{gl}(n-1,\C)$} \right\} .
\end{equation}
Then $\mathfrak q=\mathfrak l\oplus\mathfrak u$ is a parabolic
subalgebra of $\mathfrak{gl}(n,\C)$.
By \eqref{coho2}, the kernel of the restriction of $D_{\mathfrak x}$
is isomorphic to $H^k(\mathfrak u,\pi)$,
the Lie algebra cohomology of $\mathfrak u$ with respect to $\pi$.
Now Kostant's theorem determines the latter as an $\mathfrak l$-module,
where $l\in\mathfrak l$ acts on $\Lambda^k(\mathfrak u)\otimes V_\pi$ by
\begin{equation}
  -\ad(l)^*\otimes \id + \id\otimes \pi_*(l) ,
\end{equation}
see (4.138b) in \cite{KV}.
To apply Kostant's theorem, we introduce
\begin{align}
  \Delta^+(\mathfrak u) &= \{\rho_1- \rho_j : 2\le j\le n \} , \\
  \Delta^+(\mathfrak l) &= \{\rho_i- \rho_j : 2 \le i < j \le n \} .
\end{align}
For $w\in \mathcal W$, we also introduce
\begin{equation}
  \Delta^+(w) := \{ \lambda \in \Delta^+ : w^{-1}\lambda < 0 \} ,
  \quad
  \ell(w) := |\Delta^+(w)| ,
\end{equation}
and
\begin{equation}
  \mathcal W^1
  := \{ w \in \mathcal W : \Delta^+(w) \subseteq \Delta^+(\mathfrak u) \} .
\end{equation}
Then $\mathcal W^1=\{w_0,\dots,w_{n-1} \}$, where $w_0=\id$ and
\begin{equation}
  w_k^{-1} =
  \begin{pmatrix} 1 & 2 & \cdots & k+1 \\ k+1 & 1 & \cdots & k \end{pmatrix} ,
\end{equation}
for $1\le k\le n-1$.
We note that $\ell(w_k)=k$, for $0\le k\le n-1$.

Kostant's theorem implies that, as an $\mathfrak l$-module,
$H^k(\mathfrak u,\pi)$ is the irreducible representation of $\mathfrak l$
with highest weight
\begin{multline}
  w_k (\lambda + \delta) - \delta = \\
  (\lambda_{k+1} - k)\rho_1
  + \sum_{2\le j\le k+1} (\lambda_{j-1} + 1)\rho_j
  + \sum_{j>k+1} \lambda_j\rho_j ,
\end{multline}
where $\delta$ is the half sum of the positive roots of $\mathfrak{gl}(n,\C)$,
\begin{equation}
  \delta := \frac12 \sum_{j=1}^n (n+1-2j)\rho_j .
\end{equation}
Moreover, the action of the $\mathfrak k$-component
$Z^{\mathfrak k}\simeq -iE_{1,1}$ of $Z$ on $H^k(\mathfrak u,\pi)$
is given by multiplication with
\begin{equation}
  ik - i\lambda_{k+1} = ik + \id\otimes \pi(Z^{\mathfrak k}) .
\end{equation}
In particular, $\id\otimes \pi(Z^{\mathfrak k})=-i\lambda_{k+1}$.
It follows that, on $H^k(\mathfrak u,\pi)$,
\begin{equation}
  D_{\mathfrak z} = D_Z
  = (-1)^k \frac12 (2k - 2\lambda_{k+1} - n + 1) ,
\end{equation}
which is the third assertion of the theorem.
We have
\begin{equation}
  b_k = \dim H^k(\mathfrak u,\pi)
  = \prod_{\alpha\in \Delta^+(\mathfrak l)} \frac
  {(\alpha, w_k(\lambda+\delta)-\delta+\delta_\mathfrak l)}
  {(\alpha, \delta_\mathfrak l)} ,
\end{equation}
by Weyl's dimension formula,
where $\delta_\mathfrak l$ is the half sum of the positive roots of $\mathfrak l$,
\begin{equation}
  \delta_\mathfrak l = \frac12 \sum_{j=2}^n (n+2-2j) \rho_j .
\end{equation}
The second assertion of the theorem is an easy consequence.
\end{proof}

\subsection{R\'esum\'e and Examples}
\label{exaeta}
For the convenience of the reader,
we give a r\'esum\'e of the correction terms ${\rm Corr}(\mathcal C)$
to the extended index of $D^+$
for complex hyperbolic cusps $\mathcal C$ of complex dimension $n$.
We assume that the cross section of $\mathcal C$
is given as $\Gamma\backslash G_{n-1}$
as in \cref{etacor},
where $N:=G_{n-1}$ is the Heisenberg group of dimension $2n-1$
and where $\Gamma$ is a uniform lattice in $G_{n-1}$.
We denote by $d$ the type of $\Gamma$
and set $|\Gamma|=d_1\cdots d_{n-1}$ as in Section \ref{sechei}.

Assume that, over $\mathcal C$,
the Dirac bundle $E$ is up to twist of the form
$E=\mathcal P\times_\beta(\Sigma_{\mathfrak s}\otimes V)$
with $\mathfrak s\subset\mathfrak{su}(1,n)$ as in Section \ref{susech}
and $\beta$ as in \eqref{hdb}, that is, $\beta_*=\hat\alpha_*\otimes\id+\id\otimes\pi_*$,
where $\pi_*$ is a unitary representation of $\mathfrak k=\mathfrak u(n)$ on $V$.
Then sections of $E^+$ over $\mathcal C$
correspond to maps $S\to\Sigma_{\mathfrak n}\otimes V$
satisfying the twist condition \eqref{twist}.
Clearly, ${\rm Corr}(\mathcal C)$ is the sum of the parts
which are determined by the subbundles of $E^+$ over $\mathcal C$ of the form
$\mathcal P\times_\beta(\Sigma_{\mathfrak n}\otimes W)$,
where $W$ is irreducible as a joint $\mathfrak u(n)$ and $\Gamma$ module.
Therefore, we may restrict ourselves to the case where $V$ is irreducible
as a joint $\mathfrak u(n)$ and $\Gamma$ module.

Recall that we use an index $\mathcal C$ to denote objects connected
to the cusp $\mathcal C$ and that
\begin{equation}\label{corrc}
  {\rm Corr}(\mathcal C) = \frac12\big(
  \lim_{t\to\infty}\eta(A_{\mathcal C,t}^{{\rm he},+})
  + \eta(A_{\mathcal C,0}^{{\rm le},+})
  + \dim\ker A_{\mathcal C,0}^{{\rm le},+} \big) ,
\end{equation}
see \tref{exinfin}.

\begin{thm}\label{corrcom}
Suppose that $V$ is irreducible as a joint $\mathfrak u(n)$ and $\Gamma$ module.
Then there are two cases:

\noindent (1) If $\tau$ is non-trivial,
\begin{align*}
  \lim_{t\to\infty}\eta(A_{\mathcal C,t}^{{\rm he},+})
  &= (-1)^n | \Gamma | \dim V \big(\zeta_c(1-n) + (-1)^n \zeta_{1-c}(1-n)\big) , \\
  \eta(A_{\mathcal C,0}^{{\rm le},+}) &= \dim\ker(A_{\mathcal C,0}^{{\rm le},+}) = 0 ,
\end{align*}
where $c$ denotes the twist parameter of $\tau$ as in \eqref{twistpara}.

\noindent (2) If $\tau$ is trivial,
then $\pi_*$ is irreducible and
\begin{align*}
  \lim_{t\to\infty}\eta(A_{\mathcal C,t}^{{\rm he},+})
  &= (-1)^n | \Gamma | \dim V \big(\zeta(1-n) + (-1)^n \zeta(1-n)\big) , \\
  \eta(A_{\mathcal C,0}^{{\rm le},+})
  &= \sum (-1)^k b_k \sign(n - 1 - 2k + 2\lambda_{k+1}) , \\
  \dim\ker A_{\mathcal C,0}^{{\rm le},+}
  &= \sum b_k \cdot \delta_{2\lambda_{k+1},2k+1-n} ,
\end{align*}
where $\lambda_1\ge\ldots\ge\lambda_n$ denotes the highest weight of $\pi_*$ and
\begin{equation*}
  b_k =
  (n-1)! \dim V\prod_{\substack{1\le j \le n \\ j\ne k+1}}
  |\lambda_j-\lambda_{k+1}+k+1-j|^{-1} .
\end{equation*}
\end{thm}

\begin{rems}\label{corrrem}
1) Recall that $D^+$ is of Fredholm type
if and only if $\dim\ker A_{\mathcal C,0}^{{\rm le},+}=0$,
for all cusps $\mathcal C$ of $M$.

2) The dimension of $V$ in Assertion 2 is determined in \eqref{udimrep}.
\end{rems}

\begin{proof}[Proof of \tref{corrcom}]
The formulas for $ \lim_{t\to\infty}\eta(A_{\mathcal C,t}^{{\rm he},+})$
are taken from \cref{etacor}.
The formulas for $\eta(A_{\mathcal C,0}^{{\rm le},+})$
and $ \dim\ker A_{\mathcal C,0}^{{\rm le},+}$ follow from
\pref{etadd} and \tref{leech},
where we observe that $D^{\rm le}_N$ in \pref{etadd}
corresponds to $-A^{\rm le,+}_{\mathcal C,0}$.
\end{proof}

\begin{exas}\label{exaeta2}
We discuss the low energy invariants in some cases,
where the twist $\tau$ is trivial.
For convenience, we assume that all ends of $M$ are complex hyperbolic cusps.
Furthermore, in our examples,
the Dirac bundle is of the same kind along all the cusps of $M$,
that is, $\pi_*$ is the same for all the cusps.
We discuss the more precise quantities from \tref{leech}.
The formulas for the low energy invariants as in \tref{corrcom} will follow easily.

1) {\sc Dolbeault operator} on forms of bi-degree $(0,q)$:
In this case, $\pi$ is the irreducible representation
with highest weight $\lambda_j=(n+1)/2$, $1\le j\le n$.
We compute
\begin{equation}
  b_k = \dim H^k(\pi) = {\textstyle\binom{n-1}k}
\end{equation}
and
\begin{equation}
  D_Z|_{H^k(\pi)} = (-1)^{k}(k-n) ,
\end{equation}
by \tref{leech}.3.
Hence $\eta(A_{\mathcal C,0}^{\rm le,+})=0$
and  $\dim\ker A_{\mathcal C,0}^{{\rm le},+}=0$,
for all cusps $\mathcal C$ of $M$.
In particular, $D^+$ is of Fredholm type.

2) {\sc Signature operator}:
In this case, $\pi$ is the spin representation $\Sigma=\Sigma_{\mathfrak p}$,
which is the sum of the irreducible representations $V_l$
with highest weight as in \eqref{uspinl}, where $0\le l\le n$.
As for the dimension $b_{k,l}$ of $H^k(\mathfrak u, V_l)$, there are two cases:
\begin{equation}
  b_{k,l} =
  \begin{cases}
  \binom{n}{l}\binom{n}{k}\frac{l-k}{n}
  &\text{if $k<l$} , \\
  \binom{n}{l}\binom{n}{k+1}\frac{k+1-l}{n}
  &\text{if $k\ge l$} .
  \end{cases}
\end{equation}
Furthermore, we have
\begin{equation}
  D_Z|_{H^k(V_l)} =
  \begin{cases}
  (-1)^{k} (k - l)
  &\text{if $k<l$} , \\
  (-1)^{k} (k + 1 - l)
  &\text{if $k\ge l$} ,
  \end{cases}
\end{equation}
by \tref{leech}.3.
Hence $\dim\ker A_{\mathcal C,0}^{{\rm le},+}=0$, for all cusps $\mathcal C$ of $M$,
and therefore $D^+$ is a Fredholm operator.
We also get
\begin{equation}
  \eta(A_{\mathcal C,0}^{\rm le,+})
  = \sum_{k < l}(-1)^k{\textstyle\binom{n}{l}\binom{n}{k}\frac{l-k}{n}}
  + \sum_{k \ge l}(-1)^{k+1}{\textstyle\binom{n}{l}\binom{n}{k+1}\frac{k+1-l}{n}} .
\end{equation}
If we change $l$ in $n-l$ and $k$ in $n-1-k$ in the second sum,
we obtain
\begin{equation}
  \eta(A_{\mathcal C,0}^{\rm le,+}) = 0 \quad\text{if $n$ is odd} .
\end{equation}
For even $n$, we get
\begin{equation}
  \eta(A_{\mathcal C,0}^{\rm le,+})
  = 2\sum_{k < l}(-1)^k{\textstyle\binom{n}{l}\binom{n}{k}\frac{l-k}{n}} .
\end{equation}
For $1\le l\le n$, we have
\begin{equation*}
  \sum_{0\le k < l}(-1)^k{\textstyle\binom{n}{k}}
  = (-1)^{l-1} {\textstyle\binom{n-1}{l-1}}
\end{equation*}
and
\begin{equation*}
  \sum_{0\le k < l} (-1)^k {\textstyle\binom{n}{k} \frac{k}{n}}
  = \sum_{0\le k < l-1} (-1)^{k+1} {\textstyle\binom{n-1}{k}}
  = (-1)^{l-1} {\textstyle\binom{n-2}{l-2}} .
\end{equation*}
Hence
\begin{equation*}
\begin{split}
  \eta(A_{\mathcal C,0}^{\rm le,+})
  &=  2 \sum_{1\le l\le n} (-1)^{l-1} {\textstyle\binom{n}{l}  \left\{
  \binom{n-1}{l-1} \frac{l}{n} - \binom{n-2}{l-2} \right\}} \\
  &=2 \sum_{1\le l\le n} (-1)^{l-1} {\textstyle\binom{n-1}{l-1}^2
  + 2 \sum_{1\le l\le n} (-1)^{l} \binom{n}{l}\binom{n-2}{l-2}}
\end{split}
\end{equation*}
The first sum is zero since $n$ is even.
The second sum is the coefficient of $x^n$
in $(1-x)^n(1+x)^{n-2}=(1-x)^2(1-x^2)^{n-2}$
and hence
\begin{equation}
  \eta(A_{\mathcal C,0}^{\rm le,+})
 =2 (-1)^{n/2} {\textstyle\left(\binom{n-2}{n/2}-\binom{n-2}{n/2-1}\right)} .
\end{equation}

3) {\sc Dirac operator} on spinors:
In this case, $\pi_*$ is the irreducible representation with highest weight
$\lambda=0$ (where the spin structure along the cusps is periodic).
If $n$ is even, then $\ker A_{\mathcal C,0}^{\rm le,+}=0$,
for all cusps $\mathcal C$.
In particular, $D$ is a Fredholm operator.
Furthermore, each cusp $\mathcal C$ contributes a low energy $\eta$-invariant,
\begin{equation}
\begin{split}
  \eta(A_{\mathcal C,0}^{\rm le,+})
  &= \sum_{0\le k\le n-1} (-1)^k {\textstyle\binom{n-1}{k}} \sign(n-1 - 2k) \\
  &=2\sum_{0\le2k\le n-2} (-1)^k {\textstyle\binom{n-1}{k}} \\
  &=2\sum_{0\le2k\le n-2} (-1)^k
  \left( {\textstyle\binom{n-2}{k}} + {\textstyle\binom{n-2}{k-1}} \right) \\
  &=2(-1)^{\frac{n-2}{2}} {\textstyle\binom{n-2}{\frac{n-2}{2}}} .
\end{split}
\end{equation}
If $n$ is odd, we have $\eta(A_{\mathcal C,0}^{\rm le,+})=0$, for each cusp $\mathcal C$.
Furthermore, each cusp $\mathcal C$ contributes to the kernel,
\begin{equation}\label{spin17}
  \dim \ker A_{\mathcal C,0}^{\rm le,+}
  = {\textstyle\binom{n-1}{\frac{n-1}{2}}} .
\end{equation}
In particular, $D^+$ is not a Fredholm operator for odd $n$.
\end{exas}

\begin{rem}\label{remspin}
For $n$ odd, the bilinear form $\beta$ considered in Exercise 20.38 of \cite{FH}
defines a real $\Spin(2n)$-equivariant isomorphism
between $\Sigma_{2n}^+$ and $\Sigma_{2n}^-$.
By equivariance, this isomorphism induces a parallel isomorphism
between the positive and negative spinor bundles $\Sigma^\pm$
of a spin manifold $M$ of dimension $2n$
and therefore also an isomorphism between the spaces of harmonic sections
of bundles of the form $E^\pm=\Sigma^\pm\otimes_{\pi_*}V$.
We conclude that the $L^2$-index of the corresponding Dirac operator $D^+$
vanishes and that the extended indices of $D^+$ and $D^-$ coincide.
In particular, we get
\begin{equation}
  \ind D^+_{\ext} = h^+_\infty = h^-_\infty ,
\end{equation}
compare \tref{intmaxinf}.
By a similar reasoning,
we also obtain that the index densities $\omega_{D^\pm}$ vanish.
In the case of the Dirac operator as in Example \ref{exaeta2}.3 above,
the term in \eqref{spin17} times the number $\nu$ of ends of $M$
gives the extended index of $D$ 
which is equal to $h^+_\infty + h^-_\infty$, by \eqref{maxinf5}.
Hence $ \ind D^+_{\ext}$ is in this case equal to
$\nu/2$ times $\binom{n-1}{\frac{n-1}{2}}$.
\end{rem}



\end{document}